\documentclass[a4paper]{amsart}
\usepackage{a4wide,amsfonts,amssymb,amsmath,amscd,color,tikz,tikz-3dplot,tikz-cd,todonotes}
\usepackage[english]{babel}
\allowdisplaybreaks

\newtheorem{lem}{Lemma}[section]
\newtheorem{defi}[lem]{Definition}
\newtheorem{theo}[lem]{Theorem}
\newtheorem{cor}[lem]{Corollary}
\newtheorem{rem}[lem]{Remark}

\newcommand{\Abb}[5]{\begin{array}{ccccc}#1&:&#2&\longrightarrow&#3\\{}&{}&#4&\longmapsto&#5\end{array}}
\def\bs{\boldsymbol}
\def\mr{\mathring}
\def\ol{\overline}
\def\To{\longrightarrow}

\def\incl{\hookrightarrow}
\def\reals{\mathbb{R}}
\def\nat{\mathbb{N}}
\def\z{\mathbb{Z}}
\def\ga{\Gamma}
\def\gat{\ga_{\!t}}
\def\gan{\ga_{\!n}}
\def\om{\Omega}
\def\rdimom{\reals^{\dimom}}

\def\dimom{d}
\def\mcB{\mathcal{B}}
\def\mcH{\mathcal{H}}
\def\sfL{\mathsf{L}}
\def\sfH{\mathsf{H}}
\def\mbH{\bs{\mathsf{H}}}
\def\sfC{\mathsf{C}}
\newcommand{\B}[2]{\mcB^{#1}_{#2}}
\newcommand{\Harm}[2]{\mcH^{#1}_{#2}}
\renewcommand{\L}[2]{\sfL^{#1}_{#2}}
\renewcommand{\H}[2]{\sfH^{#1}_{#2}}
\newcommand{\bH}[2]{\mbH^{#1}_{#2}}
\newcommand{\C}[2]{\sfC^{#1}_{#2}}
\newcommand{\eps}{\varepsilon}
\DeclareMathOperator{\Lin}{Lin}
\DeclareMathOperator{\A}{A}
\DeclareMathOperator{\p}{\partial}
\DeclareMathOperator{\id}{id}

\DeclareMathOperator{\ed}{d}
\DeclareMathOperator{\cd}{\delta}
\DeclareMathOperator{\grad}{grad}
\DeclareMathOperator{\rot}{rot}
\DeclareMathOperator{\divergence}{div}
\def\div{\divergence}

\DeclareMathOperator{\Rot}{Rot}
\DeclareMathOperator{\Div}{Div}
\DeclareMathOperator{\symGrad}{symGrad}

\DeclareMathOperator{\devGrad}{devGrad}
\DeclareMathOperator{\RotS}{\Rot_{\S}}

\DeclareMathOperator{\DivS}{\Div_{\S}}

\DeclareMathOperator{\DivT}{\Div_{\T}}
\DeclareMathOperator{\symRot}{symRot}
\DeclareMathOperator{\symRotT}{\symRot_{\T}}
\DeclareMathOperator{\Gradgrad}{Gradgrad}
\DeclareMathOperator{\RotRot}{RotRot}

\DeclareMathOperator{\RotRottS}{\RotRot_{\S}^{\!\top}}

\DeclareMathOperator{\divDiv}{divDiv}
\DeclareMathOperator{\divDivS}{\divDiv_{\S}}
\renewcommand{\S}{\mathbb{S}}
\newcommand{\T}{\mathbb{T}}
\newcommand{\E}{\mathcal{E}}
\newcommand{\R}{\mathcal{R}}
\newcommand{\PotP}{\mathcal{P}}
\newcommand{\PotQ}{\mathcal{Q}}
\newcommand{\PotN}{\mathcal{N}}
\newcommand{\norm}[1]{|#1|}
\newcommand{\bnorm}[1]{\big|#1\big|}
\newcommand{\scp}[2]{\langle#1,#2\rangle}

\title[Hilbert Complexes with Mixed Boundary Conditions -- Part 1: De Rham Complex]
{Hilbert Complexes with Mixed Boundary Conditions\\
Part 1: De Rham Complex}
\author{Dirk Pauly}
\author{Michael Schomburg}
\address{Fakult\"at f\"ur Mathematik, Universit\"at Duisburg-Essen, Germany}
\email[Dirk Pauly]{dirk.pauly@uni-due.de}
\email[Michael Schomburg]{michael.schomburg@uni-due.de}
\date{\today; {\it Corresponding Author}: Dirk Pauly}
\thanks{}

\begin{document}

\def\titlerepude{\sf Hilbert Complexes with Mixed Boundary Conditions:\\
Regular Decompositions, Compact Embeddings,\\ 
and Functional Analysis ToolBox\\
Part 1: De Rham Complex}
\def\authorrepude{Dirk Pauly \& Michael Schomburg}
\def\daterepdue{\today}
\def\reportudemathyesno{no}
\def\reportudemathnumber{SM-UDE-826}
\def\reportudemathyear{2021}
\def\reportudematheingang{\daterepdue}
\newcommand{\preprintudemath}[5]{
\thispagestyle{empty}
\begin{center}\normalsize SCHRIFTENREIHE DER FAKULT\"AT F\"UR MATHEMATIK\end{center}
\vspace*{5mm}
\begin{center}#1\end{center}
\vspace*{5mm}
\begin{center}by\end{center}
\vspace*{0mm}
\begin{center}#2\end{center}
\vspace*{5mm}
\normalsize 
\begin{center}#3\hspace{69mm}#4\end{center}
\newpage
\thispagestyle{empty}
\vspace*{210mm}
Received: #5
\newpage
\addtocounter{page}{-2}
\normalsize
}
\ifthenelse{\equal{\reportudemathyesno}{yes}}
{\preprintudemath{\titlerepude}{\authorrepude}{\reportudemathnumber}{\reportudemathyear}{\reportudematheingang}}
{}


\begin{abstract}
We show that the de Rham Hilbert complex
with mixed boundary conditions on bounded strong Lipschitz domains
is closed and compact. The crucial results are compact embeddings 
which follow by abstract arguments using functional analysis
together with particular regular decompositions.
Higher Sobolev order results are proved as well.
\end{abstract}


\maketitle
\setcounter{tocdepth}{3}
{\small
\tableofcontents}


\section{Introduction}

In this paper we prove regular decompositions and resulting compact embeddings for 
the \emph{de Rham complex} (of vector fields)
\begin{equation*}
\def\arrowlength{5ex}
\def\arrowdistance{0}
\begin{tikzcd}[column sep=\arrowlength]
\cdots 
\arrow[r, rightarrow, shift left=\arrowdistance, "\cdots"] 
& 
\L{2}{}(\om) 
\ar[r, rightarrow, shift left=\arrowdistance, "\grad"] 
& 
[1em]
\L{2}{}(\om)
\arrow[r, rightarrow, shift left=\arrowdistance, "\rot"] 
& 
\L{2}{}(\om)
\arrow[r, rightarrow, shift left=\arrowdistance, "\div"] 
& 
\L{2}{}(\om)
\arrow[r, rightarrow, shift left=\arrowdistance, "\cdots"] 
&
\cdots,
\end{tikzcd}
\end{equation*}
and, more generally, for the \emph{de Rham complex} (of differential forms) 
\begin{equation*}
\def\arrowlength{5ex}
\def\arrowdistance{0}
\begin{tikzcd}[column sep=\arrowlength]
\cdots 
\arrow[r, rightarrow, shift left=\arrowdistance, "\cdots"] 
& 
\L{q-1,2}{}(\om) 
\ar[r, rightarrow, shift left=\arrowdistance, "\ed^{q-1}"] 
& 
[1em]
\L{q,2}{}(\om) 
\arrow[r, rightarrow, shift left=\arrowdistance, "\ed^{q}"] 
& 
\L{q+1,2}{}(\om) 
\arrow[r, rightarrow, shift left=\arrowdistance, "\cdots"] 
&
\cdots.
\end{tikzcd}
\end{equation*}
In forthcoming papers, we shall extend our results
to other more complicated complexes as well, such as 
the elasticity complex
\begin{equation*}
\def\arrowlength{5ex}
\def\arrowdistance{0}
\begin{tikzcd}[column sep=\arrowlength]
\cdots 
\arrow[r, rightarrow, shift left=\arrowdistance, "\cdots"] 
& 
\L{2}{}(\om) 
\ar[r, rightarrow, shift left=\arrowdistance, "\symGrad"] 
& 
[2.5em]
\L{2}{\S}(\om)
\arrow[r, rightarrow, shift left=\arrowdistance, "\RotRottS"] 
& 
[2.5em]
\L{2}{\S}(\om)
\arrow[r, rightarrow, shift left=\arrowdistance, "\DivS"] 
& 
[1em]
\L{2}{}(\om)
\arrow[r, rightarrow, shift left=\arrowdistance, "\cdots"] 
&
\cdots,
\end{tikzcd}
\end{equation*}
or the primal and dual biharmonic complexes
\begin{equation*}
\def\arrowlength{5ex}
\def\arrowdistance{0}
\begin{tikzcd}[column sep=\arrowlength]
\cdots 
\arrow[r, rightarrow, shift left=\arrowdistance, "\cdots"] 
& 
\L{2}{}(\om) 
\ar[r, rightarrow, shift left=\arrowdistance, "\Gradgrad"] 
& 
[2.5em]
\L{2}{\S}(\om)
\arrow[r, rightarrow, shift left=\arrowdistance, "\RotS"] 
& 
[1em]
\L{2}{\T}(\om)
\arrow[r, rightarrow, shift left=\arrowdistance, "\DivT"] 
& 
[1em]
\L{2}{}(\om)
\arrow[r, rightarrow, shift left=\arrowdistance, "\cdots"] 
&
\cdots,
\end{tikzcd}
\end{equation*}
\begin{equation*}
\def\arrowlength{5ex}
\def\arrowdistance{0}
\begin{tikzcd}[column sep=\arrowlength]
\cdots 
\arrow[r, rightarrow, shift left=\arrowdistance, "\cdots"] 
& 
\L{2}{}(\om) 
\ar[r, rightarrow, shift left=\arrowdistance, "\devGrad"] 
& 
[2.5em]
\L{2}{\T}(\om)
\arrow[r, rightarrow, shift left=\arrowdistance, "\symRotT"] 
& 
[2.5em]
\L{2}{\S}(\om)
\arrow[r, rightarrow, shift left=\arrowdistance, "\divDivS"] 
& 
[2em]
\L{2}{}(\om)
\arrow[r, rightarrow, shift left=\arrowdistance, "\cdots"] 
&
\cdots,
\end{tikzcd}
\end{equation*}
which is possible
due to the general structure and our unified approach and methods.
All complexes are considered with mixed boundary conditions
on a bounded strong Lipschitz domain $\om\subset\rdimom$.
Some of our results hold also for higher Sobolev orders.
Note that the first three complexes are formally symmetric
and that the last two complexes are formally adjoint or dual to each other.

These \emph{Hilbert complexes} share the same geometric sequence (complex) structure 
\begin{equation*}
\def\arrowlength{5ex}
\def\arrowdistance{0}
\begin{tikzcd}[column sep=\arrowlength]
\cdots 
\arrow[r, rightarrow, shift left=\arrowdistance, "\cdots"] 
& 
\H{}{0} 
\ar[r, rightarrow, shift left=\arrowdistance, "\A_{0}"] 
& 
\H{}{1}
\arrow[r, rightarrow, shift left=\arrowdistance, "\A_{1}"] 
& 
\H{}{2}
\arrow[r, rightarrow, shift left=\arrowdistance, "\cdots"] 
&
\cdots,
\end{tikzcd}
\qquad
R(\A_{0})\subset N(\A_{1}),
\end{equation*}
where $\A_{0}$ and $\A_{1}$ are densely defined and closed (unbounded) linear operators
between Hilbert spaces $\H{}{\ell}$.
The corresponding \emph{domain Hilbert complex} is denoted by
\begin{equation*}
\def\arrowlength{5ex}
\def\arrowdistance{0}
\begin{tikzcd}[column sep=\arrowlength]
\cdots 
\arrow[r, rightarrow, shift left=\arrowdistance, "\cdots"] 
& 
D(\A_{0})
\ar[r, rightarrow, shift left=\arrowdistance, "\A_{0}"] 
& 
D(\A_{1})
\arrow[r, rightarrow, shift left=\arrowdistance, "\A_{1}"] 
& 
\H{}{2}
\arrow[r, rightarrow, shift left=\arrowdistance, "\cdots"] 
&
\cdots.
\end{tikzcd}
\end{equation*}

In fact, we show that the assumptions of Lemma \ref{lem:cptembmaintheo} hold,
which provides an elegant, abstract, and short way to prove the crucial compact embeddings
\begin{align}
\label{cptemb1}
D(\A_{1})\cap D(\A_{0}^{*})\incl\H{}{1}
\end{align}
for the de Rham Hilbert complexes, cf.~Theorem \ref{theo:cptemb:derham},
Theorem \ref{theo:cptemb:derhamk}, and Theorem \ref{theo:cptemb:derhamvec},
Theorem \ref{theo:cptemb:derhamveck}.
In principle, our general technique 
-- compact embeddings by regular decompositions and Rellich's selection theorem --
works for all Hilbert complexes known in the literature, see, e.g., 
\cite{AH2020a} for a comprehensive list of such Hilbert complexes.

Roughly speaking a regular decomposition has the form
$$D(\A_{1})=\H{+}{1}+\A_{0}\H{+}{0}$$
with regular subspaces $\H{+}{0}\subset D(\A_{0})$ and $\H{+}{1}\subset D(\A_{1})$
such that the embeddings $\H{+}{0}\incl\H{}{0}$ and $\H{+}{1}\incl\H{}{1}$
are compact. The compactness is typically and simply given by Rellich's selection theorem,
which justifies the notion ``regular''.
By applying $\A_{1}$ any regular decomposition implies regular potentials
$$R(\A_{1})=\A_{1}\H{+}{1}$$
by the complex property. 
The respective regular potential and decomposition operators 
\begin{align*}
\PotP_{\A_{1}}:R(\A_{1})&\to\H{+}{1},
&
\PotQ_{\A_{1}}^{1}:D(\A_{1})&\to\H{+}{1},
&
\PotQ_{\A_{1}}^{0}:D(\A_{1})&\to\H{+}{0}
\end{align*}
are bounded and satisfy $\A_{1}\PotP_{\A_{1}}=\id_{R(\A_{1})}$ as well as
$\id_{D(\A_{1})}=\PotQ_{\A_{1}}^{1}+\A_{0}\PotQ_{\A_{1}}^{0}$.

Note that \eqref{cptemb1} implies several important results related 
to the particular Hilbert complex by the so-called FA-ToolBox, 
cf.~\cite{P2017a,P2019b,P2019a,P2020a} and \cite{PZ2016a,PZ2020a,PZ2020b}.
Upon others, one gets Friedrichs/Poincar\'e type estimates, closed ranges, compact resolvents, 
Helmholtz typ decompositions, comprehensive solution theories,
div-curl lemmas, discrete point spectra, eigenvector expansions, a posteriori error estimates,
and index theorems for related Dirac type operators.
See Theorem \ref{theo:minifatb:derham} and Theorem \ref{theo:minifatb:derhamvec} 
for a selection of such results.

For an historical overview on 
the compact embeddings \eqref{cptemb1} corresponding to the de Rham complex and Maxwell's equations,
also called Weck's or Weber-Weck-Picard's selection theorem,
see, e.g., the introductions in \cite{BPS2016a,NPW2015a},
the original papers \cite{W1974a,W1980a,P1984a,W1993a,J1997a,PWW2001a},
and the recent state of the art results 
for mixed boundary conditions and bounded weak Lipschitz domains in \cite{BPS2016a,BPS2018a,BPS2019a}.
Compact embeddings \eqref{cptemb1} corresponding to the biharmonic and the elasticity complex 
are given in \cite{PZ2020b} and \cite{PZ2016a,PZ2020a}, respectively.
Note that in the recent paper \cite{AH2020a}
similar results have been shown for the special case of no or full boundary conditions
using an alternative and more algebraic approach, 
the so-called Bernstein-Gelfand-Gelfand resolution (BGG). 

\section{FAT: FA-ToolBox}
\label{sec:FA}

We collect and present some old and new results from the so-called
functional analysis toolbox (FA-ToolBox).

\subsection{FAT I: Linear Operators, Adjoints, and Fundamental Lemmas}
\label{sec:FA1}

We shall work with bounded and unbounded linear operators.
For this, let $\H{}{0}$ and $\H{}{1}$ be Hilbert spaces.
For a \emph{bounded} linear operator $\A$ we use the notation
\begin{align}
\label{bdop}
\A:D(\A)\to\H{}{1}
\end{align}
where $D(\A)\subset\H{}{0}$ is the domain of definition of $\A$.
It's \emph{unbounded} version will be denoted by
\begin{align}
\label{unbdop}
\A:D(\A)\subset\H{}{0}\to\H{}{1}.
\end{align}
Kernel and range of $\A$ shall be denoted by $N(\A)$ and $R(\A)$, respectively.
Note that -- equipped with the standard graph inner product -- 
$D(\A)$ becomes a Hilbert space as long as $\A$ is closed.
The difference of the latter two versions of $\A$
comes from using the norm of $D(\A)$ or simply the norm of $\H{}{0}$, respectively.
Generally, inner product, norm, orthogonality, and orthogonal sum in a Hilbert space $\H{}{}$
shall be denoted by $\scp{\,\cdot\,}{\,\cdot\,}_{\H{}{}}$, $\norm{\,\cdot\,}_{\H{}{}}$,
$\bot_{\H{}{}}$, and $\oplus_{\H{}{}}$, respectively.
By $\dotplus$ we indicate a direct sum.
The dual space of a Banach or Hilbert space $\H{}{}$ 
will be written as $\H{'}{}$.

There are at least three different adjoints.
The bounded linear operator \eqref{bdop} has the \emph{Banach space adjoint}
$\A':\H{'}{1}\to D(\A)'$, which -- as usual -- may be identified with its modification 
$$\A'\R_{\H{}{1}}:\H{}{1}\to D(\A)',$$
where $\R_{\H{}{1}}:\H{}{1}\to\H{'}{1}$ denotes the Riesz isomorphism of $\H{}{1}$.
Another option is the \emph{Hilbert space adjoint} defined by
$$\A^{*}:=\R_{D(\A)}^{-1}\A'\R_{\H{}{1}}:\H{}{1}\to D(\A).$$
On the other hand, the unbounded linear operator \eqref{unbdop} has the \emph{Hilbert space adjoint}
$$\A^{*}:D(\A^{*})\subset\H{}{1}\to\H{}{0},$$
provided that $\A$ is densely defined. $\A^{*}$ is always closed and characterised by
$$\forall\,x\in D(\A)\quad
\forall\,y\in D(\A^{*})\qquad
\scp{\A x}{y}_{\H{}{1}}=\scp{x}{\A^{*}y}_{\H{}{0}}.$$
Note that the different adjoints are strongly related through the respective Riesz isomorphisms.
If the unbounded operator $\A$ is densely defined and closed, so is $\A^{*}$.
In this case, $\A^{**}=\ol{\A}=\A$ and we call $(\A,\A^{*})$ a dual pair.

Let us recall a small part of the co-called FA-ToolBox from, e.g.,
\cite[Lemma 4.1, Lemma 4.3]{P2019b},
see also \cite{P2017a,P2019a,P2020a,PZ2020a,PZ2020b},
for more details. For this, let $\A$ from \eqref{unbdop}
be \emph{densely defined} and \emph{closed}. Moreover, let 
\begin{align*}
\A_{\bot}:=\mathcal{A}:=\A|_{N(\A)^{\bot_{\H{}{0}}}}
:D(\A_{\bot})&\subset N(\A)^{\bot_{\H{}{0}}}\to N(\A^{*})^{\bot_{\H{}{1}}},
&
D(\A_{\bot})&:=D(\A)\cap N(\A)^{\bot_{\H{}{0}}},\\
\A^{*}_{\bot}:=\mathcal{A}^{*}:=\A^{*}|_{N(\A^{*})^{\bot_{\H{}{1}}}}
:D(\A^{*}_{\bot})&\subset N(\A^{*})^{\bot_{\H{}{1}}}\to N(\A)^{\bot_{\H{}{0}}},
&
D(\A^{*}_{\bot})&:=D(\A^{*})\cap N(\A^{*})^{\bot_{\H{}{1}}}
\end{align*}
denote the reduced operators, which are
densely defined, closed, and injective. 
Note that by the projection theorem we have the orthogonal Helmholtz-type decompositions
\begin{align}
\label{helm1}
\begin{aligned}
\H{}{0}&=N(\A)\oplus_{\H{}{0}}N(\A)^{\bot_{\H{}{0}}},
&
N(\A)^{\bot_{\H{}{0}}}&=\ol{R(\A^{*})},
&
N(\A)&=R(\A^{*})^{\bot_{\H{}{0}}},\\
D(\A)&=N(\A)\oplus_{\H{}{0}}D(\A_{\bot}),\\
\H{}{1}&=N(\A^{*})\oplus_{\H{}{1}}N(\A^{*})^{\bot_{\H{}{1}}},
&
N(\A^{*})^{\bot_{\H{}{1}}}&=\ol{R(\A)},
&
N(\A^{*})&=R(\A)^{\bot_{\H{}{1}}},\\
D(\A^{*})&=N(\A^{*})\oplus_{\H{}{1}}D(\A^{*}_{\bot}),
\end{aligned}
\end{align}
and thus $R(\A_{\bot})=R(\A)$ and $R(\A^{*}_{\bot})=R(\A^{*})$. 

\begin{lem}[fundamental lemma 1]
\label{lem:toolboxcpt1}
The following assertions are equivalent:
\begin{itemize}
\item[\bf(i)]
$\exists\hspace{1.3ex}c_{\A}>0\quad\forall\,x\in D(\A_{\bot})\qquad\norm{x}_{\H{}{0}}\leq c_{\A}\norm{\A x}_{\H{}{1}}$
\item[\bf(i')]
$\exists\,c_{\A^{*}}>0\quad\forall\,y\in D(\A^{*}_{\bot})\qquad\norm{y}_{\H{}{1}}\leq c_{\A^{*}}\norm{\A^{*}x}_{\H{}{0}}$
\item[\bf(ii)]
$R(\A)=R(\A_{\bot})$ is closed.
\item[\bf(ii')]
$R(\A^{*})=R(\A^{*}_{\bot})$ is closed.
\item[\bf(iii)]
$\A_{\bot}^{-1}:R(\A)\to D(\A_{\bot})$ is continuous.
\item[\bf(iii')]
$(\A^{*}_{\bot})^{-1}:R(\A^{*})\to D(\A^{*}_{\bot})$ is continuous.
\end{itemize}
Moreover, for the ``best'' constants it holds
$\bnorm{\A_{\bot}^{-1}}_{R(\A),\H{}{0}}
=c_{\A}
=c_{\A^{*}}
=\bnorm{(\A^{*}_{\bot})^{-1}}_{R(\A^{*}),\H{}{1}}$.
\end{lem}

\begin{lem}[fundamental lemma 2]
\label{lem:toolboxcpt2}
Let $D(\A_{\bot})\incl\H{}{0}$ be compact.
Then each of (i)-(iii') in Lemma \ref{lem:toolboxcpt1} holds.
\end{lem}

\begin{lem}[fundamental lemma 3]
\label{lem:toolboxcpt3}
The following assertions are equivalent:
\begin{itemize}
\item[\bf(i)]
$D(\A_{\bot})\incl\H{}{0}$ is compact.
\item[\bf(i')]
$D(\A^{*}_{\bot})\incl\H{}{1}$ is compact.
\item[\bf(ii)]
$\A_{\bot}^{-1}:R(\A)\to\H{}{0}$ is compact.
\item[\bf(ii')]
$(\A^{*}_{\bot})^{-1}:R(\A^{*})\to\H{}{1}$ is compact.
\end{itemize}
\end{lem}

\begin{rem}
\label{rem:app:toolboxcpt}
$D(\A)\incl\H{}{0}$ compact implies $D(\A_{\bot})\incl\H{}{0}$ compact, and
$D(\A^{*})\incl\H{}{1}$ compact implies $D(\A^{*}_{\bot})\incl\H{}{1}$ compact.
\end{rem}

\subsection{FAT II: Hilbert Complexes and Mini FA-ToolBox}
\label{sec:FA2}

We continue to make use of parts of the FA-ToolBox from, e.g., 
\cite{P2017a,P2019a,P2019b,P2020a} and \cite{PZ2016a,PZ2020a,PZ2020b},
together with an extension suited for so called (bounded linear) regular potential operators
and regular decompositions introduced in \cite{PZ2020b}.
Lemma \ref{lem:cptembmaintheo} provides an elegant, abstract, and short way
to prove compact embedding results for an arbitrary Hilbert complex.

For this, let $\H{}{0},\H{}{1},\H{}{2}$ be Hilbert spaces and let
\begin{equation}
\label{hcomplex}
\def\arrowlength{6ex}
\def\arrowdistance{.8}
\begin{tikzcd}[column sep=\arrowlength]
\cdots 
\arrow[r, rightarrow, shift left=\arrowdistance, "\cdots"] 
\arrow[r, leftarrow, shift right=\arrowdistance, "\cdots"']
& 
\H{}{0} 
\ar[r, rightarrow, shift left=\arrowdistance, "\A_{0}"] 
\ar[r, leftarrow, shift right=\arrowdistance, "\A_{0}^{*}"']
& 
\H{}{1}
\arrow[r, rightarrow, shift left=\arrowdistance, "\A_{1}"] 
\arrow[r, leftarrow, shift right=\arrowdistance, "\A_{1}^{*}"']
& 
\H{}{2}
\arrow[r, rightarrow, shift left=\arrowdistance, "\cdots"] 
\arrow[r, leftarrow, shift right=\arrowdistance, "\cdots"']
&
\cdots 
\end{tikzcd}
\end{equation}
be a \emph{primal and dual Hilbert complex}, i.e.,
$$\A_{0}:D(\A_{0})\subset\H{}{0}\to\H{}{1},\qquad
\A_{1}:D(\A_{1})\subset\H{}{1}\to\H{}{2}$$
are \emph{densely defined} and \emph{closed} (unbounded) linear operators
satisfying the \emph{complex property} 
\begin{align}
\label{compprop}
\A_{1}\A_{0}\subset0,
\end{align}
together with (densely defined and closed Hilbert space) adjoints
$$\A_{0}^{*}:D(\A_{0}^{*})\subset\H{}{1}\to\H{}{0},\qquad
\A_{1}^{*}:D(\A_{1}^{*})\subset\H{}{2}\to\H{}{1}.$$

\begin{rem}
\label{remhilcom}
Note that the complex property \eqref{compprop} is equivalent to $R(\A_{0})\subset N(\A_{1})$,
which is equivalent to the dual complex property $R(\A_{1}^{*})\subset N(\A_{0}^{*})$ as 
$$R(\A_{1}^{*})
\subset\ol{R(\A_{1}^{*})}
=N(\A_{1})^{\bot_{\H{}{1}}}
\subset R(\A_{0})^{\bot_{\H{}{1}}}
=N(\A_{0}^{*})$$
and vice versa. 
\end{rem}

\begin{rem}
\label{remhilcomclosure}
Let $\A_{0}$, $\A_{1}$ be given by the closures of 
densely defined (unbounded) linear operators
$$\mr{\A}_{0}:D(\mr{\A}_{0})\subset\H{}{0}\to\H{}{1},\qquad
\mr{\A}_{1}:D(\mr{\A}_{1})\subset\H{}{1}\to\H{}{2}$$
satisfying the complex property $\mr{\A}_{1}\mr{\A}_{0}\subset0$.
Then $\A_{0}=\ol{\mr{\A}_{0}}$ and $\A_{1}=\ol{\mr{\A}_{1}}$
are densely defined and closed (unbounded) linear operators
satisfying the complex property $\A_{1}\A_{0}\subset0$,
since $N(\A_{1})$ is closed and thus
$R(\mr{\A}_{0})\subset N(\mr{\A}_{1})\subset N(\A_{1})$
implies $R(\A_{0})\subset N(\A_{1})$.
\end{rem}

As in \eqref{helm1} and defining the cohomology group 
$$N_{0,1}:=N(\A_{1})\cap N(\A_{0}^{*})$$
we get the following orthogonal Helmholtz-type decompositions.

\begin{lem}[Helmholtz decomposition lemma]
\label{lem:toolboxhelm1}
The refined orthogonal Helmholtz-type decompositions
\begin{align}
\label{helm2}
\begin{aligned}
\H{}{1}&=\ol{R(\A_{0})}\oplus_{\H{}{1}}N(\A_{0}^{*}),
&
\H{}{1}&=N(\A_{1})\oplus_{\H{}{1}}\ol{R(\A_{1}^{*})},\\
N(\A_{1})&=\ol{R(\A_{0})}\oplus_{\H{}{1}}N_{0,1},
&
N(\A_{0}^{*})&=N_{0,1}\oplus_{\H{}{1}}\ol{R(\A_{1}^{*})},\\
D(\A_{1})&=\ol{R(\A_{0})}\oplus_{\H{}{1}}\big(D(\A_{1})\cap N(\A_{0}^{*})\big),
&
D(\A_{0}^{*})&=\big(N(\A_{1})\cap D(\A_{0}^{*})\big)\oplus_{\H{}{1}}\ol{R(\A_{1}^{*})},\\
D(\A_{0}^{*})&=D\big((\A_{0}^{*})_{\bot}\big)\oplus_{\H{}{1}}N(\A_{0}^{*}),
&
D(\A_{1})&=N(\A_{1})\oplus_{\H{}{1}}D\big((\A_{1})_{\bot}\big),
\end{aligned}
\end{align}
as well as $R\big((\A_{0}^{*})_{\bot}\big)=R(\A_{0}^{*})$ 
and $R\big((\A_{1})_{\bot}\big)=R(\A_{1})$ hold. Moreover,
\begin{align}
\label{helm3}
\begin{aligned}
\H{}{1}
&=\ol{R(\A_{0})}\oplus_{\H{}{1}}N_{0,1}\oplus_{\H{}{1}}\ol{R(\A_{1}^{*})},\\
D(\A_{0}^{*})
&=D\big((\A_{0}^{*})_{\bot}\big)\oplus_{\H{}{1}}N_{0,1}\oplus_{\H{}{1}}\ol{R(\A_{1}^{*})},\\
D(\A_{1})
&=\ol{R(\A_{0})}\oplus_{\H{}{1}}N_{0,1}\oplus_{\H{}{1}}D\big((\A_{1})_{\bot}\big),\\
D(\A_{1})\cap D(\A_{0}^{*})
&=D\big((\A_{0}^{*})_{\bot}\big)\oplus_{\H{}{1}}N_{0,1}\oplus_{\H{}{1}}D\big((\A_{1})_{\bot}\big).
\end{aligned}
\end{align}
\end{lem}

As 
\begin{align*}
D\big((\A_{1})_{\bot}\big)
&=D(\A_{1})\cap\ol{R(\A_{1}^{*})}
\subset D(\A_{1})\cap N(\A_{0}^{*})
\subset D(\A_{1})\cap D(\A_{0}^{*}),\\
D\big((\A_{0}^{*})_{\bot}\big)
&=\ol{R(\A_{0})}\cap D(\A_{0}^{*})
\subset N(\A_{1})\cap D(\A_{0}^{*})
\subset D(\A_{1})\cap D(\A_{0}^{*})
\end{align*}
with continuous embeddings we get the following result.

\begin{lem}[compactness lemma]
\label{lem:toolboxcpt4}
The following assertions are equivalent:
\begin{itemize}
\item[\bf(i)]
$D\big((\A_{0})_{\bot}\big)\incl\H{}{0}$,
$D\big((\A_{1})_{\bot}\big)\incl\H{}{1}$, and 
$N_{0,1}\incl\H{}{1}$ are compact.
\item[\bf(ii)]
$D(\A_{1})\cap D(\A_{0}^{*})\incl\H{}{1}$ is compact.
\end{itemize}
In this case, the cohomology group $N_{0,1}$ has finite dimension.
\end{lem}

Summarising the latter results we get the following theorem.

\begin{theo}[mini FAT]
\label{theo:toolboxgenmain}
Let $D(\A_{1})\cap D(\A_{0}^{*})\incl\H{}{1}$ be compact. Then:
\begin{itemize}
\item[\bf(i)]
The ranges $R(\A_{0})$, $R(\A_{0}^{*})$ and $R(\A_{1})$, $R(\A_{1}^{*})$ are closed.
\item[\bf(ii)]
The inverse operators $(\A_{0})_{\bot}^{-1}$, $(\A_{0}^{*})_{\bot}^{-1}$ 
and $(\A_{1})_{\bot}^{-1}$, $(\A_{1}^{*})_{\bot}^{-1}$ are compact.
\item[\bf(iii)]
The cohomology group $N_{0,1}=N(\A_{1})\cap N(\A_{0}^{*})$ has finite dimension.
\item[\bf(iv)]
The orthogonal Helmholtz-type decomposition
$\H{}{1}=R(\A_{0})\oplus_{\H{}{1}}N_{0,1}\oplus_{\H{}{1}}R(\A_{1}^{*})$ 
holds.
\item[\bf(v)]
There exist $c_{\A_{0}},c_{\A_{1}}>0$ such that
\begin{align*}
\forall\,x&\in D\big((\A_{0})_{\bot}\big)
=D(\A_{0})\cap N(\A_{0})^{\bot_{\H{}{0}}}=D(\A_{0})\cap R(\A_{0}^{*})
&
\norm{x}_{\H{}{0}}&\leq c_{\A_{0}}\norm{\A_{0}x}_{\H{}{1}},\\
\forall\,y&\in D\big((\A_{0}^{*})_{\bot}\big)
=D(\A_{0}^{*})\cap N(\A_{0}^{*})^{\bot_{\H{}{1}}}=D(\A_{0}^{*})\cap R(\A_{0})
&
\norm{y}_{\H{}{1}}&\leq c_{\A_{0}}\norm{\A_{0}^{*}y}_{\H{}{0}},\\
\forall\,y&\in D\big((\A_{1})_{\bot}\big)
=D(\A_{1})\cap N(\A_{1})^{\bot_{\H{}{1}}}=D(\A_{1})\cap R(\A_{1}^{*})
&
\norm{y}_{\H{}{1}}&\leq c_{\A_{1}}\norm{\A_{1}y}_{\H{}{2}},\\
\forall\,z&\in D\big((\A_{1}^{*})_{\bot}\big)
=D(\A_{1}^{*})\cap N(\A_{1}^{*})^{\bot_{\H{}{2}}}=D(\A_{1}^{*})\cap R(\A_{1})
&
\norm{z}_{\H{}{2}}&\leq c_{\A_{1}}\norm{\A_{1}^{*}z}_{\H{}{1}}.
\end{align*}
\item[\bf(v')]
With $c_{\A_{0}}$ and $c_{\A_{1}}$ from (v) it holds
$$\forall\,y\in D(\A_{1})\cap D(\A_{0}^{*})\cap N_{0,1}^{\bot_{\H{}{1}}}
\qquad\norm{y}_{\H{}{1}}^{2}\leq 
c_{\A_{0}}^{2}\norm{\A_{0}^{*}y}_{\H{}{0}}^{2}
+c_{\A_{1}}^{2}\norm{\A_{1}y}_{\H{}{2}}^{2}.$$
\end{itemize}
\end{theo}

\begin{defi}
\label{defihilcom}
The Hilbert complex \eqref{hcomplex} is called 
\begin{itemize}
\item
closed, if $R(\A_{0})$ and $R(\A_{1})$ are closed,
\item
compact, if the embedding $D(\A_{1})\cap D(\A_{0}^{*})\incl\H{}{1}$ is compact.
\end{itemize}
\end{defi}

\begin{rem}
\label{remhilcom2}
A compact Hilbert complex is already closed.
\end{rem}

\subsection{FAT III: Bounded Regular Decompositions and Potentials}
\label{sec:FA3}

Bounded regular decompositions and bounded regular potentials are very powerful tools.
In particular, compact embeddings can easily be proved, cf.~Lemma \ref{lem:cptembmaintheo},
which then -- in combination with the FA-ToolBox --
immediately lead to a comprehensive list of important results for the underlying Hilbert complex,
cf.~Theorem \ref{theo:toolboxgenmain} and \cite{P2020a}. 

Throughout this subsection, 
let $\A_{0}$ and $\A_{1}$ be \emph{densely defined} and \emph{closed} linear operators
satisfying the \emph{complex property}, i.e., $R(\A_{0})\subset N(\A_{1})$.
Moreover, we fix some \emph{regular subspaces}
$\H{+}{0}$, $\H{+}{1}$, and $\H{+}{2}$, such that either
\begin{align}
\begin{aligned}
\label{regsubspaceintro}
&&
\H{+}{0}\incl&D(\A_{0})\incl\H{}{0}
&&\quad\text{and}\quad&
\H{+}{1}\incl&D(\A_{1})\incl\H{}{1},\\
\text{or}\quad&&
\H{+}{1}\incl&D(\A_{0}^{*})\incl\H{}{1}
&&\quad\text{and}\quad&
\H{+}{2}\incl&D(\A_{1}^{*})\incl\H{}{2}
\end{aligned}
\end{align}
hold with continuous embeddings.
In the following, we consider \emph{regular decompositions} of $D(\A_{1})$ and $D(\A_{0}^{*})$
of the following type
\begin{align}
\label{bdregdeco1}
D(\A_{1})=\H{+}{1}+\A_{0}\H{+}{0},\qquad
D(\A_{0}^{*})=\H{+}{1}+\A_{1}^{*}\H{+}{2}.
\end{align}

For the rest of this subsection we concentrate on the first 
regular decomposition in \eqref{bdregdeco1}.
Analogous results hold true for the second regular decomposition in \eqref{bdregdeco1},
and we leave the corresponding reformulations to the interested reader.

\begin{defi}[bounded regular decompositions]
\label{def:cptembmaindef1}
In \eqref{bdregdeco1} we call the regular decomposition
$D(\A_{1})=\H{+}{1}+\A_{0}\H{+}{0}$ bounded, if there exist bounded linear operators
$$\PotQ_{\A_{1},1}:D(\A_{1})\to\H{+}{1},\qquad
\PotQ_{\A_{1},0}:D(\A_{1})\to\H{+}{0},$$
such that 
$$\PotQ_{\A_{1},1}+\A_{0}\PotQ_{\A_{1},0}=\id_{D(\A_{1})}.$$
$\PotQ_{\A_{1},1}$ and $\PotQ_{\A_{1},0}$ are then called bounded linear regular decomposition operators.

More precisely, for each $x\in D(\A_{1})$ there exist two potentials 
$$x_{1}:=\PotQ_{\A_{1},1}x\in\H{+}{1},\qquad
z:=\PotQ_{\A_{1},0}x\in\H{+}{0},$$ 
such that $x=x_{1}+\A_{0}z$ 
and $\norm{x_{1}}_{\H{+}{1}}+\norm{z}_{\H{+}{0}}\leq c\norm{x}_{D(\A_{1})}$
with some $c>0$ independent of $x,x_{1},z$.
\end{defi}

\begin{defi}[weak bounded regular decompositions]
\label{def:cptembmaindef1weak}
$D(\A_{1})=\H{+}{1}+N(\A_{1})$
is called a weak bounded regular decomposition, if there exist bounded linear operators
$$\PotQ_{\A_{1},1}:D(\A_{1})\to\H{+}{1},\qquad
\PotN_{\A_{1}}:D(\A_{1})\to N(\A_{1})$$
such that $\PotQ_{\A_{1},1}+\PotN_{\A_{1}}=\id_{D(\A_{1})}$.
$\PotQ_{\A_{1},1}$ and $\PotN_{\A_{1}}$ are again called bounded linear regular decomposition operators.

More precisely, for each $x\in D(\A_{1})$ there exist
$$x_{1}:=\PotQ_{\A_{1},1}x\in\H{+}{1},\qquad
x_{0}:=\PotN_{\A_{1}}x\in N(\A_{1}),$$ 
such that $x=x_{1}+x_{0}$ 
and $\norm{x_{1}}_{\H{+}{1}}+\norm{x_{0}}_{\H{}{1}}\leq c\norm{x}_{D(\A_{1})}$
with some $c>0$ independent of $x,x_{1},x_{0}$.
\end{defi}

\begin{rem}[bounded regular decompositions]
\label{rem:cptembmaindef1}
For bounded regular decompositions it holds:
\begin{itemize}
\item[\bf(i)]
For $\PotQ_{\A_{1},1}$ from Definition \ref{def:cptembmaindef1} or Definition \ref{def:cptembmaindef1weak}
we have $\A_{1}\PotQ_{\A_{1},1}=\A_{1}$ by the complex property.
Hence $N(\A_{1})$ is invariant under $\PotQ_{\A_{1},1}$, i.e.,
$\PotQ_{\A_{1},1}N(\A_{1})\subset N(\A_{1})$.
\item[\bf(ii)]
A bounded regular decomposition from Definition \ref{def:cptembmaindef1}
implies a weak bounded regular decomposition from Definition \ref{def:cptembmaindef1weak}
by setting $\PotN_{\A_{1}}:=\A_{0}\PotQ_{\A_{1},0}$
since $\A_{0}\H{+}{0}\subset N(\A_{1})$ holds by the complex property.
\end{itemize}
\end{rem}

\begin{defi}[bounded regular potentials]
\label{def:cptembmaindef2}
We call $R(\A_{1})=\A_{1}\H{+}{1}$ a bounded regular potential representation,
if there exists a bounded linear operator
$$\PotP_{\A_{1}}:R(\A_{1})\to\H{+}{1}
\qquad\text{with}\qquad
\A_{1}\PotP_{\A_{1}}=\id_{R(\A_{1})}.$$
We say that $\PotP_{\A_{1}}$ is a bounded linear regular potential operator of $\A_{1}$.
In particular, $\PotP_{\A_{1}}$ is a bounded linear right inverse of $\A_{1}$.
\end{defi}

Analogously, we extend the latter definition to the operators $\A_{0}$, $\A_{0}^{*}$, and $\A_{1}^{*}$.

\begin{rem}[bounded regular potentials]
\label{rem:cptembmaindef2rem}
We state two simple facts about potential operators:
\begin{itemize}
\item[\bf(i)]
Let a linear operator
$$\PotP_{\A_{0}}:N(\A_{1})\cap N_{0,1}^{\bot_{\H{}{1}}}\to D(\A_{0})
\qquad\text{with}\qquad
\A_{0}\PotP_{\A_{0}}=\id_{N(\A_{1})\cap N_{0,1}^{\bot_{\H{}{1}}}}$$
be given. Then $R(\A_{0})$ is closed as
$\ol{R(\A_{0})}=N(\A_{1})\cap N_{0,1}^{\bot_{\H{}{1}}}=R(\A_{0}\PotP_{\A_{0}})\subset R(\A_{0})$.
\item[\bf(ii)]
Let a bounded linear operator
$$\PotP_{\A_{0}}:N(\A_{1})\cap N_{0,1}^{\bot_{\H{}{1}}}\to\H{+}{0}
\qquad\text{with}\qquad
\A_{0}\PotP_{\A_{0}}=\id_{N(\A_{1})\cap N_{0,1}^{\bot_{\H{}{1}}}}$$
be given. Then (as above) $R(\A_{0})=N(\A_{1})\cap N_{0,1}^{\bot_{\H{}{1}}}=\A_{0}\H{+}{0}$ 
is closed and
$$\PotP_{\A_{0}}:R(\A_{0})\to\H{+}{0}
\qquad\text{with}\qquad
\A_{0}\PotP_{\A_{0}}=\id_{R(\A_{0})}$$
is a bounded linear regular potential operator of $\A_{0}$.
\end{itemize}
\end{rem}

\begin{lem}[bounded regular potentials by weak bounded regular decompositions]
\label{lem:regpotdeco1}
Let $R(\A_{1})$ be closed,
and let $D(\A_{1})=\H{+}{1}+N(\A_{1})$ be a weak bounded regular decomposition.
Then the bounded regular potential representation
$R(\A_{1})=\A_{1}\H{+}{1}$
holds and 
$$\PotP_{\A_{1}}:=\PotQ_{\A_{1},1}(\A_{1})_{\bot}^{-1}:R(\A_{1})\to\H{+}{1}
\qquad\text{with}\qquad
\A_{1}\PotP_{\A_{1}}=\id_{R(\A_{1})}$$
is a respective bounded linear regular potential operator of $\A_{1}$.
\end{lem}

\begin{proof}
As $R(\A_{1})$ is closed, Lemma \ref{lem:toolboxcpt1}
shows that $(\A_{1})_{\bot}^{-1}:R(\A_{1})\to D(\A_{1})$ is bounded.
Hence so is $\PotP_{\A_{1}}$. Moreover, 
$\A_{1}\PotP_{\A_{1}}
=\A_{1}\PotQ_{\A_{1},1}(\A_{1})_{\bot}^{-1}
=\A_{1}(\A_{1})_{\bot}^{-1}
=\id_{R(\A_{1})}$
by Remark \ref{rem:cptembmaindef1}.
\end{proof}

\begin{lem}[weak bounded regular decompositions by bounded regular potentials]
\label{lem:regpotdeco2}
Let a bounded regular potential representation $R(\A_{1})=\A_{1}\H{+}{1}$ be given
with bounded linear regular potential operator
$\PotP_{\A_{1}}:R(\A_{1})\to\H{+}{1}$
satisfying
$\A_{1}\PotP_{\A_{1}}=\id_{R(\A_{1})}$.
Then
$$\PotQ_{\A_{1},1}:=\PotP_{\A_{1}}\A_{1}:D(\A_{1})\to\H{+}{1},\qquad
\PotN_{\A_{1}}:=\id_{D(\A_{1})}-\PotQ_{\A_{1},1}:D(\A_{1})\to N(\A_{1})$$
are bounded linear regular decomposition operators with 
$$\PotQ_{\A_{1},1}+\PotN_{\A_{1}}=\id_{D(\A_{1})}$$
and implying the weak bounded regular decompositions
$$D(\A_{1})
=\H{+}{1}+N(\A_{1})
=R(\PotQ_{\A_{1},1})\dotplus N(\A_{1})
=R(\PotQ_{\A_{1},1})\dotplus R(\PotN_{\A_{1}}).$$
It holds $\A_{1}\PotQ_{\A_{1},1}=\A_{1}$, i.e., $N(\A_{1})$ is invariant under $\PotQ_{\A_{1},1}$.
Note that $R(\PotQ_{\A_{1},1})=R(\PotP_{\A_{1}})$.
\end{lem}

\begin{proof}
$\PotQ_{\A_{1},1}$ and $\PotN_{\A_{1}}$ are bounded.
Let $x\in D(\A_{1})$. 
Then $\A_{1}x\in R(\A_{1})$
and $\PotP_{\A_{1}}\A_{1}x\in\H{+}{1}$ with
$\widetilde{x}:=x-\PotP_{\A_{1}}\A_{1}x\in N(\A_{1})$.
For the directness, let 
$x=\PotQ_{\A_{1},1}\varphi=\PotP_{\A_{1}}\A_{1}\varphi\in N(\A_{1})$
with $\varphi\in D(\A_{1})$. Then $0=\A_{1} x=\A_{1}\varphi$
and hence $x=0$.
\end{proof}

\begin{rem}
\label{rem:regpotdeco2}
Note that $\PotQ_{\A_{1},1}^2=\PotQ_{\A_{1},1}$ 
and $\PotQ_{\A_{1},1}\PotN_{\A_{1}}=\PotN_{\A_{1}}\PotQ_{\A_{1},1}=0$
hold for the special bounded linear regular decomposition operator 
$\PotQ_{\A_{1},1}=\PotP_{\A_{1}}\A_{1}$ from the latter lemma. Hence:
\begin{itemize}
\item[\bf(i)]
$\PotQ_{\A_{1},1}$ and $\PotN_{\A_{1}}$ are projections.
\item[\bf(ii)]
For $I_{\pm}:=\PotQ_{\A_{1},1}\pm\PotN_{\A_{1}}$ we observe
$I_{+}=I_{-}^2=\id_{D(\A_{1})}$. Thus the operators $I_{+}$, $I_{-}^2$, as well as
$I_{-}=2\PotQ_{\A_{1},1}-\id_{D(\A_{1})}$ are topological isomorphisms on $D(\A_{1})$.
\item[\bf(iii)]
There exists $c>0$ such that for $x\in D(\A_{1})$ it holds
\begin{align*}
c\norm{\PotQ_{\A_{1},1}x}_{\H{+}{1}}
&\leq\norm{\A_{1}x}_{\H{}{2}}
\leq\norm{x}_{D(\A_{1})},
&
\norm{\PotN_{\A_{1}}x}_{\H{}{1}}
&\leq\norm{x}_{\H{}{1}}
+\norm{\PotQ_{\A_{1},1}x}_{\H{}{1}}.
\end{align*}
\item[\bf(iii')]
For $x\in N(\A_{1})$ we have $\PotQ_{\A_{1},1}x=0$
and $\PotN_{\A_{1}}x=x$, i.e.,
$\PotQ_{\A_{1},1}|_{N(\A_{1})}=0$
as well as $\PotN_{\A_{1}}|_{N(\A_{1})}=\id_{N(\A_{1})}$.
In particular, $\PotN_{\A_{1}}$ is onto.
\end{itemize}
\end{rem}

\begin{cor}[bounded regular decompositions by bounded regular potentials]
\label{cor:regpotdeco2}
Let the complex be exact, i.e., $N(\A_{1})=R(\A_{0})$, and let
$R(\A_{1})=\A_{1}\H{+}{1}$ 
as well as $R(\A_{0})=\A_{0}\H{+}{0}$
be bounded regular potential representations
with bounded linear regular potential operators
$\PotP_{\A_{1}}:R(\A_{1})\to\H{+}{1}$ 
and $\PotP_{\A_{0}}:R(\A_{0})\to\H{+}{0}$
satisfying 
$\A_{1}\PotP_{\A_{1}}=\id_{R(\A_{1})}$
and
$\A_{0}\PotP_{\A_{0}}=\id_{R(\A_{0})}$,
respectively. Then
$$\PotQ_{\A_{1},1}:D(\A_{1})\to\H{+}{1},\qquad
\PotQ_{\A_{1},0}:=\PotP_{\A_{0}}\PotN_{\A_{1}}:D(\A_{1})\to\H{+}{0}$$
with $\PotQ_{\A_{1},1}=\PotP_{\A_{1}}\A_{1}$ and
$\PotN_{\A_{1}}=\id_{D(\A_{1})}-\PotQ_{\A_{1},1}$ from Lemma \ref{lem:regpotdeco2}
are bounded linear regular decomposition operators with 
$$\PotQ_{\A_{1},1}+\A_{0}\PotQ_{\A_{1},0}=\id_{D(\A_{1})}$$
and implying bounded regular decompositions 
$$D(\A_{1})
=\H{+}{1}+\A_{0}\H{+}{0}
=R(\PotQ_{\A_{1},1})\dotplus\A_{0}\H{+}{0}
=R(\PotQ_{\A_{1},1})\dotplus\A_{0}R(\PotQ_{\A_{1},0}).$$
It holds $\A_{1}\PotQ_{\A_{1},1}=\A_{1}$, i.e., $N(\A_{1})$ is invariant under $\PotQ_{\A_{1},1}$.
Note that $R(\PotQ_{\A_{1},1})=R(\PotP_{\A_{1}})$
and $R(\PotQ_{\A_{1},0})=R(\PotP_{\A_{0}})$.
\end{cor}

\begin{proof}
$\PotQ_{\A_{1},1}$ and $\PotQ_{\A_{1},0}$ are bounded.
Let $x\in D(\A_{1})$. 
Then $\A_{1}x\in R(\A_{1})$
and $\PotP_{\A_{1}}\A_{1}x\in\H{+}{1}$ with
$\widetilde{x}:=x-\PotP_{\A_{1}}\A_{1}x\in N(\A_{1})=R(\A_{0})$.
Thus $z:=\PotP_{\A_{0}}\widetilde{x}\in\H{+}{0}$ and $\A_{0}z=\widetilde{x}$, i.e.,
$$x=\PotP_{\A_{1}}\A_{1}x+\widetilde{x}
=\PotP_{\A_{1}}\A_{1}x+\A_{0}\PotP_{\A_{0}}\widetilde{x}
=\PotP_{\A_{1}}\A_{1}x+\A_{0}\PotP_{\A_{0}}(1-\PotP_{\A_{1}}\A_{1})x.$$
Directness is clear by Lemma \ref{lem:regpotdeco2} as 
$\A_{0}\H{+}{0}\subset N(\A_{1})$ holds by the complex property.
\end{proof}

\begin{rem}
\label{rem:regpotdeco3}
There exists $c>0$ such that for $x\in D(\A_{1})$ it holds
\begin{align*}
c\norm{\PotQ_{\A_{1},1}x}_{\H{+}{1}}
&\leq\norm{\A_{1}x}_{\H{}{2}}
\leq\norm{x}_{D(\A_{1})},
&
c\norm{\PotQ_{\A_{1},0}x}_{\H{+}{0}}
&\leq\norm{\PotN_{\A_{1}}x}_{\H{}{1}}
\leq\norm{x}_{\H{}{1}}
+\norm{\PotQ_{\A_{1},1}x}_{\H{}{1}}.
\end{align*}
Note that $\PotQ_{\A_{1},1}|_{N(\A_{1})}=0$.
\end{rem}

\subsection{FAT IV: Compactness Results and Mini FA-ToolBox}
\label{sec:FA4}

From \cite[Theorem 2.8, Corollary 2.9]{PZ2020b} we cite the following compactness result.

\begin{lem}[compact embedding by regular decompositions]
\label{lem:cptembmaintheo}
Let $\A_{0}$ and $\A_{1}$ be densely defined and closed linear operators
satisfying the complex property, i.e., $R(\A_{0})\subset N(\A_{1})$.
Moreover, let
\begin{itemize}
\item[\bf(i)] 
either the bounded regular decomposition
$D(\A_{1})=\H{+}{1}+\A_{0}\H{+}{0}$ hold 
with compact embeddings $\H{+}{0}\incl\H{}{0}$ and $\H{+}{1}\incl\H{}{1}$,
\item[\bf(ii)] 
or the bounded regular decomposition
$D(\A_{0}^{*})=\H{+}{1}+\A_{1}^{*}\H{+}{2}$ hold 
with compact embeddings $\H{+}{1}\incl\H{}{1}$ and $\H{+}{2}\incl\H{}{2}$.
\end{itemize}
Then the embedding $D(\A_{1})\cap D(\A_{0}^{*})\incl\H{}{1}$ is compact.
\end{lem}

For convenience we repeat the proof of \cite[Theorem 2.8]{PZ2020b}.

\begin{proof}
Let $(x_{n})\subset D(\A_{1})\cap D(\A_{0}^{*})$ be a bounded sequence, 
i.e., there exists $c>0$ such that for all $n$ we have
$\norm{x_{n}}_{\H{}{1}}
+\norm{\A_{1} x_{n}}_{\H{}{2}}
+\norm{\A_{0}^{*}x_{n}}_{\H{}{0}}
\leq c$.
By assumption we decompose 
$x_{n}=p_{1,n}+\A_{0}p_{0,n}$ with $p_{1,n}\in\H{+}{1}$
and $p_{0,n}\in\H{+}{0}$ satisfying
$\norm{p_{1,n}}_{\H{+}{1}}
+\norm{p_{0,n}}_{\H{+}{0}}
\leq c\norm{x_{n}}_{D(\A_{1})}
\leq c.$
Hence $(p_{\ell,n})\subset\H{+}{\ell}$ is bounded in $\H{+}{\ell}$, $\ell=0,1$,
and thus we can extract convergent subsequences, again denoted by $(p_{\ell,n})$,
such that $(p_{\ell,n})$ are convergent in $\H{}{\ell}$, $\ell=0,1$. 
Then with $x_{n,m}:=x_{n}-x_{m}$ and $p_{\ell,n,m}:=p_{\ell,n}-p_{\ell,m}$ we get
$$\norm{x_{n,m}}_{\H{}{1}}^2
=\scp{x_{n,m}}{p_{1,n,m}}_{\H{}{1}}
+\scp{\A_{0}^{*}x_{n,m}}{p_{0,n,m}}_{\H{}{0}}\\
\leq c\big(\norm{p_{1,n,m}}_{\H{}{1}}
+\norm{p_{0,n,m}}_{\H{}{0}}\big),$$
which shows that $(x_{n})$ is a Cauchy sequence in $\H{}{1}$.
Hence we have shown (i), and (ii) follows analogously.
\end{proof}

\begin{theo}[mini FAT by regular decompositions]
\label{theo:cptembmaintheo1}
Let the assumptions of Lemma \ref{lem:cptembmaintheo} (i) hold
with the bounded linear regular decomposition operators
$\PotQ_{\A_{1},1}:D(\A_{1})\to\H{+}{1}$ as well as $\PotQ_{\A_{1},0}:D(\A_{1})\to\H{+}{0}$. 
Then:
\begin{itemize}
\item[\bf(i)] 
The embedding $D(\A_{1})\cap D(\A_{0}^{*})\incl\H{}{1}$ is compact.
\item[\bf(ii)] 
The assertions of Theorem \ref{theo:toolboxgenmain} (mini FAT) hold.
\item[\bf(iii)] 
The bounded regular potential representation
$R(\A_{1})=\A_{1}\H{+}{1}$
holds with bounded linear regular potential operator
$\PotP_{\A_{1}}=\PotQ_{\A_{1},1}(\A_{1})_{\bot}^{-1}:R(\A_{1})\to\H{+}{1}$
satisfying $\A_{1}\PotP_{\A_{1}}=\id_{R(\A_{1})}$.
\item[\bf(iv)] 
$\widetilde\PotQ_{\A_{1},1}=\PotP_{\A_{1}}\A_{1}:D(\A_{1})\to\H{+}{1}$
and $\widetilde\PotN_{\A_{1}}=\id_{D(\A_{1})}-\widetilde\PotQ_{\A_{1},1}:D(\A_{1})\to N(\A_{1})$
are bounded linear regular decomposition operators with 
$\widetilde\PotQ_{\A_{1},1}+\widetilde\PotN_{\A_{1}}=\id_{D(\A_{1})}$
and the bounded regular decompositions
$$D(\A_{1})
=\H{+}{1}+\A_{0}\H{+}{0}
=\H{+}{1}+N(\A_{1})
=R(\widetilde\PotQ_{\A_{1},1})\dotplus N(\A_{1})
=R(\widetilde\PotQ_{\A_{1},1})\dotplus R(\widetilde\PotN_{\A_{1}})$$
hold. Moreover, $R(\widetilde\PotQ_{\A_{1},1})=R(\PotP_{\A_{1}})$.
\item[\bf(iv')] 
$\A_{1}\widetilde\PotQ_{\A_{1},1}=\A_{1}\PotQ_{\A_{1},1}=\A_{1}$, i.e., $N(\A_{1})$ 
is invariant under $\PotQ_{\A_{1},1}$ and $\widetilde\PotQ_{\A_{1},1}$.
It holds $\widetilde\PotQ_{\A_{1},1}=\PotQ_{\A_{1},1}(\A_{1})_{\bot}^{-1}\A_{1}$.
Hence $\widetilde\PotQ_{\A_{1},1}|_{D((\A_{1})_{\bot})}=\PotQ_{\A_{1},1}|_{D((\A_{1})_{\bot})}$
and thus $\widetilde\PotQ_{\A_{1},1}$ may differ from $\PotQ_{\A_{1},1}$ only on $N(\A_{1})$.
\end{itemize}
\end{theo}

\begin{proof}
(i) and (ii) are trivial. (iii) follows by Lemma \ref{lem:regpotdeco1}
and Lemma \ref{lem:regpotdeco2} shows (iv). It holds
\begin{align*}
\widetilde\PotQ_{\A_{1},1}|_{D((\A_{1})_{\bot})}
=\PotQ_{\A_{1},1}(\A_{1})_{\bot}^{-1}\A_{1}|_{D((\A_{1})_{\bot})}
&=\PotQ_{\A_{1},1}(\A_{1})_{\bot}^{-1}(\A_{1})_{\bot}\\
&=\PotQ_{\A_{1},1}\id_{D((\A_{1})_{\bot})}
=\PotQ_{\A_{1},1}|_{D((\A_{1})_{\bot})},
\end{align*}
which shows the last assertion of (iv').
\end{proof}

\begin{cor}[mini FAT by regular decompositions]
\label{cor:cptembmaintheo2}
Let the assumptions of Lemma \ref{lem:cptembmaintheo} (ii) hold
with the bounded linear regular decomposition operators
$\PotQ_{\A_{0}^{*},1}:D(\A_{1})\to\H{+}{1}$ as well as $\PotQ_{\A_{0}^{*},2}:D(\A_{1})\to\H{+}{2}$. 
Then (i) and (ii) of Theorem \ref{theo:cptembmaintheo1} hold. Moreover:
\begin{itemize}
\item[\bf(iii)] 
The bounded regular potential representation
$R(\A_{0}^{*})=\A_{0}^{*}\H{+}{1}$
holds with bounded linear regular potential operator
$\PotP_{\A_{0}^{*}}=\PotQ_{\A_{0}^{*},1}(\A_{0}^{*})_{\bot}^{-1}:R(\A_{0}^{*})\to\H{+}{1}$
satisfying $\A_{0}^{*}\PotP_{\A_{0}^{*}}=\id_{R(\A_{0}^{*})}$.
\item[\bf(iv)] 
$\widetilde\PotQ_{\A_{0}^{*},1}=\PotP_{\A_{0}^{*}}\A_{0}^{*}:D(\A_{0}^{*})\to\H{+}{1}$
and $\widetilde\PotN_{\A_{0}^{*}}=\id_{D(\A_{0}^{*})}-\widetilde\PotQ_{\A_{0}^{*},1}:D(\A_{0}^{*})\to N(\A_{0}^{*})$
are bounded linear regular decomposition operators with 
$\widetilde\PotQ_{\A_{0}^{*},1}+\widetilde\PotN_{\A_{0}^{*}}=\id_{D(\A_{0}^{*})}$
and the bounded regular decompositions
$$D(\A_{0}^{*})
=\H{+}{1}+\A_{1}^{*}\H{+}{2}
=\H{+}{1}+N(\A_{0}^{*})
=R(\widetilde\PotQ_{\A_{0}^{*},1})\dotplus N(\A_{0}^{*})
=R(\widetilde\PotQ_{\A_{0}^{*},1})\dotplus R(\widetilde\PotN_{\A_{0}^{*}})$$
hold. Moreover,
$R(\widetilde\PotQ_{\A_{0}^{*},1})=R(\PotP_{\A_{0}^{*}})$.
\item[\bf(iv')] 
$\A_{0}^{*}\widetilde\PotQ_{\A_{0}^{*},1}=\A_{0}^{*}\PotQ_{\A_{0}^{*},1}=\A_{0}^{*}$, i.e., $N(\A_{0}^{*})$ 
is invariant under $\PotQ_{\A_{0}^{*},1}$ and $\widetilde\PotQ_{\A_{0}^{*},1}$.
It holds $\widetilde\PotQ_{\A_{0}^{*},1}=\PotQ_{\A_{0}^{*},1}(\A_{0}^{*})_{\bot}^{-1}\A_{0}^{*}$.
Hence $\widetilde\PotQ_{\A_{0}^{*},1}|_{D((\A_{0}^{*})_{\bot})}=\PotQ_{\A_{0}^{*},1}|_{D((\A_{0}^{*})_{\bot})}$
and thus $\widetilde\PotQ_{\A_{0}^{*},1}$ may differ from $\PotQ_{\A_{0}^{*},1}$ only on $N(\A_{0}^{*})$.
\end{itemize}
\end{cor}

\subsection{FAT V: Long Hilbert Complexes}
\label{sec:FA5}

As a typical situation in 3D (extending literally to any dimension)
we have a \emph{long primal and dual Hilbert complex}
\begin{equation}
\label{lhcomplex}
\def\arrowlength{7ex}
\def\arrowdistance{.8}
\begin{tikzcd}[column sep=\arrowlength]
\H{}{-1} 
\arrow[r, rightarrow, shift left=\arrowdistance, "\A_{-1}"] 
\arrow[r, leftarrow, shift right=\arrowdistance, "\A_{-1}^{*}"']
& 
\H{}{0} 
\ar[r, rightarrow, shift left=\arrowdistance, "\A_{0}"] 
\ar[r, leftarrow, shift right=\arrowdistance, "\A_{0}^{*}"']
& 
\H{}{1}
\arrow[r, rightarrow, shift left=\arrowdistance, "\A_{1}"] 
\arrow[r, leftarrow, shift right=\arrowdistance, "\A_{1}^{*}"']
& 
\H{}{2}
\arrow[r, rightarrow, shift left=\arrowdistance, "\A_{2}"] 
\arrow[r, leftarrow, shift right=\arrowdistance, "\A_{2}^{*}"']
&
\H{}{3}
\arrow[r, rightarrow, shift left=\arrowdistance, "\A_{3}"] 
\arrow[r, leftarrow, shift right=\arrowdistance, "\A_{3}^{*}"']
&
\H{}{4}.
\end{tikzcd}
\end{equation}
Here, $\A_{0},\A_{1},\A_{2}$ are densely defined and closed (unbounded) linear operators
between three Hilbert spaces $\H{}{0},\H{}{1},\H{}{2}$
satisfying the complex properties 
$$R(\A_{0})\subset N(\A_{1}),\qquad
R(\A_{1})\subset N(\A_{2}).$$ 
$\A_{0}^{*},\A_{1}^{*},\A_{2}^{*}$
are the corresponding (Hilbert space) adjoints.
Moreover, $\A_{-1},\A_{4}$ and $\H{}{-1}$, $\H{}{4}$ are particular operators and kernels, 
respectively, i.e.,
$$\H{}{-1}:=N(\A_{0})=R(\A_{0}^{*})^{\bot_{\H{}{0}}},\qquad
\H{}{4}:=N(\A_{2}^{*})=R(\A_{2})^{\bot_{\H{}{3}}}$$
with corresponding bounded embeddings
$$\A_{-1}:=\iota_{N(\A_{0})}:N(\A_{0})\to\H{}{0},\qquad
\A_{3}^{*}:=\iota_{N(\A_{2}^{*})}:N(\A_{2}^{*})\to\H{}{3}.$$

\begin{rem}
\label{firstadjoints}
It holds $\A_{-1}^{*}=\iota_{N(\A_{0})}^{*}=\pi_{N(\A_{0})}:\H{}{0}\to N(\A_{0})$, 
the ``orthonormal projection'' onto the kernel of $\A_{0}$.
To see this, we note $\A_{-1}^{*}:\H{}{0}\to N(\A_{0})$ and for $x\in\H{}{0}$
and $\varphi\in N(\A_{0})$
$$\scp{\A_{-1}\varphi}{x}_{\H{}{0}}
=\scp{\varphi}{x}_{\H{}{0}}
=\scp{\pi_{N(\A_{0})}\varphi}{x}_{\H{}{0}}
=\scp{\varphi}{\pi_{N(\A_{0})}x}_{\H{}{0}}
=\scp{\varphi}{\pi_{N(\A_{0})}x}_{N(\A_{0})}.$$
Actually, the correct orthonormal projection onto $N(\A_{0})$
is then given by the self-adjoint bounded linear operator
$\A_{-1}\A_{-1}^{*}
=\iota_{N(\A_{0})}\iota_{N(\A_{0})}^{*}
=\pi_{N(\A_{0})}:\H{}{0}\to\H{}{0}$
with $R(\pi_{N(\A_{0})})=N(\A_{0})$.
Analogously, 
$\A_{3}=\iota_{N(\A_{2}^{*})}^{*}
=\pi_{N(\A_{2}^{*})}:\H{}{3}\to N(\A_{2}^{*})$
and 
$\A_{3}^{*}\A_{3}
=\iota_{N(\A_{2}^{*})}\iota_{N(\A_{2}^{*})}^{*}
=\pi_{N(\A_{2}^{*})}:\H{}{3}\to\H{}{3}$,
respectively, with $R(\pi_{N(\A_{2}^{*})})=N(\A_{2}^{*})$.
\end{rem}

The latter arguments show that the long primal and dual Hilbert complex \eqref{lhcomplex} reads
\begin{equation}
\label{lhcomplex2}
\def\arrowlength{15ex}
\def\arrowdistance{.8}
\begin{tikzcd}[column sep=\arrowlength]
N(\A_{0})
\arrow[r, rightarrow, shift left=\arrowdistance, "\A_{-1}=\iota_{N(\A_{0})}"] 
\arrow[r, leftarrow, shift right=\arrowdistance, "\A_{-1}^{*}=\pi_{N(\A_{0})}"']
& 
\H{}{0} 
\ar[r, rightarrow, shift left=\arrowdistance, "\A_{0}"] 
\ar[r, leftarrow, shift right=\arrowdistance, "\A_{0}^{*}"']
&
[-3em]
\H{}{1}
\arrow[r, rightarrow, shift left=\arrowdistance, "\A_{1}"] 
\arrow[r, leftarrow, shift right=\arrowdistance, "\A_{1}^{*}"']
& 
[-3em]
\H{}{2}
\arrow[r, rightarrow, shift left=\arrowdistance, "\A_{2}"] 
\arrow[r, leftarrow, shift right=\arrowdistance, "\A_{2}^{*}"']
&
[-3em]
\H{}{3}
\arrow[r, rightarrow, shift left=\arrowdistance, "\A_{3}=\pi_{N(\A_{2}^{*})}"] 
\arrow[r, leftarrow, shift right=\arrowdistance, "\A_{3}^{*}=\iota_{N(\A_{2}^{*})}"']
&
[-.5em]
N(\A_{2}^{*})
\end{tikzcd}
\end{equation}
with the complex properties 
\begin{align*}
R(\A_{-1})&=N(\A_{0}),
&
R(\A_{0})&\subset N(\A_{1}),
&
R(\A_{1})&\subset N(\A_{2}),
&
\ol{R(\A_{2})}&=N(\A_{3}),\\
\ol{R(\A_{0}^{*})}&=N(\A_{-1}^{*}),
&
R(\A_{1}^{*})&\subset N(\A_{0}^{*}),
&
R(\A_{2}^{*})&\subset N(\A_{1}^{*}),
&
R(\A_{3}^{*})&=N(\A_{2}^{*}).
\end{align*}

\begin{defi}
\label{defihilcom2}
The long Hilbert complex \eqref{lhcomplex2} is called 
\begin{itemize}
\item
closed, if $R(\A_{0})$, $R(\A_{1})$, and $R(\A_{2})$ are closed,
\item
compact, if the embeddings 
$D(\A_{1})\cap D(\A_{0}^{*})\incl\H{}{1}$ and $D(\A_{2})\cap D(\A_{1}^{*})\incl\H{}{1}$ 
as well as 
$$D(\A_{0})\cap D(\A_{-1}^{*})=D(\A_{0})\incl\H{}{0},\qquad
D(\A_{3})\cap D(\A_{2}^{*})=D(\A_{2}^{*})\incl\H{}{3}$$ 
are compact.
\end{itemize}
\end{defi}

\begin{rem}
\label{remhilcom3}
A compact long Hilbert complex is already closed.
\end{rem}

Note that the cohomology groups at both ends are trivial, i.e.,
\begin{align}
\label{cohomspecial}
\begin{aligned}
N_{-1,0}
&=N(\A_{0})\cap N(\A_{-1}^{*})=N(\A_{0})\cap N(\A_{0})^{\bot_{\H{}{0}}}=\{0\},\\
N_{2,3}
&=N(\A_{3})\cap N(\A_{2}^{*})=N(\A_{2}^{*})^{\bot_{\H{}{3}}}\cap N(\A_{2}^{*})=\{0\}.
\end{aligned}
\end{align}

\section{Notations and Preliminaries}

\subsection{Domains}
\label{sec:domains}

Throughout this paper, let $\om\subset\rdimom$, $\dimom\in\nat$, 
be a bounded strong Lipschitz domain
(locally $\om$ lies above a graph of some Lipschitz function).
Moreover, let the boundary $\ga$ of $\om$ be decomposed into two
strong Lipschitz subsets $\gat$ and $\gan:=\ga\setminus\ol{\gat}$
forming the interface $\ol{\gat}\cap\ol{\gan}$ for the mixed boundary conditions
(tangential and normal). See~\cite{BPS2016a,BPS2018a,BPS2019a}
for exact definitions. We call $(\om,\gat)$ a bounded strong Lipschitz pair.

We also recall the notion of an extendable strong Lipschitz domain
through either one of the boundary parts $\gat$ or $\gan$, 
see \cite[Section 5.4]{BPS2019a} and \cite[Section 7]{BPS2018a} for a definition.
Roughly speaking, a bounded strong Lipschitz pair $(\om,\gat)$ 
is called \emph{extendable}, if\\[-1mm]
\begin{minipage}[b]{105mm}
\begin{itemize}
\item
$\om$ and $\gat$ are \emph{topologically trivial}, and
\item
$\om$ can be \emph{extended} through $\gat$ to some 
topologically trivial and bounded strong Lipschitz domain $\widehat{\om}$,
resulting in a new topologically trivial and bounded strong Lipschitz domain 
$\widetilde{\om}=\text{\rm int}(\ol\om\cup\ol{\widehat\om})$,
cf.~the figure on the right or \cite[Figure 3.2]{BPS2019a}
for more details.
\end{itemize}
\end{minipage}
{\footnotesize		
\tdplotsetmaincoords{60}{140}
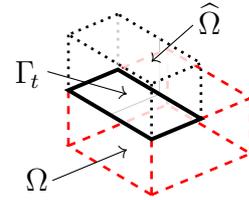
\begin{tikzpicture}[scale=.85,tdplot_main_coords]
	\draw [black] (0,0,0) -- (1,0,0);
	\draw [black] (0,0,0) -- (0,1,0);
	\draw [black] (0,0,0) -- (0,0,1);
	\draw [red,dashed,line width=1pt] (1,1,-1) -- (1,-1,-1) -- (-1,-1,-1) -- (-1,1,-1) -- cycle;	
	\draw [red,dashed,line width=1pt] (1,1,0) -- (1,1,-1);							
	\draw [red,dashed,line width=1pt] (1,-1,0) -- (1,-1,-1);							
	\draw [red,dashed,line width=1pt] (-1,-1,0) -- (-1,-1,-1);						
	\draw [red,dashed,line width=1pt] (-1,1,0) -- (-1,1,-1);							
	\fill [white,opacity=.8] (-1,-1,0) -- (1,-1,0) -- (1,1,0) -- (-1,1,0) -- cycle;				
	\draw [red,dashed,line width=1pt] (0,-1,0) -- (0,1,0) -- (-1,1,0) -- (-1,-1,0) -- cycle; 	
	\draw [black,line width=1.5pt] (1,1,0) -- (1,-1,0) -- (0,-1,0) -- (0,1,0) -- cycle;		
	\draw [black,dotted,line width=1pt] (0,-1,1) -- (0,-1,0);							
	\draw [black,dotted,line width=1pt] (1,-1,1) -- (1,-1,0);							
	\draw [black,dotted,line width=1pt] (0,1,0) -- (0,1,1);							
	\draw [black,dotted,line width=1pt] (1,1,1) -- (1,1,0);							
	\fill [white,opacity=.8] (0,-1,1) -- (1,-1,1) -- (1,1,1) -- (0,1,1) -- cycle;				
	\draw [black,dotted,line width=1pt] (0,-1,1) -- (1,-1,1) -- (1,1,1) -- (0,1,1) -- cycle;	
  \draw [color=black,->] (3,1,0) -- (1.5,1,0);
  \draw [color=black,->] (1,2,2.5) -- (1,1,1.4);
  \draw [color=black,->] (3.1,1,2) -- (1.5,1,1);
  \draw (1,2.4,2.8) node [] {\Large$\widehat{\om}$};
  \draw (3.4,1.1,0.2) node [] {\Large$\om$};
  \draw (3.5,1,2.1) node [] {\Large$\gat$};
\end{tikzpicture}
}

\begin{lem}
\label{extdom}
Any bounded strong Lipschitz pair $(\om,\gat)$ can be decomposed 
into a finite union of extendable bounded strong Lipschitz pairs $(\om_{\ell},\ga_{\!t,\ell})$
together with a subordinate partition of unity.
\end{lem}

\subsection{Sobolev Spaces of Scalar, Vector, and Tensor Fields}
\label{sec:defsobolevvecfields}

In this subsection let $d=3$.
The usual Lebesgue and Sobolev Hilbert spaces 
(of scalar, vector, or tensor valued fields)
are denoted by $\L{2}{}(\om)$, $\H{k}{}(\om)$, $\H{}{}(\rot,\om)$, $\H{}{}(\div,\om)$
for $k\in\z$, and by $\H{}{0}(\rot,\om)$ and $\H{}{0}(\div,\om)$ we indicate 
the spaces with vanishing $\rot$ and $\div$, respectively. 
Homogeneous boundary conditions 
for these standard differential operators $\grad$, $\rot$, and $\div$
are introduced in the \emph{strong sense} as closures of respective test fields from
$$\C{\infty}{\gat}(\om)
:=\big\{\phi|_{\om}\,:\,\phi\in\C{\infty}{}(\rdimom),\;
\mathrm{supp}\,\phi\;\mathrm{compact},\;
\mathrm{dist}(\mathrm{supp}\,\phi,\gat)>0\big\},$$
i.e., for $k\in\nat_{0}$
$$\H{k}{\gat}(\om):=\ol{\C{\infty}{\gat}(\om)}^{\H{k}{}(\om)},\quad
\H{}{\gat}(\rot,\om):=\ol{\C{\infty}{\gat}(\om)}^{\H{}{}(\rot,\om)},\quad
\H{}{\gat}(\div,\om):=\ol{\C{\infty}{\gat}(\om)}^{\H{}{}(\div,\om)},$$
and we have $\H{k}{\emptyset}(\om)=\H{k}{}(\om)$, $\H{}{\emptyset}(\rot,\om)=\H{}{}(\rot,\om)$,
and $\H{}{\emptyset}(\div,\om)=\H{}{}(\div,\om)$,
which are well known density results and incorporated into the notation by purpose.
Spaces with vanishing $\rot$ and $\div$
are again denoted by $\H{}{\gat,0}(\rot,\om)$ and $\H{}{\gat,0}(\div,\om)$, respectively. 
Note that for $k=0$ we have $\H{0}{\gat}(\om)=\L{2}{}(\om)$
and for the gradient we can also write $\H{1}{\gat}(\om)=\H{}{\gat}(\grad,\om)$.
Moreover, we introduce for $k\in\nat_{0}$ the non-standard Sobolev spaces
\begin{align*}
\H{k}{}(\rot,\om)
&:=\big\{v\in\H{k}{}(\om):\rot v\in\H{k}{}(\om)\big\},\\
\H{k}{\gat}(\rot,\om)
&:=\big\{v\in\H{k}{\gat}(\om)\cap\H{}{\gat}(\rot,\om):\rot v\in\H{k}{\gat}(\om)\big\},\\
\H{k}{}(\div,\om)
&:=\big\{v\in\H{k}{}(\om):\div v\in\H{k}{}(\om)\big\},\\
\H{k}{\gat}(\div,\om)
&:=\big\{v\in\H{k}{\gat}(\om)\cap\H{}{\gat}(\div,\om):\div v\in\H{k}{\gat}(\om)\big\}.
\end{align*}
We see $\H{k}{\emptyset}(\rot,\om)=\H{k}{}(\rot,\om)$
and for $k=0$ we have
$\H{0}{\emptyset}(\rot,\om)=\H{0}{}(\rot,\om)=\H{}{}(\rot,\om)$
and $\H{0}{\gat}(\rot,\om)=\H{}{\gat}(\rot,\om)$.
Note that for $\gat\neq\emptyset$ and $k\geq1$ it holds
$$\H{k}{\gat}(\rot,\om)
=\big\{v\in\H{k}{\gat}(\om):\rot v\in\H{k}{\gat}(\om)\big\},$$
but for $\gat\neq\emptyset$ and $k=0$ \big(as $\H{0}{\gat}(\om)=\L{2}{}(\om)$\big)
\begin{align*}
\H{0}{\gat}(\rot,\om)
&=\big\{v\in\H{0}{\gat}(\om)\cap\H{}{\gat}(\rot,\om):\rot v\in\H{0}{\gat}(\om)\big\}
=\H{}{\gat}(\rot,\om)\\
&\subsetneq\big\{v\in\H{0}{\gat}(\om):\rot v\in\H{0}{\gat}(\om)\big\}
=\H{0}{\emptyset}(\rot,\om)
=\H{}{}(\rot,\om).
\end{align*}
As before, 
$$\H{k}{\gat,0}(\rot,\om)
:=\H{k}{\gat}(\om)\cap\H{}{\gat,0}(\rot,\om)
=\H{k}{\gat}(\rot,\om)\cap\H{}{0}(\rot,\om)
=\big\{v\in\H{k}{\gat}(\rot,\om):\rot v=0\big\}.$$
The corresponding remarks and definitions extend 
to the $\H{k}{\gat}(\div,\om)$-spaces as well.

At this point, let us note that boundary conditions can also be defined 
in the \emph{weak sense} by
\begin{align*}
\bH{k}{\gat}(\om)
&:=\big\{u\in\H{k}{}(\om):
\scp{\p^{\alpha}u}{\phi}_{\L{2}{}(\om)}=(-1)^{|\alpha|}\scp{u}{\p^{\alpha}\phi}_{\L{2}{}(\om)}
\quad\forall\,\phi\in\C{\infty}{\gan}(\om)\quad\forall\,|\alpha|\leq k\big\},\\
\bH{}{\gat}(\rot,\om)
&:=\big\{v\in\H{}{}(\rot,\om):
\scp{\rot v}{\psi}_{\L{2}{}(\om)}=\scp{v}{\rot\psi}_{\L{2}{}(\om)}
\quad\forall\,\psi\in\C{\infty}{\gan}(\om)\big\},\\
\bH{}{\gat}(\div,\om)
&:=\big\{v\in\H{}{}(\div,\om):
\scp{\div v}{\phi}_{\L{2}{}(\om)}=-\scp{v}{\grad\phi}_{\L{2}{}(\om)}
\quad\forall\,\phi\in\C{\infty}{\gan}(\om)\big\}.
\end{align*}
Analogously, we define the Sobolev spaces $\bH{k}{\gat}(\rot,\om)$, $\bH{k}{\gat}(\div,\om)$
and $\bH{k}{\gat,0}(\rot,\om)$, $\bH{k}{\gat,0}(\div,\om)$
using the respective Sobolev spaces with weak boundary conditions.
Note that ``\emph{strong $\subset$ weak}'' holds, e.g., 
$$\H{k}{\gat}(\om)\subset\bH{k}{\gat}(\om),\qquad
\H{}{\gat}(\rot,\om)\subset\bH{}{\gat}(\rot,\om),\qquad
\H{k}{\gat}(\div,\om)\subset\bH{k}{\gat}(\div,\om),$$
and that the complex properties hold in both the strong and the weak case, e.g.,
$$\grad\H{k+1}{\gat}(\om)\subset\H{k}{\gat,0}(\rot,\om),\qquad
\rot\bH{k}{\gat}(\rot,\om)\subset\bH{k}{\gat,0}(\div,\om),$$
which follows immediately by the definitions.
The next lemma shows that indeed ``\emph{strong $=$ weak}'' holds.

\begin{lem}[{\cite[Theorem 4.5]{BPS2016a}}]
\label{lem:wsbcderhamvec}
The Sobolev spaces defined by weak and strong boundary conditions coincide, e.g.,
$\bH{k}{\gat}(\om)=\H{k}{\gat}(\om)$, $\bH{}{\gat}(\rot,\om)=\H{}{\gat}(\rot,\om)$, 
and $\bH{k}{\gat}(\div,\om)=\H{k}{\gat}(\div,\om)$, cf.~Lemma \ref{lem:wsbcderham}.
\end{lem}

Finally, we introduce the cohomology space of Dirichlet/Neumann fields (generalised harmonic fields)
$$\Harm{}{\gat,\gan,\eps}(\om)
:=\H{}{\gat,0}(\rot,\om)\cap\eps^{-1}\H{}{\gan,0}(\div,\om).$$
The classical Dirichlet and Neumann fields are then given by
$\Harm{}{\ga,\emptyset,\eps}(\om)$ anf $\Harm{}{\emptyset,\ga,\eps}(\om)$, respectively.
Here, $\eps:\L{2}{}(\om)\to\L{2}{}(\om)$ is a symmetric and positive 
topological isomorphism (symmetric and positive bijective bounded linear operator),
which introduces a new inner product
$$\scp{\,\cdot\,}{\,\cdot\,}_{\L{2}{\eps}(\om)}
:=\scp{\eps\,\cdot\,}{\,\cdot\,}_{\L{2}{}(\om)},$$
where $\L{2}{\eps}(\om):=\L{2}{}(\om)$ (as linear space)
equipped with the inner product $\scp{\,\cdot\,}{\,\cdot\,}_{\L{2}{\eps}(\om)}$.
Such \emph{weights} $\eps$ shall be called \emph{admissible}
and a typical example is given by a 
symmetric, $\L{\infty}{}$-bounded, and uniformly positive definite tensor (matrix) field
$\eps:\om\to\reals^{3\times3}$.

\subsection{Sobolev Spaces of Differential Forms}
\label{sec:defsobolevdforms}

For spaces of differential forms we follow the same rationale.
Instead of the differential operators $\grad$, $\rot$, and $\div$
we now have only the exterior derivative $\ed$ and the co-derivative $\cd=\pm*\ed*$,
given by $\ed$ and the Hodge star operator $*$.
The standard Lebesgue and Sobolev Hilbert spaces 
are denoted by $\L{q,2}{}(\om)$, $\H{q,k}{}(\om)$, $\H{q}{}(\ed,\om)$, $\H{q}{}(\cd,\om)$
for $k\in\z$, and by $\H{q}{0}(\ed,\om)$ and $\H{q}{0}(\cd,\om)$ we indicate 
the spaces with vanishing $\ed$ and $\cd$, respectively. 
Here $q\in\z$ marks the rank of the respective differential forms.
As before, homogeneous boundary conditions 
for $\ed$ and $\cd$ are introduced in the \emph{strong sense} as closures of respective test forms from
$$\C{q,\infty}{\gat}(\om)
:=\big\{\Phi|_{\om}\,:\,\Phi\in\C{q,\infty}{}(\rdimom),\;
\mathrm{supp}\,\Phi\;\mathrm{compact},\;
\mathrm{dist}(\mathrm{supp}\,\Phi,\gat)>0\big\},$$
i.e., for $k\in\nat_{0}$
$$\H{q,k}{\gat}(\om):=\ol{\C{q,\infty}{\gat}(\om)}^{\H{q,k}{}(\om)},\quad
\H{q}{\gat}(\ed,\om):=\ol{\C{q,\infty}{\gat}(\om)}^{\H{q}{}(\ed,\om)},\quad
\H{q}{\gat}(\cd,\om):=\ol{\C{q,\infty}{\gat}(\om)}^{\H{q}{}(\cd,\om)},$$
and we have $\H{q,k}{\emptyset}(\om)=\H{q,k}{}(\om)$, $\H{q}{\emptyset}(\ed,\om)=\H{q}{}(\ed,\om)$,
and $\H{q}{\emptyset}(\cd,\om)=\H{q}{}(\cd,\om)$,
which are well known density results and incorporated into the notation by purpose.
Spaces with vanishing $\ed$ and $\cd$
are again denoted by $\H{q}{\gat,0}(\ed,\om)$ and $\H{q}{\gat,0}(\cd,\om)$, respectively. 
Note that for $k=0$ we have $\H{q,0}{\gat}(\om)=\L{q,2}{}(\om)$
and for $q=0$ we can also write $\H{0,1}{\gat}(\om)=\H{0}{\gat}(\ed,\om)\cong\H{\dimom}{\gat}(\cd,\om)$.
Moreover, we introduce for $k\in\nat_{0}$ the non-standard Sobolev spaces of $q$-forms
\begin{align*}
\H{q,k}{}(\ed,\om)
&:=\big\{E\in\H{q,k}{}(\om):\ed E\in\H{q+1,k}{}(\om)\big\},\\
\H{q,k}{\gat}(\ed,\om)
&:=\big\{E\in\H{q,k}{\gat}(\om)\cap\H{q}{\gat}(\ed,\om):\ed E\in\H{q+1,k}{\gat}(\om)\big\},\\
\H{q,k}{}(\cd,\om)
&:=\big\{E\in\H{q,k}{}(\om):\cd E\in\H{q-1,k}{}(\om)\big\},\\
\H{q,k}{\gat}(\cd,\om)
&:=\big\{E\in\H{q,k}{\gat}(\om)\cap\H{q}{\gat}(\cd,\om):\cd E\in\H{q-1,k}{\gat}(\om)\big\}.
\end{align*}
We see $\H{q,k}{\emptyset}(\ed,\om)=\H{q,k}{}(\ed,\om)$
and for $k=0$ we have
$\H{q,0}{\emptyset}(\ed,\om)=\H{q,0}{}(\ed,\om)=\H{q}{}(\ed,\om)$
and $\H{q,0}{\gat}(\ed,\om)=\H{q}{\gat}(\ed,\om)$.
Note that for $\gat\neq\emptyset$ and $k\geq1$ it holds
$$\H{q,k}{\gat}(\ed,\om)
=\big\{E\in\H{q,k}{\gat}(\om):\ed E\in\H{q+1,k}{\gat}(\om)\big\},$$
but for $\gat\neq\emptyset$ and $k=0$ \big(as $\H{q,0}{\gat}(\om)=\L{q,2}{}(\om)$\big)
\begin{align*}
\H{q,0}{\gat}(\ed,\om)
&=\big\{E\in\H{q,0}{\gat}(\om)\cap\H{q}{\gat}(\ed,\om):\ed E\in\H{q+1,0}{\gat}(\om)\big\}
=\H{q}{\gat}(\ed,\om)\\
&\subsetneq\big\{E\in\H{q,0}{\gat}(\om):\ed E\in\H{q+1,0}{\gat}(\om)\big\}
=\H{q,0}{\emptyset}(\ed,\om)
=\H{q}{}(\ed,\om).
\end{align*}
As before, 
$$\H{q,k}{\gat,0}(\ed,\om)
:=\H{q,k}{\gat}(\om)\cap\H{q}{\gat,0}(\ed,\om)
=\H{q,k}{\gat}(\ed,\om)\cap\H{q}{0}(\ed,\om)
=\big\{E\in\H{q,k}{\gat}(\ed,\om):\ed E=0\big\}.$$
The corresponding remarks hold 
for the $\H{q,k}{\gat}(\cd,\om)$-spaces as well.

Again, let us note that boundary conditions can also be defined 
in the \emph{weak sense} by
\begin{align*}
\bH{q,k}{\gat}(\om)
&:=\big\{E\in\H{q,k}{}(\om):
\scp{\p^{\alpha}E}{\Phi}_{\L{q,2}{}(\om)}=(-1)^{|\alpha|}\scp{E}{\p^{\alpha}\Phi}_{\L{q,2}{}(\om)}
\;\forall\,\Phi\in\C{q,\infty}{\gan}(\om)\;\forall\,|\alpha|\leq k\big\},\\
\bH{q}{\gat}(\ed,\om)
&:=\big\{E\in\H{q}{}(\ed,\om):
\scp{\ed E}{\Phi}_{\L{q+1,2}{}(\om)}=-\scp{E}{\cd\Phi}_{\L{q,2}{}(\om)}
\quad\forall\,\Phi\in\C{q+1,\infty}{\gan}(\om)\big\},\\
\bH{q}{\gat}(\cd,\om)
&:=\big\{E\in\H{q}{}(\cd,\om):
\scp{\cd E}{\Phi}_{\L{q-1,2}{}(\om)}=-\scp{E}{\ed\Phi}_{\L{q,2}{}(\om)}
\quad\forall\,\Phi\in\C{q-1,\infty}{\gan}(\om)\big\}.
\end{align*}
Analogously, we define the Sobolev spaces $\bH{q,k}{\gat}(\ed,\om)$, $\bH{q,k}{\gat}(\cd,\om)$
and $\bH{q,k}{\gat,0}(\ed,\om)$, $\bH{q,k}{\gat,0}(\cd,\om)$
using the respective Sobolev spaces with weak boundary conditions.
Note that ``\emph{strong $\subset$ weak}'' holds, e.g., 
$$\H{q,k}{\gat}(\om)\subset\bH{q,k}{\gat}(\om),\qquad
\H{q}{\gat}(\ed,\om)\subset\bH{q}{\gat}(\ed,\om),\qquad
\H{q,k}{\gat}(\cd,\om)\subset\bH{q,k}{\gat}(\cd,\om),$$
and that the complex properties hold in both the strong and the weak case, e.g.,
$$\ed\H{q,k}{\gat}(\ed,\om)\subset\H{q+1,k}{\gat,0}(\ed,\om),\qquad
\cd\bH{q,k}{\gat}(\cd,\om)\subset\bH{q-1,k}{\gat,0}(\cd,\om),$$
which follows immediately by the definitions.
The next lemma shows that indeed ``\emph{strong $=$ weak}'' holds.

\begin{lem}[{\cite[Theorem 4.7]{BPS2019a}}]
\label{lem:wsbcderham}
The Sobolev spaces defined by weak and strong boundary conditions coincide, e.g.,
$\bH{q,k}{\gat}(\om)=\H{q,k}{\gat}(\om)$, $\bH{q}{\gat}(\ed,\om)=\H{q}{\gat}(\ed,\om)$, 
and $\bH{q,k}{\gat}(\cd,\om)=\H{q,k}{\gat}(\cd,\om)$.
\end{lem}

For convenience, a self-contained proof of Lemma \ref{lem:wsbcderham}
(and hence also of Lemma \ref{lem:wsbcderhamvec}) 
is given as a part of Lemma \ref{lem:highorderregdecoedpbcLip}, 
cf.~Lemma \ref{highorderregpotedpbc} and Lemma \ref{highorderregdecoedpbc}.

\begin{lem}[Schwarz' lemma]
\label{lem:edpalpha}
Let $|\alpha|\leq k$.
\begin{itemize}
\item[\bf(i)]
For $E\in\H{q,k}{\gat}(\ed,\om)$ it holds 
$\p^{\alpha}E\in\H{q,0}{\gat}(\ed,\om)$ and $\ed\p^{\alpha}E=\p^{\alpha}\ed E$.
\item[\bf(ii)]
For $H\in\H{q,k}{\gat}(\cd,\om)$ it holds
$\p^{\alpha}H\in\H{q,0}{\gat}(\cd,\om)$ and $\cd\p^{\alpha}H=\p^{\alpha}\cd H$.
\end{itemize}
\end{lem}

\begin{proof}
(i) can be seen as follows: For $\Phi\in\C{q+1,\infty}{\gan}(\om)$ we have 
\begin{align*}
\scp{\p^{\alpha}E}{\cd\Phi}_{\L{q,2}{}(\om)}
&=(-1)^{|\alpha|}\scp{E}{\cd\p^{\alpha}\Phi}_{\L{q,2}{}(\om)}\\
&=(-1)^{|\alpha|+1}\scp{\ed E}{\p^{\alpha}\Phi}_{\L{q+1,2}{}(\om)}
=-\scp{\p^{\alpha}\ed E}{\Phi}_{\L{q+1,2}{}(\om)}
\end{align*}
as $E\in\H{q,k}{\gat}(\om)\cap\H{q,0}{\gat}(\ed,\om)$ and
$\ed E\in\H{q+1,k}{\gat}(\om)$.
Hence
$\p^{\alpha}E\in\bH{q,0}{\gat}(\ed,\om)=\H{q,0}{\gat}(\ed,\om)$
by Lemma \ref{lem:wsbcderham} and $\ed\p^{\alpha}E=\p^{\alpha}\ed E$.
(ii) follows analogously or by the Hodge $\star$-operator.
\end{proof}

Finally, we introduce the cohomology space of Dirichlet/Neumann forms (generalised harmonic forms)
\begin{align}
\label{dfdirneudef}
\Harm{q}{\gat,\gan,\eps}(\om)
:=\H{q}{\gat,0}(\ed,\om)\cap\eps^{-1}\H{q}{\gan,0}(\cd,\om).
\end{align}
The classical Dirichlet and Neumann fields are then given by
$\Harm{q}{\ga,\emptyset,\eps}(\om)$ anf $\Harm{q}{\emptyset,\ga,\eps}(\om)$, respectively.
Here, $\eps=\eps_{q}:\L{q,2}{}(\om)\to\L{q,2}{}(\om)$ is a symmetric and positive 
topological isomorphism (symmetric and positive bijective bounded linear operator),
which introduces a new inner product
$$\scp{\,\cdot\,}{\,\cdot\,}_{\L{q,2}{\eps}(\om)}
:=\scp{\eps\,\cdot\,}{\,\cdot\,}_{\L{q,2}{}(\om)},$$
where $\L{q,2}{\eps}(\om):=\L{q,2}{}(\om)$ (as linear space)
equipped with the inner product $\scp{\,\cdot\,}{\,\cdot\,}_{\L{q,2}{\eps}(\om)}$.
Such \emph{weights} $\eps$ shall be called \emph{admissible}
and a typical example is given by a 
symmetric, $\L{\infty}{}$-bounded, and uniformly positive definite tensor (matrix) field
$\eps:\om\to\reals^{\binom{N}{q}\times\binom{N}{q}}$.

\subsection{Some Useful and Important Results}

In \cite{HLZ2012a} the existence of a crucial universal extension operator 
for the Sobolev spaces $\H{q,k}{}(\ed,\om)$ has been shown,
which is based on the universal extension operator from Stein's book \cite{S1970a}.

\begin{lem}[universal Stein extension operator {\cite[Theorem 3.6]{HLZ2012a}}, cf.~{\cite[Lemma 2.15]{BPS2019a}}]
\label{stein}
Let $\om\subset\rdimom$ be a bounded strong Lipschitz domain.
For all $k\in\nat_{0}$ and all $q$ there exists 
a (universal) bounded linear extension operator
$$\E=\E^{q,k}:\H{q,k}{}(\ed,\om)\to\H{q,k}{}(\ed,\rdimom).$$
More precisely, there exists $c>0$ such that for all $E\in\H{q,k}{}(\ed,\om)$
it holds $\E E\in\H{q,k}{}(\ed,\rdimom)$ and $\E E=E$ in $\om$ as well as
$\norm{\E E}_{\H{q,k}{}(\ed,\rdimom)}\leq c\norm{E}_{\H{q,k}{}(\ed,\om)}$.
Furthermore, $\E$ can be chosen such that $\E E$ has fixed compact support in $\rdimom$
for all $E\in\H{q,k}{}(\ed,\om)$.
\end{lem}

From \cite[Theorem 5.2]{BPS2019a} we have the following Helmholtz decompositions.

\begin{lem}[Helmholtz decompositions]
\label{helmom}
Let $\om\subset\rdimom$ be a bounded strong Lipschitz domain.
For all $q$ the orthonormal Helmholtz decompositions 
\begin{align*}
\L{q,2}{\eps}(\om)
&=\ed\H{q-1,0}{\gat}(\ed,\om)
\oplus_{\L{q,2}{\eps}(\om)}
\eps^{-1}\H{q,0}{\gan,0}(\cd,\om)\\
&=\H{q,0}{\gat,0}(\ed,\om)
\oplus_{\L{q,2}{\eps}(\om)}
\eps^{-1}\cd\H{q+1,0}{\gan}(\cd,\om)\\
&=\ed\H{q-1,0}{\gat}(\ed,\om)
\oplus_{\L{q,2}{\eps}(\om)}
\Harm{q}{\gat,\gan,\eps}(\om)
\oplus_{\L{q,2}{\eps}(\om)}
\eps^{-1}\cd\H{q+1,0}{\gan}(\cd,\om)
\end{align*}
hold. In particular, the ranges
\begin{align*}
\ed\H{q-1,0}{\gat}(\ed,\om)
&=\H{q,0}{\gat,0}(\ed,\om)
\cap\Harm{q}{\gat,\gan,\eps}(\om)^{\bot_{\L{q,2}{\eps}(\om)}},\\
\cd\H{q+1,0}{\gan}(\cd,\om)
&=\H{q,0}{\gan,0}(\cd,\om)
\cap\Harm{q}{\gat,\gan,\eps}(\om)^{\bot_{\L{q,2}{}(\om)}}
\end{align*}
are closed subspaces of $\L{q,2}{\eps}(\om)$
and the potentials can be chosen such that they depend continuously on the data.
\end{lem}

Note that Lemma \ref{helmom} even holds for 
bounded weak Lipschitz domains $\om\subset\rdimom$.
From \cite{P1982a}, cf.~\cite[Lemma 2.19]{BPS2019a}, 
we have the following Helmholtz decompositions for the special case $\om=\rdimom$.

\begin{lem}[Helmholtz decompositions in the whole space]
\label{helmrN}
For all $q$
\begin{align*}
\L{q,2}{}(\rdimom)&=\H{q}{0}(\ed,\rdimom)\oplus_{\L{q,2}{}(\rdimom)}\H{q}{0}(\cd,\rdimom),\\
\H{q}{}(\ed,\rdimom)&=\H{q}{0}(\ed,\rdimom)\oplus_{\L{q,2}{}(\rdimom)}\big(\H{q}{}(\ed,\rdimom)\cap\H{q}{0}(\cd,\rdimom)\big).
\end{align*}
Let $\pi_{q,\rdimom}:\L{q,2}{}(\rdimom)\to\H{q}{0}(\cd,\rdimom)$ denote the orthonormal projector onto
$\H{q}{0}(\cd,\rdimom)$. Then for all $E\in\H{q}{}(\ed,\rdimom)$ it holds
$\pi_{q,\rdimom}E\in\H{q}{}(\ed,\rdimom)\cap\H{q}{0}(\cd,\rdimom)$
and $\ed\pi_{q,\rdimom}E=\ed E$ as well as
$\norm{\pi_{q,\rdimom}E}_{\H{q}{}(\ed,\rdimom)}\leq\norm{E}_{\H{q}{}(\ed,\rdimom)}$.
\end{lem}

From \cite[Lemma 4.2(i)]{KP2010a}, cf.~\cite[Lemma 2.20]{BPS2019a}, 
we have the following regularity result.

\begin{lem}[regularity in the whole space]
\label{regrN}
For $k\in\nat_{0}$ and all $q$ it holds 
$$\big\{E\in\H{q}{}(\ed,\rdimom)\cap\H{q}{}(\cd,\rdimom):
\ed E\in\H{q+1,k}{}(\rdimom)\,\wedge\,\cd E\in\H{q-1,k}{}(\rdimom)\big\}
=\H{q,k+1}{}(\rdimom).$$
More precisely, $E\in\H{q}{}(\ed,\rdimom)\cap\H{q}{}(\cd,\rdimom)$
with $\ed E\in\H{q+1,k}{}(\rdimom)$ and $\cd E\in\H{q-1,k}{}(\rdimom)$,
if and only if $E\in\H{q,k+1}{}(\rdimom)$ and
$$\frac{1}{c}\norm{E}_{\H{q,k+1}{}(\rdimom)}
\leq\norm{E}_{\L{q,2}{}(\rdimom)}
+\norm{\ed E}_{\H{q+1,k}{}(\rdimom)}
+\norm{\cd E}_{\H{q-1,k}{}(\rdimom)}
\leq c\norm{E}_{\H{q,k+1}{}(\rdimom)}$$
with some $c>0$ independent of $E$.
\end{lem}

In \cite[Lemma 3.1]{BPS2019a}, see also 
\cite{BPS2016a,BPS2018a} for more details,
the following lemma about the existence of regular potentials without boundary conditions has been shown.

\begin{lem}[regular potential for $\ed$ without boundary condition]
\label{regpoted}
Let $\om\subset\rdimom$ be a bounded strong Lipschitz domain.
For all $q\in\{1,\dots,\dimom\}$ there exists a bounded linear potential operator
$$\PotP_{\ed,\emptyset}^{q,0}:\H{q,0}{\emptyset,0}(\ed,\om)
\cap\Harm{q}{\emptyset,\ga,\id}(\om)^{\bot_{\L{q,2}{}(\om)}}
\To\H{q-1,1}{0}(\cd,\rdimom),$$ 
such that 
$\ed\PotP_{\ed,\emptyset}^{q,0}
=\id|_{\H{q,0}{\emptyset,0}(\ed,\om)
\cap\Harm{q}{\emptyset,\ga,\id}(\om)^{\bot_{\L{q,2}{}(\om)}}}$, 
i.e., for all 
$E\in\H{q,0}{\emptyset,0}(\ed,\om)
\cap\Harm{q}{\emptyset,\ga,\id}(\om)^{\bot_{\L{q,2}{}(\om)}}$ 
$$\ed\PotP_{\ed,\emptyset}^{q,0}E=E\quad\text{in }\om.$$
In particular, 
$$\H{q,0}{\emptyset,0}(\ed,\om)
\cap\Harm{q}{\emptyset,\ga,\id}(\om)^{\bot_{\L{q,2}{}(\om)}}
=\ed\H{q-1,0}{\emptyset}(\cd,\om)
=\ed\H{q-1,1}{\emptyset}(\om)
=\ed\H{q-1,1}{\emptyset,0}(\cd,\om)$$ 
and the potentials can be chosen such that they depend continuously on the data.
Especially, these are closed subspaces of $\L{q,2}{}(\om)$ 
and $\PotP_{\ed,\emptyset}^{q,0}$ is a right inverse to $\ed$.
\end{lem}

\section{De Rham Complex}
\label{sec:derham}

In this section we shall apply the FA-ToolBox from
Section \ref{sec:FA} to the de Rham complex.

\subsection{Zero Order De Rham Complex}
\label{sec:derhamlow}

Let the exterior derivatives
be realised as densely defined (unbounded) linear operators
$$\mr{\ed}_{\gat}^{q}:D(\mr{\ed}_{\gat}^{q})\subset\L{q,2}{}(\om)\to\L{q+1,2}{}(\om);
\,E\mapsto\ed E,\qquad
D(\mr{\ed}_{\gat}^{q}):=\C{q,\infty}{\gat}(\om),\qquad
q=0,\dots,\dimom-1,$$
satisfying the complex properties
$$\mr{\ed}_{\gat}^{q}\mr{\ed}_{\gat}^{q-1}\subset0.$$
Then the closures $\ed_{\gat}^{q}:=\ol{\mr{\ed}_{\gat}^{q}}$
and Hilbert space adjoints $(\ed_{\gat}^{q})^{*}=(\mr{\ed}_{\gat}^{q})^{*}$ are given by
$$\ed_{\gat}^{q}:D(\ed_{\gat}^{q})\subset\L{q,2}{}(\om)\to\L{q+1,2}{}(\om);
\,E\mapsto\ed E,\qquad
D(\ed_{\gat}^{q})=\H{q,0}{\gat}(\ed,\om),$$
and
$$(\ed_{\gat}^{q})^{*}=-\cd_{\gan}^{q+1}:D(\cd_{\gan}^{q+1})\subset\L{q+1,2}{}(\om)\to\L{q,2}{}(\om);
\,H\mapsto-\cd H,\qquad
D(\cd_{\gan}^{q+1})=\H{q+1,0}{\gan}(\cd,\om),$$
where indeed $D(\cd_{\gan}^{q+1})=\H{q+1,0}{\gan}(\cd,\om)$ 
holds by Lemma \ref{lem:wsbcderham}, cf.~\cite[Section 5.2]{BPS2019a},
(weak and strong boundary conditions coincide).

\begin{rem}
\label{rem:weakeqstrongderham}
Note that by definition the adjoints are given by 
$$(\ed_{\gat}^{q})^{*}=(\mr{\ed}_{\gat}^{q})^{*}=-\bs{\cd}_{\gan}^{q+1}:
D(\bs{\cd}_{\gan}^{q+1})\subset\L{q+1,2}{}(\om)\to\L{q,2}{}(\om);
\,H\mapsto-\cd H,$$
with $D(\bs{\cd}_{\gan}^{q+1})=\bH{q+1,0}{\gan}(\cd,\om)$.
Lemma \ref{lem:wsbcderham} (weak and strong boundary conditions coincide)
shows indeed 
$D(\bs{\cd}_{\gan}^{q+1})
=\bH{q+1,0}{\gan}(\cd,\om)
=\H{q+1,0}{\gan}(\cd,\om)
=D(\cd_{\gan}^{q+1})$, i.e.,
$\bs{\cd}_{\gan}^{q+1}=\cd_{\gan}^{q+1}$.
\end{rem}

By definition the densely defined and closed (unbounded) linear operators
$$\A_{q}:=\ed_{\gat}^{q},\qquad
\A_{q}^{*}=-\cd_{\gan}^{q+1},\qquad
q=0,\dots,\dimom-1,$$
define dual pairs 
$\big(\ed_{\gat}^{q},(\ed_{\gat}^{q})^{*}\big)=(\ed_{\gat}^{q},-\cd_{\gan}^{q+1})$.
Remark \ref{remhilcom} and Remark \ref{remhilcomclosure} show the complex properties 
$R(\ed_{\gat}^{q-1})\subset N(\ed_{\gat}^{q})$ and
$R(\cd_{\gan}^{q+1})\subset N(\cd_{\gan}^{q})$, i.e.,
the complex properties 
$$\ed_{\gat}^{q}\ed_{\gat}^{q-1}\subset0,\qquad
\cd_{\gan}^{q}\cd_{\gan}^{q+1}\subset0.$$
Note that with $\A_{0}=\ed_{\gat}^{0}$ 
and $\A_{\dimom-1}^{*}=(\ed_{\gat}^{\dimom-1})^{*}=-\cd_{\gan}^{\dimom}$ as well as
$$\A_{-1}:=\iota_{N(\A_{0})},\quad
\A_{-1}^{*}=\pi_{N(\A_{0})},\qquad
\A_{\dimom}^{*}:=\iota_{N(\A_{\dimom-1}^{*})},\quad
\A_{\dimom}=\pi_{N(\A_{\dimom-1}^{*})}$$
(actually, $\A_{-1}\A_{-1}^{*}=\pi_{N(\A_{0})}$ and 
$\A_{\dimom}^{*}\A_{\dimom}=\pi_{N(\A_{\dimom-1}^{*})}$,
cf.~Remark \ref{firstadjoints})
we have
$$N(\A_{0})=N(\ed_{\gat}^{0})=\reals_{\gat},\qquad
N(\A_{\dimom-1}^{*})=N(\cd_{\gan}^{\dimom})=*\reals_{\gan},\qquad
\reals_{\Sigma}
:=\begin{cases}
\reals&\text{if }\Sigma=\emptyset,\\
\{0\}&\text{otherwise},
\end{cases}$$
and that the long (here even longer) primal and dual de Rham Hilbert complex \eqref{lhcomplex2} reads
\begin{equation*}
\def\arrowlength{15ex}
\def\arrowdistance{.8}
\begin{tikzcd}[column sep=\arrowlength]
\reals_{\gat}
\arrow[r, rightarrow, shift left=\arrowdistance, "\iota_{\reals_{\gat}}"] 
\arrow[r, leftarrow, shift right=\arrowdistance, "\pi_{\reals_{\gat}}"']
& 
[-3em]
\L{0,2}{}(\om) 
\ar[r, rightarrow, shift left=\arrowdistance, "\ed_{\gat}^{0}"] 
\ar[r, leftarrow, shift right=\arrowdistance, "-\cd_{\gan}^{1}"']
&
[-3em]
\L{1,2}{}(\om) 
\ar[r, rightarrow, shift left=\arrowdistance, "\ed_{\gat}^{1}"] 
\ar[r, leftarrow, shift right=\arrowdistance, "-\cd_{\gan}^{2}"']
& 
[-3em]
\L{2,2}{}(\om) 
\arrow[r, rightarrow, shift left=\arrowdistance, "\cdots"] 
\arrow[r, leftarrow, shift right=\arrowdistance, "\cdots"']
& 
[-4em]
\cdots
\end{tikzcd}
\end{equation*}
\begin{equation}
\label{derhamcomplex1}
\def\arrowlength{15ex}
\def\arrowdistance{.8}
\begin{tikzcd}[column sep=\arrowlength]
\cdots
\arrow[r, rightarrow, shift left=\arrowdistance, "\cdots"] 
\arrow[r, leftarrow, shift right=\arrowdistance, "\cdots"']
& 
[-4em]
\L{q-1,2}{}(\om) 
\ar[r, rightarrow, shift left=\arrowdistance, "\ed_{\gat}^{q-1}"] 
\ar[r, leftarrow, shift right=\arrowdistance, "-\cd_{\gan}^{q}"']
&
[-2em]
\L{q,2}{}(\om) 
\ar[r, rightarrow, shift left=\arrowdistance, "\ed_{\gat}^{q}"] 
\ar[r, leftarrow, shift right=\arrowdistance, "-\cd_{\gan}^{q+1}"']
& 
[-2em]
\L{q+1,2}{}(\om) 
\arrow[r, rightarrow, shift left=\arrowdistance, "\cdots"] 
\arrow[r, leftarrow, shift right=\arrowdistance, "\cdots"']
& 
[-4em]
\cdots
\end{tikzcd}
\end{equation}
\begin{equation*}
\def\arrowlength{15ex}
\def\arrowdistance{.8}
\begin{tikzcd}[column sep=\arrowlength]
\cdots
\arrow[r, rightarrow, shift left=\arrowdistance, "\cdots"] 
\arrow[r, leftarrow, shift right=\arrowdistance, "\cdots"']
& 
[-4em]
\L{\dimom-2,2}{}(\om) 
\ar[r, rightarrow, shift left=\arrowdistance, "\ed_{\gat}^{\dimom-2}"] 
\ar[r, leftarrow, shift right=\arrowdistance, "-\cd_{\gan}^{\dimom-1}"']
&
[-2em]
\L{\dimom-1,2}{}(\om) 
\ar[r, rightarrow, shift left=\arrowdistance, "\ed_{\gat}^{\dimom-1}"] 
\ar[r, leftarrow, shift right=\arrowdistance, "-\cd_{\gan}^{\dimom}"']
&
[-2em]
\L{\dimom,2}{}(\om) 
\arrow[r, rightarrow, shift left=\arrowdistance, "\pi_{*\reals_{\gan}}"] 
\arrow[r, leftarrow, shift right=\arrowdistance, "\iota_{*\reals_{\gan}}"']
&
[-2em]
*\reals_{\gan}
\end{tikzcd}
\end{equation*}
with the complex properties
$$R(\ed_{\gat}^{q-1})\subset N(\ed_{\gat}^{q}),\qquad
R(\cd_{\gan}^{q+1})\subset N(\cd_{\gan}^{q}),\qquad
q=1,\dots,\dimom-1,$$
and
\begin{align*}
R(\iota_{\reals_{\gat}})&=N(\ed_{\gat}^{0})=\reals_{\gat},
&
\ol{R(\ed_{\gat}^{\dimom-1})}&=N(\pi_{*\reals_{\gan}})=(*\reals_{\gan})^{\bot_{\L{\dimom,2}{}(\om)}},\\
\ol{R(\cd_{\gan}^{1})}&=N(\pi_{\reals_{\gat}})=(\reals_{\gat})^{\bot_{\L{0,2}{}(\om)}},
&
R(\iota_{*\reals_{\gan}})&=N(\cd_{\gan}^{\dimom})=*\reals_{\gan}.
\end{align*}
We emphasise that the definition of the Dirichlet/Neumann forms \eqref{dfdirneudef}
is consistent with the definition of the cohomology groups
$N_{q-1,q}=N(\A_{q})\cap N(\A_{q-1}^{*})$ as long as $1\leq q\leq\dimom-1$. 
For $q=0$ and $q=\dimom$ we have the deviations
\begin{align*}
\{0\}
=N_{-1,0}
&\subset N(\A_{0})
=\H{0}{\gat,0}(\ed,\om)
=\Harm{0}{\gat,\gan,\eps}(\om)
=\reals_{\gat},\\
\{0\}
=N_{\dimom-1,\dimom}
&\subset N(\A_{\dimom-1}^{*})
=\eps^{-1}\H{\dimom}{\gan,0}(\cd,\om)
=\Harm{\dimom}{\gat,\gan,\eps}(\om)
=\eps^{-1}*\reals_{\gan},
\end{align*}
cf. \eqref{cohomspecial},
which is intended and usefull for latter formulations.

\subsection{Higher Order De Rham Complex}
\label{sec:derhamhigh}

Similar to \eqref{derhamcomplex1}
we can also investigate the higher Sobolev order primal de Rham complex
\begin{equation*}
\def\arrowlength{5ex}
\def\arrowdistance{0}
\begin{tikzcd}[column sep=\arrowlength]
\cdots 
\arrow[r, rightarrow, shift left=\arrowdistance, "\cdots"] 
& 
\H{q-1,k}{\gat}(\om)
\ar[r, rightarrow, shift left=\arrowdistance, "\ed^{q-1,k}_{\gat}"] 
& 
[1em]
\H{q,k}{\gat}(\om)
\arrow[r, rightarrow, shift left=\arrowdistance, "\ed^{q,k}_{\gat}"] 
& 
\H{q+1,k}{\gat}(\om)
\arrow[r, rightarrow, shift left=\arrowdistance, "\cdots"] 
&
\cdots
\end{tikzcd}
\end{equation*}
together with its \emph{formal} adjoint, the higher Sobolev order dual de Rham complex
\begin{equation*}
\def\arrowlength{5ex}
\def\arrowdistance{0}
\begin{tikzcd}[column sep=\arrowlength]
\cdots 
\arrow[r, rightarrow, shift left=\arrowdistance, "\cdots"] 
& 
\H{q-1,k}{\gan}(\om)
\ar[r, leftarrow, shift left=\arrowdistance, "-\cd^{q,k}_{\gan}"] 
& 
[1em]
\H{q,k}{\gan}(\om)
\arrow[r, leftarrow, shift left=\arrowdistance, "-\cd^{q+1,k}_{\gan}"] 
& 
\H{q+1,k}{\gan}(\om)
\arrow[r, leftarrow, shift left=\arrowdistance, "\cdots"] 
&
\cdots.
\end{tikzcd}
\end{equation*}
More precisely, we consider
\begin{align*}
\ed_{\gat}^{q,k}:D(\ed_{\gat}^{q,k})\subset\H{q,k}{\gat}(\om)&\to\H{q+1,k}{\gat}(\om);
\,E\mapsto\ed E,
&
D(\ed_{\gat}^{q,k})&:=\H{q,k}{\gat}(\ed,\om),
\intertext{with formal adjoints}
-\cd_{\gan}^{q+1,k}:D(\cd_{\gan}^{q+1,k})\subset\H{q+1,k}{\gan}(\om)&\to\H{q,k}{\gan}(\om);
\,H\mapsto-\cd H,
&
D(\cd_{\gan}^{q+1,k})&:=\H{q+1,k}{\gan}(\cd,\om).
\end{align*}
Note that $\ed_{\gat}^{q,k}$ and $\cd_{\gan}^{q+1,k}$ are densely defined and closed as, e.g.,
$$\C{q,\infty}{\gat}(\om)
\subset\H{q,k}{\gat}(\ed,\om)
\subset\H{q,k}{\gat}(\om)
=\ol{\C{q,\infty}{\gat}(\om)}^{\H{q,k}{}(\om)},$$
and that indeed the complex properties 
$R(\ed_{\gat}^{q-1,k})\subset N(\ed_{\gat}^{q,k})$ and
$R(\cd_{\gan}^{q+1,k})\subset N(\cd_{\gan}^{q,k})$ 
hold.

Unfortunately, the respectively adjoints
\begin{align*}
(\ed_{\gat}^{q,k})^{*}:D\big((\ed_{\gat}^{q,k})^{*}\big)\subset\H{q+1,k}{\gat}(\om)
&\to\H{q,k}{\gat}(\om),\\
-(\cd_{\gan}^{q+1,k})^{*}:D\big((\cd_{\gan}^{q+1,k})^{*}\big)\subset\H{q,k}{\gan}(\om)
&\to\H{q+1,k}{\gan}(\om)
\end{align*}
are hard to compute.
Therefore, only some parts of the FA-ToolBox from Section \ref{sec:FA} 
apply to the higher order de Rham complex,
and a few results have to proved in a less general setting.

Note that for $E\in D(\ed_{\gat}^{q,k})$ and for 
$H\in D(\cd_{\ga}^{q+1,k})\subset\H{q+1,k}{\gat}(\cd,\om)\cap\H{q+1,k}{\gan}(\cd,\om)$
we have
\begin{align*}
\scp{\ed E}{H}_{\H{q+1,k}{\gat}(\om)}
&=\sum_{|\alpha|\leq k}^{}\scp{\p^{\alpha}\ed E}{\p^{\alpha}H}_{\L{q+1,2}{}(\om)}
=-\sum_{|\alpha|\leq k}^{}\scp{\p^{\alpha}E}{\p^{\alpha}\cd H}_{\L{q,2}{}(\om)}
=-\scp{E}{\cd H}_{\H{q,k}{\gat}(\om)}
\end{align*}
by Lemma \ref{lem:edpalpha}.

\begin{rem}[higher order adjoints for the de Rham complex]
\label{rem:adjhighorderderham}
It holds $-\cd_{\ga}^{q+1,k}\subset(\ed_{\gat}^{q,k})^{*}$
and $-\ed_{\ga}^{q-1,k}\subset(\cd_{\gan}^{q,k})^{*}$, i.e.,
\begin{align*}
D(\cd_{\ga}^{q+1,k})&\subset D\big((\ed_{\gat}^{q,k})^{*}\big)
&&\text{and}&
(\ed_{\gat}^{q,k})^{*}|_{D(\cd_{\ga}^{q+1,k})}&=-\cd_{\ga}^{q+1,k},\\
D(\ed_{\ga}^{q-1,k})&\subset D\big((\cd_{\gan}^{q,k})^{*}\big)
&&\text{and}&
(\cd_{\gan}^{q,k})^{*}|_{D(\ed_{\ga}^{q-1,k})}&=-\ed_{\ga}^{q-1,k}.
\end{align*}
Note that, here, we identify $-\cd_{\ga}^{q+1,k}$ with 
$-\cd_{\ga}^{q+1,k}:D(\cd_{\ga}^{q+1,k})\subset\H{q+1,k}{\gat}(\om)\to\H{q,k}{\gat}(\om)$,
which is not densely defined. The same holds for $-\ed_{\ga}^{q-1,k}$.
\end{rem}

\subsection{Regular Potentials Without Boundary Conditions}

The next lemma generalises Lemma \ref{regpoted} and
ensures the existence of regular $\H{q,k}{\emptyset}(\om)$-potentials 
without boundary conditions for strong Lipschitz domains. 

\begin{lem}[regular potential for $\ed$ without boundary condition]
\label{highorderregpoted}
Let $\om\subset\rdimom$ be a bounded strong Lipschitz domain and let $k\geq0$
and $q\in\{1,\dots,\dimom\}$.
Then there exists a bounded linear regular potential operator
$$\PotP_{\ed,\emptyset}^{q,k}:\H{q,k}{\emptyset,0}(\ed,\om)\cap\Harm{q}{\emptyset,\ga,\id}(\om)^{\bot_{\L{q,2}{}(\om)}}
\To\H{q-1,k+1}{0}(\cd,\rdimom),$$ 
such that $\ed\PotP_{\ed,\emptyset}^{q,k}=\id|_{\H{q,k}{\emptyset,0}(\ed,\om)\cap\Harm{q}{\emptyset,\ga,\id}(\om)^{\bot_{\L{q,2}{}(\om)}}}$, i.e.,
for all $E\in\H{q,k}{\emptyset,0}(\ed,\om)\cap\Harm{q}{\emptyset,\ga,\id}(\om)^{\bot_{\L{q,2}{}(\om)}}$ 
$$\ed\PotP_{\ed,\emptyset}^{q,k}E=E\quad\text{in }\om.$$
In particular, the bounded regular potential representations
$$R(\ed^{q-1,k}_{\emptyset})
=\H{q,k}{\emptyset,0}(\ed,\om)\cap\Harm{q}{\emptyset,\ga,\id}(\om)^{\bot_{\L{q,2}{}(\om)}}
=\ed\H{q-1,k}{\emptyset}(\ed,\om)
=\ed\H{q-1,k+1}{\emptyset}(\om)
=\ed\H{q-1,k+1}{\emptyset,0}(\cd,\om)$$ 
hold and the potentials can be chosen such that they depend continuously on the data.
Especially, these are closed subspaces of $\H{q,k}{\emptyset}(\om)=\H{q,k}{}(\om)$ 
and $\PotP_{\ed,\emptyset}^{q,k}$ is a right inverse to $\ed$.
By a simple cut-off technique $\PotP_{\ed,\emptyset}^{q,k}$ may be modified to 
$$\PotP_{\ed,\emptyset}^{q,k}:
\H{q,k}{\emptyset,0}(\ed,\om)\cap\Harm{q}{\emptyset,\ga,\id}(\om)^{\bot_{\L{q,2}{}(\om)}}
\To\H{q-1,k+1}{}(\cd,\rdimom)$$ 
such that $\PotP_{\ed,\emptyset}^{q,k}E$ has a fixed compact support in $\rdimom$ 
for all $E\in\H{q,k}{\emptyset,0}(\ed,\om)\cap\Harm{q}{\emptyset,\ga,\id}(\om)^{\bot_{\L{q,2}{}(\om)}}$.
\end{lem}

\begin{proof}
Lemma \ref{regpoted} shows the assertions for $k=0$ and $\PotP_{\ed,\emptyset}^{q,0}$.
Moreover, the inclusions
$$\ed\H{q-1,k+1}{\emptyset,0}(\cd,\om)
\subset\ed\H{q-1,k+1}{\emptyset}(\om)
\subset\ed\H{q-1,k}{\emptyset}(\ed,\om)
\subset\H{q,k}{\emptyset,0}(\ed,\om)\cap\Harm{q}{\emptyset,\ga,\id}(\om)^{\bot_{\L{q,2}{}(\om)}}$$
hold. Suppose 
$E\in\H{q,k}{\emptyset,0}(\ed,\om)\cap\Harm{q}{\emptyset,\ga,\id}(\om)^{\bot_{\L{q,2}{}(\om)}}$,
$k\geq1$.
Then $E\in\H{q,k-1}{\emptyset,0}(\ed,\om)\cap\Harm{q}{\emptyset,\ga,\id}(\om)^{\bot_{\L{q,2}{}(\om)}}$.
By assumption of induction there exists $\PotP_{\ed,\emptyset}^{q,k-1}E\in\H{q-1,k}{\emptyset}(\om)$ 
with $\ed\PotP_{\ed,\emptyset}^{q,k-1}E=E$ in $\om$ and 
$$\norm{\PotP_{\ed,\emptyset}^{q,k-1}E}_{\H{q-1,k}{}(\om)}
\leq c\norm{E}_{\H{q,k-1}{}(\om)}.$$
Hence $\PotP_{\ed,\emptyset}^{q,k-1}E\in\H{q-1,k}{\emptyset}(\ed,\om)$
and by Lemma \ref{stein} we have
$\E\PotP_{\ed,\emptyset}^{q,k-1}E\in\H{q-1,k}{}(\ed,\rdimom)$ with compact support and 
$$\norm{\E\PotP_{\ed,\emptyset}^{q,k-1}E}_{\H{q-1,k}{}(\ed,\rdimom)}
\leq c\norm{\PotP_{\ed,\emptyset}^{q,k-1}E}_{\H{q-1,k}{}(\ed,\om)}
\leq c\big(\norm{\PotP_{\ed,\emptyset}^{q,k-1}E}_{\H{q-1,k}{}(\om)}
+\norm{E}_{\H{q,k}{}(\om)}\big).$$
Using Lemma \ref{helmrN} we obtain a uniquely determined
$$\PotP_{\ed,\emptyset}^{q,k}E:=\pi_{q-1,\rdimom}\E\PotP_{\ed,\emptyset}^{q,k-1}E
\in\H{q-1,0}{}(\ed,\rdimom)\cap\H{q-1,0}{0}(\cd,\rdimom)$$
with $\ed\PotP_{\ed,\emptyset}^{q,k}E=\ed\E\PotP_{\ed,\emptyset}^{q,k-1}E\in\H{q,k}{}(\rdimom)$.
Lemma \ref{regrN} shows 
$\PotP_{\ed,\emptyset}^{q,k}E\in\H{q-1,k+1}{}(\rdimom)$ with
\begin{align*}
\norm{\PotP_{\ed,\emptyset}^{q,k}E}_{\H{q-1,k+1}{}(\rdimom)}
\leq c\big(\norm{\PotP_{\ed,\emptyset}^{q,k}E}_{\L{q-1,2}{}(\rdimom)}
+\norm{\ed\E\PotP_{\ed,\emptyset}^{q,k-1}E}_{\H{q,k}{}(\rdimom)}\big)
\leq c\norm{\E\PotP_{\ed,\emptyset}^{q,k-1}E}_{\H{q-1,k}{}(\ed,\rdimom)}.
\end{align*}
Finally, $\PotP_{\ed,\emptyset}^{q,k}E\in\H{q-1,k+1}{0}(\cd,\rdimom)$ meets our needs as it holds
$\norm{\PotP_{\ed,\emptyset}^{q,k}E}_{\H{q-1,k+1}{}(\rdimom)}
\leq c\norm{E}_{\H{q,k}{}(\om)}$
and $\ed\PotP_{\ed,\emptyset}^{q,k}E=\ed\E\PotP_{\ed,\emptyset}^{q,k-1}E=\ed\PotP_{\ed,\emptyset}^{q,k-1}E=E$ in $\om$.
\end{proof}

By Hodge $\star$-duality we get a corresponding result for the $\cd$-operator, cf.~Lemma \ref{highorderregpotcd}.

\subsection{Regular Potentials and Decompositions With Boundary Conditions}

Now we construct regular $\H{q,k}{}(\om)$-potentials with (partial) boundary conditions. 
Recall the definitions of Section \ref{sec:domains} for the different
assumptions on the domain $\om\subset\rdimom$.

\subsubsection{Extendable Domains}

\begin{lem}[regular potential for $\ed$ with partial boundary condition for extendable domains]
\label{highorderregpotedpbc}
Let $(\om,\gat)$ be an extendable bounded strong Lipschitz pair
and let $1\leq q\leq\dimom-1$ as well as $k\geq0$. 
Then there exists a bounded linear regular potential operator
$$\PotP_{\ed,\gat}^{q,k}:\bH{q,k}{\gat,0}(\ed,\om)
\To\H{q-1,k+1}{}(\rdimom)\cap\H{q-1,k+1}{\gat}(\om),$$
such that $\ed\PotP_{\ed,\gat}^{q,k}=\id|_{\bH{q,k}{\gat,0}(\ed,\om)}$, i.e.,
for all $E\in\bH{q,k}{\gat,0}(\ed,\om)$
$$\ed\PotP_{\ed,\gat}^{q,k}E=E\quad\text{in }\om.$$
In particular, the bounded regular potential representation
$$\bH{q,k}{\gat,0}(\ed,\om)
=\H{q,k}{\gat,0}(\ed,\om)
=\ed\H{q-1,k+1}{\gat}(\om)
=\ed\H{q-1,k}{\gat}(\ed,\om)
=R(\ed^{q-1,k}_{\gat})$$
holds and the potentials can be chosen such that they depend continuously on the data.
Especially, these spaces are closed subspaces of 
$\H{q,k}{\emptyset}(\om)=\H{q,k}{}(\om)$ and $\PotP_{\ed,\gat}^{q,k}$ is a right inverse to $\ed$.
Without loss of generality, 
$\PotP_{\ed,\gat}^{q,k}$ maps to forms with a fixed compact support in $\rdimom$.

The results extend literally to the case $q=\dimom$ if $\gat\neq\ga$
and the case $q=0$ is trivial since $\bH{0,k}{\gat,0}(\ed,\om)=\reals_{\gat}$. 
In the special case $q=\dimom$ and $\gat=\ga$ the results still remain valid if
$$\bH{\dimom,k}{\ga,0}(\ed,\om)
=\bH{\dimom,k}{\ga}(\om),\qquad
\H{\dimom,k}{\ga,0}(\ed,\om)
=\H{\dimom,k}{\ga}(\om)$$
are replaced by the slightly smaller spaces
$$\bH{\dimom,k}{\ga}(\om)\cap(*\,\reals)^{\bot_{\L{\dimom,2}{}(\om)}},\qquad
\H{\dimom,k}{\ga}(\om)\cap(*\,\reals)^{\bot_{\L{\dimom,2}{}(\om)}},$$
respectively.
\end{lem}

\begin{proof}
The case $\gat=\emptyset$ is done in Lemma \ref{highorderregpoted}. 
For $\gat\neq\emptyset$, suppose $E\in\bH{q,k}{\gat,0}(\ed,\om)$ and define 
$\widetilde{E}\in\L{q,2}{}(\widetilde{\om})$
as extension of $E$ by zero to $\widehat{\om}$.
By definition we see $\widetilde{E}\in\H{q,k}{\emptyset,0}(\ed,\widetilde{\om})$.
Since $\widetilde{\om}$ is bounded, strong Lipschitz,
and topologically trivial, in particular $\Harm{q}{\emptyset,\widetilde{\ga},\id}(\widetilde{\om})=\{0\}$,
Lemma \ref{highorderregpoted} yields a regular potential 
$\PotP_{\ed,\emptyset}^{q,k}\widetilde{E}\in\H{q-1,k+1}{0}(\cd,\rdimom)\subset\H{q-1,k+1}{}(\rdimom)$
with $\ed\PotP_{\ed,\emptyset}^{q,k}\widetilde{E}=\widetilde{E}$ in $\widetilde{\om}$ and 
$$c\norm{\PotP_{\ed,\emptyset}^{q,k}\widetilde{E}}_{\H{q-1,k+1}{}(\rdimom)}
\leq\norm{\widetilde{E}}_{\H{q,k}{}(\widetilde{\om})}
=\norm{E}_{\H{q,k}{}(\om)}.$$
Let $\iota_{\widehat{\om}}$ denote the restriction to $\widehat{\om}$.
Then $\iota_{\widehat{\om}}\PotP_{\ed,\emptyset}^{q,k}\widetilde{E}\in\H{q-1,k+1}{\emptyset}(\widehat{\om})$ 
and $\ed\iota_{\widehat{\om}}\PotP_{\ed,\emptyset}^{q,k}\widetilde{E}=\iota_{\widehat{\om}}\widetilde{E}=0$ in $\widehat{\om}$,
i.e., $\iota_{\widehat{\om}}\PotP_{\ed,\emptyset}^{q,k}\widetilde{E}\in\H{q-1,k+1}{\emptyset,0}(\ed,\widehat\om)$.
Using Lemma \ref{highorderregpoted} again, this time in $\widehat{\om}$,
which is bounded, strong Lipschitz, and topologically trivial as well,
we obtain $\PotP_{\ed,\emptyset}^{q-1,k+1}\iota_{\widehat{\om}}\PotP_{\ed,\emptyset}^{q,k}\widetilde{E}\in\H{q-2,k+2}{}(\rdimom)$ with 
$\ed\PotP_{\ed,\emptyset}^{q-1,k+1}\iota_{\widehat{\om}}\PotP_{\ed,\emptyset}^{q,k}\widetilde{E}
=\iota_{\widehat{\om}}\PotP_{\ed,\emptyset}^{q,k}\widetilde{E}$ in $\widehat{\om}$ and 
$$\norm{\PotP_{\ed,\emptyset}^{q-1,k+1}\iota_{\widehat{\om}}\PotP_{\ed,\emptyset}^{q,k}\widetilde{E}}_{\H{q-2,k+2}{}(\rdimom)}
\leq c\norm{\PotP_{\ed,\emptyset}^{q,k}\widetilde{E}}_{\H{q-1,k+1}{}(\widehat{\om})}.$$
Then
$$\Abb{\PotP_{\ed,\gat}^{q,k}}{\H{q,k}{\gat,0}(\ed,\om)}{\H{q-1,k+1}{}(\rdimom)}
{E}{\PotP_{\ed,\emptyset}^{q,k}\widetilde{E}-\ed(\PotP_{\ed,\emptyset}^{q-1,k+1}\iota_{\widehat{\om}}\PotP_{\ed,\emptyset}^{q,k}\widetilde{E})}$$
is linear and bounded since
$$\norm{\PotP_{\ed,\gat}^{q,k}E}_{\H{q-1,k+1}{}(\rdimom)}
\leq\norm{\PotP_{\ed,\emptyset}^{q,k}\widetilde{E}}_{\H{q-1,k+1}{}(\rdimom)}
+\norm{\PotP_{\ed,\emptyset}^{q-1,k+1}\iota_{\widehat{\om}}\PotP_{\ed,\emptyset}^{q,k}\widetilde{E}}_{\H{q-2,k+2}{}(\rdimom)}
\leq c\norm{E}_{\H{q,k}{}(\om)}.$$
Since $\PotP_{\ed,\gat}^{q,k}E=0$ in $\widehat{\om}$, we obtain by standard arguments for Sobolev spaces
$\PotP_{\ed,\gat}^{q,k}E\in\H{q-1,k+1}{\gat}(\om)$, 
cf.~\cite[Lemma 2.14]{BPS2019a}
\big(weak and strong boundary conditions coincide for $\H{q,k}{}(\om)$\big).
Moreover, it holds $\ed\PotP_{\ed,\gat}^{q,k}E=\ed\PotP_{\ed,\emptyset}^{q,k}\widetilde{E}=\widetilde{E}$ in $\widetilde\om$,
in particular, $\ed\PotP_{\ed,\gat}^{q,k}E=E$ in $\om$.
Finally,
$$\ed\H{q-1,k+1}{\gat}(\om)
\subset\ed\H{q-1,k}{\gat}(\ed,\om)
\subset\H{q,k}{\gat,0}(\ed,\om)
\subset\bH{q,k}{\gat,0}(\ed,\om)
\subset\ed\H{q-1,k+1}{\gat}(\om),$$
completing the proof of the main part.
In the special case $q=\dimom$ and $\gat=\ga$ 
we also have to take care of the constant $\dimom$-forms in $*\,\reals$.
\end{proof}

Hodge $\star$-duality yields a corresponding result 
for the $\cd$-operator, cf.~Lemma \ref{highorderregpotdecocdpbc} (i).

\begin{lem}[regular decompositions for $\ed$ with partial boundary condition for extendable domains]
\label{highorderregdecoedpbc}
Let $(\om,\gat)$ be an extendable bounded strong Lipschitz pair and let $k\geq0$.
Then the bounded regular decompositions
\begin{align*}
\bH{q,k}{\gat}(\ed,\om)
=\H{q,k}{\gat}(\ed,\om)
&=\H{q,k+1}{\gat}(\om)+\ed\H{q-1,k+1}{\gat}(\om)\\
&=\PotQ_{\ed,\gat,1}^{q,k}\H{q,k}{\gat}(\ed,\om)
\dotplus\ed\PotQ_{\ed,\gat,0}^{q,k}\H{q,k}{\gat}(\ed,\om)\\
&=\PotQ_{\ed,\gat,1}^{q,k}\H{q,k}{\gat}(\ed,\om)
\dotplus\ed\H{q-1,k+1}{\gat}(\om)\\
&=\PotQ_{\ed,\gat,1}^{q,k}\H{q,k}{\gat}(\ed,\om)
\dotplus\H{q,k}{\gat,0}(\ed,\om)
\end{align*}
hold with bounded linear regular decomposition operators
\begin{align*}
\PotQ_{\ed,\gat,1}^{q,k}:=\PotP_{\ed,\gat}^{q+1,k}\ed:\H{q,k}{\gat}(\ed,\om)
&\to\H{q,k+1}{\gat}(\om),\\
\PotQ_{\ed,\gat,0}^{q,k}:=\PotP_{\ed,\gat}^{q,k}(1-\PotP_{\ed,\gat}^{q+1,k}\ed):\H{q,k}{\gat}(\ed,\om)
&\to\H{q-1,k+1}{\gat}(\om).
\end{align*}
More precisely,
it holds $\bH{q,k}{\gat}(\ed,\om)=\H{q,k}{\gat}(\ed,\om)$ and
$\PotQ_{\ed,\gat,1}^{q,k}+\ed\PotQ_{\ed,\gat,0}^{q,k}
=\id|_{\H{q,k}{\gat}(\ed,\om)}$, i.e.,
$$E=\PotQ_{\ed,\gat,1}^{q,k}E+\ed\PotQ_{\ed,\gat,0}^{q,k}E
\in\H{q,k+1}{\gat}(\om)+\ed\H{q-1,k+1}{\gat}(\om)$$
for all $E\in\H{q,k}{\gat}(\ed,\om)$.
Moreover, it holds $\ed\PotQ_{\ed,\gat,1}^{q,k}=\ed^{q,k}_{\gat}$
and thus $\H{q,k}{\gat,0}(\ed,\om)$ is invariant under $\PotQ_{\ed,\gat,1}^{q,k}$.
Note that for the ranges  
$\PotQ_{\ed,\gat,1}^{q,k}\H{q,k}{\gat}(\ed,\om)
=R(\PotQ_{\ed,\gat,1}^{q,k})
=R(\PotP_{\ed,\gat}^{q+1,k})$
as well as
$\PotQ_{\ed,\gat,0}^{q,k}\H{q,k}{\gat}(\ed,\om)
=R(\PotQ_{\ed,\gat,0}^{q,k})
=R(\PotP_{\ed,\gat}^{q,k})$ hold.
\end{lem}

The proof follows by Corollary \ref{cor:regpotdeco2} and Lemma \ref{highorderregpotedpbc}.
For convenience, we give a self-contained proof here. 

\begin{proof}
Let $E\in\bH{q,k}{\gat}(\ed,\om)$. 
Then $\ed E\in\bH{q+1,k}{\gat,0}(\ed,\om)$
and we see 
$\PotP_{\ed,\gat}^{q+1,k}\ed E\in\H{q,k+1}{\gat}(\om)$ 
with $\ed\PotP_{\ed,\gat}^{q+1,k}\ed E=\ed E$ by Lemma \ref{highorderregpotedpbc}.
Thus $E-\PotP_{\ed,\gat}^{q+1,k}\ed E\in\bH{q,k}{\gat,0}(\ed,\om)=\ed\H{q-1,k+1}{\gat}(\om)$ 
and $\PotP_{\ed,\gat}^{q,k}(E-\PotP_{\ed,\gat}^{q+1,k}\ed E)\in\H{q-1,k+1}{\gat}(\om)$
with $\ed\PotP_{\ed,\gat}^{q,k}(E-\PotP_{\ed,\gat}^{q+1,k}\ed E)=E-\PotP_{\ed,\gat}^{q+1,k}\ed E$
by Lemma \ref{highorderregpotedpbc}. This yields
$$E=\PotP_{\ed,\gat}^{q+1,k}\ed E+\ed\PotP_{\ed,\gat}^{q,k}(1-\PotP_{\ed,\gat}^{q+1,k}\ed)E
\in\H{q,k+1}{\gat}(\om)+\ed\H{q-1,k+1}{\gat}(\om)
\subset\H{q,k}{\gat}(\ed,\om),$$
which proves the regular decompositions and 
also the assertions about the bounded linear regular decomposition operators.
To show the directness of the sums, let
$H=\PotP_{\ed,\gat}^{q+1,k}\ed E\in\H{q,0}{\gat,0}(\ed,\om)$
with some $E\in\H{q,k}{\gat}(\ed,\om)$.
Then $0=\ed H=\ed E$ as $\ed E\in\H{q+1,k}{\gat,0}(\ed,\om)$
and thus $H=0$.
\end{proof}

Again, by Hodge $\star$-duality we get a corresponding result 
for the $\cd$-operator, cf.~Lemma \ref{highorderregpotdecocdpbc} (ii).

\subsubsection{General Lipschitz Domains}

\begin{lem}[regular decompositions for $\ed$ with partial boundary condition]
\label{lem:highorderregdecoedpbcLip}
Let $(\om,\gat)$ be a bounded strong Lipschitz pair and let $k\geq0$. 
Then the bounded regular decompositions
$$\bH{q,k}{\gat}(\ed,\om)
=\H{q,k}{\gat}(\ed,\om)
=\H{q,k+1}{\gat}(\om)
+\ed\H{q-1,k+1}{\gat}(\om)$$
hold with bounded linear regular decomposition operators
\begin{align*}
\PotQ_{\ed,\gat,1}^{q,k}:\H{q,k}{\gat}(\ed,\om)\to\H{q,k+1}{\gat}(\om),\qquad
\PotQ_{\ed,\gat,0}^{q,k}:\H{q,k}{\gat}(\ed,\om)\to\H{q-1,k+1}{\gat}(\om)
\end{align*}
satisfying $\PotQ_{\ed,\gat,1}^{q,k}+\ed\PotQ_{\ed,\gat,0}^{q,k}=\id_{\H{q,k}{\gat}(\ed,\om)}$.
In particular, weak and strong boundary conditions coincide.
Moreover, it holds $\ed\PotQ_{\ed,\gat,1}^{q,k}=\ed^{q,k}_{\gat}$
and thus $\H{q,k}{\gat,0}(\ed,\om)$ is invariant under $\PotQ_{\ed,\gat,1}^{q,k}$.
\end{lem}

\begin{proof}
According to Lemma \ref{extdom},
let us introduce a partition of unity $(U_{\ell},\chi_{\ell})$
as in \cite[Section 4.2]{BPS2019a}
or \cite[Section 4.2]{BPS2018a},
such that $(\om_{\ell},\widehat\ga_{t,\ell})$ 
is an extendable bounded strong Lipschitz pair
for all $l=1,\dots,L_{+}$. Using the notations from \cite{BPS2019a} we have 
$$\om_{\ell}=\om\cap U_{\ell},\qquad
\Sigma_{\ell}=\p\om_{\ell}\setminus\ga,\qquad
\ga_{\!t,\ell}=\gat\cap U_{\ell},\qquad
\widehat\ga_{\!t,\ell}=\mathrm{int}(\ga_{\!t,\ell}\cup\ol{\Sigma}_{\ell}).$$
Maybe $U_{0}=\om$ has to be replaced by more neighbourhoods $U_{-L_{-}},\dots,U_{0}$ 
to ensure that all pairs $(\om_{\ell},\widehat\ga_{\!t,\ell})$, $\ell=-L_{-},\dots,L_{+}$,
are topologically trivial. Note that for all ``inner'' indices $\ell=-L_{-},\dots,0$
we have $\om_{\ell}=U_{\ell}$ as well as 
$\widehat\ga_{\!t,\ell}=\Sigma_{\ell}=\p\om_{\ell}=\p U_{\ell}$.

Then for $E\in\bH{q,k}{\gat}(\ed,\om)$ we have
$\chi_{\ell}E\in\bH{q,k}{\widehat\ga_{\!t,\ell}}(\ed,\om_{\ell})
=\H{q,k}{\widehat\ga_{\!t,\ell}}(\ed,\om_{\ell})$
for all $\ell$
and Lemma \ref{highorderregdecoedpbc} shows the bounded regular decompositions
\begin{align*}
\chi_{\ell}E=E_{\ell}+\ed H_{\ell}
\in\H{q,k+1}{\widehat\ga_{\!t,\ell}}(\om_{\ell})
+\ed\H{q-1,k+1}{\widehat\ga_{\!t,\ell}}(\om_{\ell})
\end{align*}
with $E_{\ell}$ and $H_{\ell}$ depending continuously on $\chi_{\ell}E$.
Extending $E_{\ell}$ and $H_{\ell}$ by zero to $\om$
yields forms $\widetilde{E}_{\ell}\in\H{q,k+1}{\gat}(\om)$ 
and $\widetilde{H}_{\ell}\in\H{q-1,k+1}{\gat}(\om)$ as well as the representation 
$$\bH{q,k}{\gat}(\ed,\om)\ni E
=\sum_{\ell}\chi_{k}E=\sum_{\ell}\widetilde{E}_{\ell}+\ed\sum_{\ell}\widetilde{H}_{\ell}
\in\H{q,k+1}{\gat}(\om)+\ed\H{q-1,k+1}{\gat}(\om)
\subset\H{q,k}{\gat}(\ed,\om).$$
As all operations have been linear and continuous we set
\begin{align*}
\PotQ_{\ed,\gat,1}^{q,k}E:=\sum_{\ell}\widetilde{E}_{\ell}\in\H{q,k+1}{\gat}(\om),\qquad
\PotQ_{\ed,\gat,0}^{q,k}E:=\sum_{\ell}\widetilde{H}_{\ell}\in\H{q-1,k+1}{\gat}(\om),
\end{align*}
and obtain the assertions.
\end{proof}

Hodge $\star$-duality shows a corresponding result 
for the $\cd$-operator, cf.~Lemma \ref{app:lem:highorderregdecoedpbcLip}.

\begin{cor}[regular decompositions for $\ed$ with partial boundary condition]
\label{cor:highorderregdecoedpbcLip}
Let $(\om,\gat)$ be a bounded strong Lipschitz pair and let $k\geq0$. 
Then the regular potential representations
\begin{align*}
R(\ed^{q-1,k}_{\gat})
=\ed\H{q-1,k}{\gat}(\ed,\om)
=\ed\H{q-1,k+1}{\gat}(\om)
=\H{q,k}{\gat,0}(\ed,\om)
\cap\Harm{q}{\gat,\gan,\eps}(\om)^{\bot_{\L{q,2}{\eps}(\om)}},\\
R(\cd^{q+1,k}_{\gan})
=\cd\H{q+1,k}{\gan}(\cd,\om)
=\cd\H{q+1,k+1}{\gan}(\om)
=\H{q,k}{\gan,0}(\cd,\om)
\cap\Harm{q}{\gat,\gan,\eps}(\om)^{\bot_{\L{q,2}{}(\om)}}
\end{align*}
hold. In particular, these spaces are closed subspaces of 
$\H{q,k}{\emptyset}(\om)=\H{q,k}{}(\om)$.
\end{cor}

\begin{proof}
Lemma \ref{lem:highorderregdecoedpbcLip} yields
\begin{align}
\label{regdecogenLip}
R(\ed^{q-1,k}_{\gat})
=\ed\H{q-1,k}{\gat}(\ed,\om)
=\ed\H{q-1,k+1}{\gat}(\om)
&\subset\H{q,k}{\gat,0}(\ed,\om)
\cap\Harm{q}{\gat,\gan,\eps}(\om)^{\bot_{\L{q,2}{\eps}(\om)}}.
\end{align}
For $k=0$ we get by \eqref{regdecogenLip} and Lemma \ref{helmom}
\begin{align}
\label{hemldecospdeco}
\ed\H{q-1,1}{\gat}(\om)
=\ed\H{q-1,0}{\gat}(\ed,\om)
&=\H{q,0}{\gat,0}(\ed,\om)
\cap\Harm{q}{\gat,\gan,\eps}(\om)^{\bot_{\L{q,2}{\eps}(\om)}}.
\end{align}
Let $E\in\H{q,k}{\gat,0}(\ed,\om)\cap\Harm{q}{\gat,\gan,\eps}(\om)^{\bot_{\L{q,2}{\eps}(\om)}}$.
By \eqref{hemldecospdeco} we observe
$E\in\H{q,k}{\gat}(\om)\cap\ed\H{q-1,1}{\gat}(\om)$, i.e.,
$E=\ed E_{1}\in\H{q,k}{\gat}(\om)$ with $E_{1}\in\H{q-1,1}{\gat}(\om)$. 
Thus $E_{1}\in\H{q-1,1}{\gat}(\ed,\om)$ and $E\in\ed\H{q-1,1}{\gat}(\ed,\om)$. 
By \eqref{regdecogenLip} there is 
$E_{2}\in\H{q-1,2}{\gat}(\om)$ with 
$E=\ed E_{2}\in\ed\H{q-1,k}{\gat}(\om)$, i.e.,
$E_{2}\in\H{q-1,2}{\gat}(\ed,\om)$ as well as $E\in\ed\H{q-1,2}{\gat}(\ed,\om)$. 
After $k$ induction steps we obtain $E\in\ed\H{q-1,k}{\gat}(\ed,\om)$.
Hodge $\star$-duality shows the assertions for $\cd$.
\end{proof}

Note that in Corollary \ref{cor:highorderregdecoedpbcLip} we claim
nothing about bounded regular potential operators, leaving 
the question of bounded potentials to the next sections.

\subsection{Zero Order Mini FA-ToolBox}

We shall apply Theorem \ref{theo:cptembmaintheo1} 
from the FA-ToolBox to the zero order de Rham complex.
In Section \ref{sec:derhamlow} we have seen that 
\begin{align*}
\A_{0}:=\ed_{\gat}^{q-1}:\H{q-1,0}{\gat}(\ed,\om)\subset\L{q-1,2}{}(\om)
&\to\L{q,2}{}(\om),\\
\A_{1}:=\ed_{\gat}^{q}:\H{q,0}{\gat}(\ed,\om)\subset\L{q,2}{}(\om)
&\to\L{q+1,2}{}(\om),\\
\A_{0}^{*}=-\cd_{\gan}^{q}:\H{q,0}{\gan}(\cd,\om)\subset\L{q,2}{}(\om)
&\to\L{q-1,2}{}(\om),\\
\A_{1}^{*}=-\cd_{\gan}^{q+1}:\H{q+1,0}{\gan}(\cd,\om)\subset\L{q+1,2}{}(\om)
&\to\L{q,2}{}(\om)
\end{align*}
are densely defined and closed and
form a Hilbert complex of dual pairs, i.e., 
the long primal and dual Hilbert complex \eqref{derhamcomplex1}.
Recall also \eqref{lhcomplex2} and Definition \ref{defihilcom2}
are well as Remark \ref{remhilcom3}.

Lemma \ref{lem:highorderregdecoedpbcLip} for $k=0$ 
yields the bounded regular decomposition
\begin{align*}
D(\A_{1})
=\H{q,0}{\gat}(\ed,\om)
=\H{q,1}{\gat}(\om)
+\ed\H{q-1,1}{\gat}(\om)
=\H{+}{1}
+\A_{0}\H{+}{0}
\end{align*}
with $\H{+}{1}:=\H{q,1}{\gat}(\om)$ and $\H{+}{0}:=\H{q-1,1}{\gat}(\om)$
and $\H{}{1}:=\L{q,2}{}(\om)$ and $\H{}{0}:=\L{q-1,2}{}(\om)$.
Rellich's selection theorem shows that the assumptions 
of Lemma \ref{lem:cptembmaintheo} (i) and Theorem \ref{theo:cptembmaintheo1} as satisfied.
Note that it holds $D(\ed_{\gat}^{0})=\H{0,1}{\gat}(\om)$ and $D(\cd_{\gan}^{\dimom})=\H{\dimom,1}{\gan}(\om)$.

\begin{theo}[compact embedding for the de Rham complex]
\label{theo:cptemb:derham}
Let $(\om,\gat)$ be a bounded strong Lipschitz pair. 
Then for all $q$ the embedding 
$$D(\A_{1})\cap D(\A_{0}^{*})
=D(\ed_{\gat}^{q})\cap D(\cd_{\gan}^{q})
=\H{q,0}{\gat}(\ed,\om)\cap\H{q,0}{\gan}(\cd,\om)
\incl\L{q,2}{}(\om)$$
is compact. Moreover, the long primal and dual de Rham Hilbert complex \eqref{derhamcomplex1}
is compact. In particular, the complex is closed.
\end{theo}

\begin{proof}
Apply Theorem \ref{theo:cptembmaintheo1} (i).
\end{proof}

\begin{theo}[mini FA-ToolBox for the de Rham complex]
\label{theo:minifatb:derham}
Let $(\om,\gat)$ be a bounded strong Lipschitz pair. 
Then for all $q$ 
\begin{itemize}
\item[\bf(i)]
the ranges $R(\ed_{\gat}^{q})$ and $R(\cd_{\gan}^{q})$ are closed,
\item[\bf(ii)]
the inverse operators $(\ed_{\gat}^{q})_{\bot}^{-1}$
and $(\cd_{\gan}^{q})_{\bot}^{-1}$ are compact,
\item[\bf(iii)]
the cohomology group $\Harm{q}{\gat,\gan,\id}(\om)
=\H{q}{\gat,0}(\ed,\om)\cap\H{q}{\gan,0}(\cd,\om)$ has finite dimension,
\item[\bf(iv)]
the orthogonal Helmholtz-type decomposition
$$\L{q,2}{}(\om)
=\ed\H{q-1,0}{\gat}(\ed,\om)
\oplus_{\L{q,2}{}(\om)}
\Harm{q}{\gat,\gan,\id}(\om)
\oplus_{\L{q,2}{}(\om)}
\cd\H{q+1,0}{\gan}(\cd,\om)$$
holds,
\item[\bf(v)]
there exists $c_{q}>0$ such that
\begin{align*}
\forall\,E&\in D\big((\ed_{\gat}^{q})_{\bot}\big)
&
\norm{E}_{\L{q,2}{}(\om)}&\leq c_{q}\norm{\ed E}_{\L{q+1,2}{}(\om)},\\
\forall\,H&\in D\big((\cd_{\gan}^{q+1})_{\bot}\big)
&
\norm{H}_{\L{q+1,2}{}(\om)}&\leq c_{q}\norm{\cd H}_{\L{q,2}{}(\om)},
\end{align*}
where 
\begin{align*}
D\big((\ed_{\gat}^{q})_{\bot}\big)
&=D(\ed_{\gat}^{q})\cap N(\ed_{\gat}^{q})^{\bot_{\L{q,2}{}(\om)}}=D(\ed_{\gat}^{q})\cap R(\cd_{\gan}^{q+1}),\\
D\big((\cd_{\gan}^{q+1})_{\bot}\big)
&=D(\cd_{\gan}^{q+1})\cap N(\cd_{\gan}^{q+1})^{\bot_{\L{q+1,2}{}(\om)}}=D(\cd_{\gan}^{q+1})\cap R(\ed_{\gat}^{q}),
\end{align*}
\item[\bf(v')]
with $c_{q}$ from (v) it holds for all
$E\in D(\ed_{\gat}^{q})\cap D(\cd_{\gan}^{q})\cap\Harm{q}{\gat,\gan,\id}(\om)^{\bot_{\L{q,2}{}(\om)}}$
$$\norm{E}_{\L{q,2}{}(\om)}^2
\leq c_{q}^2\norm{\ed E}_{\L{q+1,2}{}(\om)}^2
+c_{q-1}^2\norm{\cd E}_{\L{q-1,2}{}(\om)}^2,$$
\item[\bf(vi)]
$\Harm{q}{\gat,\gan,\id}(\om)=\{0\}$, if $(\om,\gat)$ is additionally extendable.
\end{itemize}
\end{theo}

\begin{proof}
Apply Theorem \ref{theo:cptembmaintheo1} (ii), i.e.,
Theoren \ref{theo:cptemb:derham} and Theorem \ref{theo:toolboxgenmain} show (i)-(v').
For $k=0$ Lemma \ref{highorderregpotedpbc} and Lemma \ref{helmom} imply
$\ed\H{q-1,0}{\gat}(\ed,\om)
=\H{q,0}{\gat,0}(\ed,\om)
=\ed\H{q-1,0}{\gat}(\ed,\om)
\oplus_{\L{q,2}{}(\om)}
\Harm{q}{\gat,\gan,\id}(\om)$, 
i.e., (vi).
\end{proof}

\begin{rem}[mini FA-ToolBox for the de Rham complex]
\label{rem:derhameps}
Recall the admissible weights $\eps$ from Section \ref{sec:defsobolevdforms}.
In \cite[Lemma 5.1, Lemma 5.2]{PW2020a} we have shown that
the compactness in Theoren \ref{theo:cptemb:derham} and 
the dimensions of the cohomology groups do not depend on the particular $\eps$.
Hence, for all $q$
\begin{itemize}
\item[\bf(i)]
the embedding $\H{q,0}{\gat}(\ed,\om)\cap\eps^{-1}\H{q,0}{\gan}(\cd,\om)\incl\L{q,2}{}(\om)$
is compact,
\item[\bf(ii)]
$d_{\om,\gat}^{q}:=\dim\Harm{q}{\gat,\gan,\eps}(\om)=\dim\Harm{q}{\gat,\gan,\id}(\om)$.
\item[\bf(iii)]
Theorem \ref{theo:minifatb:derham} holds with appropriate modifications for including $\eps$.
\end{itemize}
Compare to the more explicit formulations from Section \ref{sec:vecderham}
for the vector de Rham complex. All these results carry over literally. 
In particular, cf.~Theorem \ref{theo:minifatb:derham} (v'), we have with 
$c_{q}$ (now depending also on $\eps$ and $\mu$) for all
$E\in D(\mu^{-1}\ed_{\gat}^{q})\cap D(\cd_{\gan}^{q}\eps)\cap\Harm{q}{\gat,\gan,\eps}(\om)^{\bot_{\L{q,2}{\eps}(\om)}}$
$$\norm{E}_{\L{q,2}{\eps}(\om)}^2
\leq c_{q}^2\norm{\mu^{-1}\ed E}_{\L{q+1,2}{\mu}(\om)}^2
+c_{q-1}^2\norm{\cd\eps E}_{\L{q-1,2}{}(\om)}^2.$$
Moreover,
\begin{itemize}
\item[\bf(iv)]
Theorem \ref{theo:cptemb:derham} and hence Theorem \ref{theo:minifatb:derham}
and (i)-(iii) of this remark hold more generally
for bounded weak Lipschitz pairs $(\om,\gat)$, see \cite{BPS2018a,BPS2019a}. 
\end{itemize}
\end{rem}

\begin{theo}[bounded regular potentials for the de Rham complex]
\label{theo:regpotderhamzero}
Let $(\om,\gat)$ be a bounded strong Lipschitz pair
and let $\PotQ_{\ed,\gat,1}^{q,0}$ be given from Lemma \ref{lem:highorderregdecoedpbcLip}. 
Then for all $q\in\{1,\dots,\dimom\}$
there exists a bounded linear regular potential operator
$$\PotP_{\ed,\gat}^{q,0}:=\PotQ_{\ed,\gat,1}^{q-1,0}(\ed_{\gat}^{q-1})_{\bot}^{-1}:
\H{q,0}{\gat,0}(\ed,\om)
\cap\Harm{q}{\gat,\gan,\eps}(\om)^{\bot_{\L{q,2}{\eps}(\om)}}
\To\H{q-1,1}{\gat}(\om),$$
such that 
$\ed\PotP_{\ed,\gat}^{q,0}=\id|_{\H{q,0}{\gat,0}(\ed,\om)\cap\Harm{q}{\gat,\gan,\eps}(\om)^{\bot_{\L{q,2}{\eps}(\om)}}}$.
In particular, the bounded regular potential representations
$$R(\ed_{\gat}^{q-1})
=\H{q,0}{\gat,0}(\ed,\om)
\cap\Harm{q}{\gat,\gan,\eps}(\om)^{\bot_{\L{q,2}{\eps}(\om)}}
=\ed\H{q-1,0}{\gat}(\ed,\om)
=\ed\H{q-1,1}{\gat}(\om)$$
hold and the potentials can be chosen such that they depend continuously on the data.
\end{theo}

\begin{proof}
Apply Theorem \ref{theo:cptembmaintheo1} (iii).
Note that $R(\ed_{\gat}^{q-1})$ is closed by Theorem \ref{theo:minifatb:derham} and hence
$$R(\ed_{\gat}^{q-1})
=\ed\H{q-1,0}{\gat}(\ed,\om)
=\H{q,0}{\gat,0}(\ed,\om)
\cap\Harm{q}{\gat,\gan,\eps}(\om)^{\bot_{\L{q,2}{\eps}(\om)}}$$
holds by Lemma \ref{helmom}.
\end{proof}

\begin{rem}[Dirichlet/Neumann forms]
\label{highorderregpotedpbcLipremdirneu}
Note that 
$\Harm{\dimom}{\gat,\gan,\eps}(\om)
=\eps^{-1}\H{\dimom}{\gan,0}(\cd,\om)
=\eps^{-1}*\reals_{\gan}$
and 
$\Harm{\dimom}{\gat,\gan,\eps}(\om)^{\bot_{\L{\dimom,2}{\eps}(\om)}}
=(*\reals_{\gan})^{\bot_{\L{\dimom,2}{}(\om)}}$
holds in the special case $q=\dimom$.
\end{rem}

\begin{theo}[bounded regular decompositions for the de Rham complex]
\label{theo:highorderregdecoedpbcLip2}
Let $(\om,\gat)$ be a bounded strong Lipschitz pair
and let $\PotP_{\ed,\gat}^{q,0}$ and $\PotQ_{\ed,\gat,1}^{q,0}$
be given from Theorem \ref{theo:regpotderhamzero}
and from Lemma \ref{lem:highorderregdecoedpbcLip}, respectively.
Then the bounded regular decompositions
\begin{align*}
\H{q}{\gat}(\ed,\om)
=\H{q,0}{\gat}(\ed,\om)
&=\H{q,1}{\gat}(\om)
+\H{q,0}{\gat,0}(\ed,\om)
=\H{q,1}{\gat}(\om)
+\ed\H{q-1,1}{\gat}(\om)\\
&=R(\widetilde\PotQ_{\ed,\gat,1}^{q,0})
\dotplus\H{q,0}{\gat,0}(\ed,\om)
=R(\widetilde\PotQ_{\ed,\gat,1}^{q,0})
\dotplus R(\widetilde\PotN_{\ed,\gat}^{q,0})
\end{align*}
hold with bounded linear regular decomposition operators
\begin{align*}
\widetilde\PotQ_{\ed,\gat,1}^{q,0}:=\PotP_{\ed,\gat}^{q+1,0}\ed^{q}_{\gat}:\H{q,0}{\gat}(\ed,\om)\to\H{q,1}{\gat}(\om),\qquad
\widetilde\PotN_{\ed,\gat}^{q,0}:\H{q,0}{\gat}(\ed,\om)\to\H{q,0}{\gat,0}(\ed,\om)
\end{align*}
satisfying $\widetilde\PotQ_{\ed,\gat,1}^{q,0}+\widetilde\PotN_{\ed,\gat}^{q,0}=\id_{\H{q,0}{\gat}(\ed,\om)}$.
Moreover, it holds $\ed\widetilde\PotQ_{\ed,\gat,1}^{q,0}=\ed\PotQ_{\ed,\gat,1}^{q,0}=\ed^{q}_{\gat}$
and thus $\H{q,0}{\gat,0}(\ed,\om)$ is invariant under 
$\PotQ_{\ed,\gat,1}^{q,0}$ and $\widetilde\PotQ_{\ed,\gat,1}^{q,0}$.
Furthermore, $R(\widetilde\PotQ_{\ed,\gat,1}^{q,0})=R(\PotP_{\ed,\gat}^{q+1,0})$ and 
$\widetilde\PotQ_{\ed,\gat,1}^{q,0}
=\PotP_{\ed,\gat}^{q+1,0}\ed^{q}_{\gat}
=\PotQ_{\ed,\gat,1}^{q,0}(\ed_{\gat}^{q})_{\bot}^{-1}\ed^{q}_{\gat}$.
Hence 
$\widetilde\PotQ_{\ed,\gat,1}^{q,0}|_{D((\ed_{\gat}^{q})_{\bot})}
=\PotQ_{\ed,\gat,1}^{q,0}|_{D((\ed_{\gat}^{q})_{\bot})}$
and thus $\widetilde\PotQ_{\ed,\gat,1}^{q,0}$
may differ from $\PotQ_{\ed,\gat,1}^{q,0}$ 
only on $\H{q,0}{\gat,0}(\ed,\om)$.
\end{theo}

\begin{proof}
Apply Theorem \ref{theo:cptembmaintheo1} (iv) and (iv').
\end{proof}

Again, Theorem \ref{theo:regpotderhamzero} and Theorem \ref{theo:highorderregdecoedpbcLip2}
have dual versions for the $\cd$-operator by Hodge $\star$-duality, 
cf.~Theorem \ref{highorderregpotcdpbcLip} for $k=0$.

\subsection{Higher Order Mini FA-ToolBox}

Some results from the latter section hold even for higher Sobolev orders.
As pointed out in Section \ref{sec:derhamhigh}, the adjoints are much more complicated.
Hence Lemma \ref{lem:cptembmaintheo} and Theorem \ref{theo:cptembmaintheo1} from the FA-ToolBox
are not directly applicable, so that some detours and modifications are needed.

In Section \ref{sec:derhamhigh} we have introduced
the higher order primal and dual de Rham Hilbert complex composed of
the densely defined and closed linear operators
\begin{align*}
\ed_{\gat}^{q,k}:D(\ed_{\gat}^{q,k})\subset\H{q,k}{\gat}(\om)
&\to\H{q+1,k}{\gat}(\om),
&
D(\ed_{\gat}^{q,k})&=\H{q,k}{\gat}(\ed,\om),\\
\cd_{\gan}^{q,k}:D(\cd_{\gan}^{q,k})\subset\H{q,k}{\gan}(\om)
&\to\H{q-1,k}{\gan}(\om),
&
D(\cd_{\gan}^{q,k})&=\H{q,k}{\gan}(\cd,\om).
\end{align*}

By Corollary \ref{cor:highorderregdecoedpbcLip} be see:

\begin{theo}[higher order closed ranges for the de Rham complex]
\label{theo:clran:derhamk}
Let $(\om,\gat)$ be a bounded strong Lipschitz pair. 
Then for all $q$ and for all $k\in\nat_{0}$ the ranges
\begin{align*}
R(\ed^{q-1,k}_{\gat})
&=\ed\H{q-1,k}{\gat}(\ed,\om)
=\ed\H{q-1,k+1}{\gat}(\om)
=\H{q,k}{\gat,0}(\ed,\om)
\cap\Harm{q}{\gat,\gan,\id}(\om)^{\bot_{\L{q,2}{}(\om)}},\\
R(\cd^{q+1,k}_{\gan})
&=\cd\H{q+1,k}{\gan}(\cd,\om)
=\cd\H{q+1,k+1}{\gan}(\om)
=\H{q,k}{\gan,0}(\cd,\om)
\cap\Harm{q}{\gat,\gan,\id}(\om)^{\bot_{\L{q,2}{}(\om)}}
\end{align*}
are closed, i.e., closed subspaces of $\H{q,k}{}(\om)$.
In particular, the higher order long primal and dual de Rham complex 
from Section \ref{sec:derhamhigh} is closed.
\end{theo}

The corresponding reduced operators read
\begin{align*}
(\ed_{\gat}^{q,k})_{\bot}:D\big((\ed_{\gat}^{q,k})_{\bot}\big)
&\subset\H{q,k}{\gat,0}(\ed,\om)^{\bot_{\H{q,k}{\gat}(\om)}}
\to\ed\H{q,k}{\gat}(\ed,\om),
&
N(\ed_{\gat}^{q,k})&=\H{q,k}{\gat,0}(\ed,\om),\\
-(\cd_{\gan}^{q,k})_{\bot}:D\big((\cd_{\gan}^{q,k})_{\bot}\big)
&\subset\H{q,k}{\gan,0}(\cd,\om)^{\bot_{\H{q,k}{\gan}(\om)}}
\to\cd\H{q,k}{\gan}(\cd,\om),
&
N(\cd_{\gan}^{q,k})&=\H{q,k}{\gan,0}(\cd,\om),
\end{align*}
with
\begin{align*}
D\big((\ed_{\gat}^{q,k})_{\bot}\big)
&=\H{q,k}{\gat}(\ed,\om)\cap\H{q,k}{\gat,0}(\ed,\om)^{\bot_{\H{q,k}{\gat}(\om)}}
=\H{q,k}{\gat}(\ed,\om)\cap R\big((\ed_{\gat}^{q,k})^{*}\big),\\
D\big((\cd_{\gan}^{q,k})_{\bot}\big)
&=\H{q,k}{\gan}(\cd,\om)\cap\H{q,k}{\gan,0}(\cd,\om)^{\bot_{\H{q,k}{\gan}(\om)}}
=\H{q,k}{\gan}(\cd,\om)\cap R\big((\cd_{\gan}^{q,k})^{*}\big),
\end{align*}
and we have by Lemma \ref{lem:toolboxcpt1} and Theorem \ref{theo:clran:derhamk}:

\begin{theo}[higher order fundamental lemma 1 for the de Rham complex]
\label{theo:highorderFAT1}
Let $(\om,\gat)$ be a bounded strong Lipschitz pair. 
Then for all $q$ and for all $k\in\nat_{0}$
the following assertions hold and are equivalent:
\begin{itemize}
\item[\bf(i)]
$\exists\;c>0\quad\forall\,E\in D\big((\ed_{\gat}^{q,k})_{\bot}\big)\qquad
\norm{E}_{\H{q,k}{}(\om)}\leq c\norm{\ed E}_{\H{q+1,k}{}(\om)}$
\item[\bf(ii)]
$R(\ed_{\gat}^{q,k})=R\big((\ed_{\gat}^{q,k})_{\bot}\big)=\ed\H{q,k}{\gat}(\ed,\om)$ is closed.
\item[\bf(iii)]
$(\ed_{\gat}^{q,k})_{\bot}^{-1}:R(\ed_{\gat}^{q,k})\to D\big((\ed_{\gat}^{q,k})_{\bot}\big)$ is bounded.
\item[\bf(iii')]
$(\ed_{\gat}^{q,k})_{\bot}^{-1}:R(\ed_{\gat}^{q,k})\to D(\ed_{\gat}^{q,k})$ is bounded.
\end{itemize}
The corresponding results hold for the $\cd_{\gan}^{q,k}$ as well.
\end{theo}

The higher order version of Theorem \ref{theo:cptemb:derham} reads as follows:

\begin{theo}[higher order compact embedding for the de Rham complex]
\label{theo:cptemb:derhamk}
Let $(\om,\gat)$ be a bounded strong Lipschitz pair. 
Then for all $q$ and for all $k\in\nat_{0}$ the embedding 
$$D(\ed_{\gat}^{q,k})\cap D(\cd_{\gan}^{q,k})
=\H{q,k}{\gat}(\ed,\om)\cap\H{q,k}{\gan}(\cd,\om)
\incl\H{q,k}{\ga}(\om)$$
is compact. 
\end{theo}

\begin{proof}
We follow in close lines the proof of \cite[Theorem 4.11]{PZ2020b} using induction.
The case $k=0$ is given by Theorem \ref{theo:cptemb:derham}.
Let $k\geq1$ and let $(E_{n})$ be a bounded sequence in 
$\H{q,k}{\gat}(\ed,\om)\cap\H{q,k}{\gan}(\cd,\om)$.
Note that 
$$\H{q,k}{\gat}(\ed,\om)\cap\H{q,k}{\gan}(\cd,\om)
\subset\H{q,k}{\gat}(\om)\cap\H{q,k}{\gan}(\om)
=\H{q,k}{\ga}(\om).$$
By assumption and w.l.o.g. we have that $(E_{n})$ 
is a Cauchy sequence in $\H{q,k-1}{\ga}(\om)$.
Moreover, for all $|\alpha|=k$ we have
$\p^{\alpha}E_{n}\in\H{q,0}{\gat}(\ed,\om)\cap\H{q,0}{\gan}(\cd,\om)$
with $\ed\p^{\alpha}E_{n}=\p^{\alpha}\ed E_{n}$
and $\cd\p^{\alpha}E_{n}=\p^{\alpha}\cd E_{n}$
by Lemma \ref{lem:edpalpha}.
Hence $(\p^{\alpha}E_{n})$ is a bounded sequence 
in $\H{q,0}{\gat}(\ed,\om)\cap\H{q,0}{\gan}(\cd,\om)$.
Thus, w.l.o.g. $(\p^{\alpha}E_{n})$ is a Cauchy sequence in $\L{q,2}{}(\om)$
by Theorem \ref{theo:cptemb:derham}.
Finally, $(E_{n})$ is a Cauchy sequence in $\H{q,k}{\ga}(\om)$,
finishing the proof.
\end{proof}

Higher order analogues of Theorem \ref{theo:minifatb:derham} and Remark \ref{rem:derhameps} hold.
Some of these results are formulated in the following theorem.

\begin{theo}[higher order Friedrichs/Poincar\'e type estimates for the de Rham complex]
\label{theo:cptemb:derhamk2}
Let $(\om,\gat)$ be a bounded strong Lipschitz pair. 
Then for all $q$ and for all $k\geq0$ there exists $\widetilde{c}_{q,k}>0$ such that for all
$E\in\H{q,k}{\gat}(\ed,\om)\cap\H{q,k}{\gan}(\cd,\om)\cap\Harm{q}{\gat,\gan,\id}(\om)^{\bot_{\L{q,2}{}(\om)}}$
$$\norm{E}_{\H{q,k}{}(\om)}
\leq\widetilde{c}_{q,k}\big(\norm{\ed E}_{\H{q+1,k}{}(\om)}
+\norm{\cd E}_{\H{q-1,k}{}(\om)}\big).$$
The condition $\Harm{q}{\gat,\gan,\id}(\om)^{\bot_{\L{q,2}{}(\om)}}$ can be replaced by the weaker conditions
$\Harm{q,k}{\gat,\gan,\id}(\om)^{\bot_{\L{q,2}{}(\om)}}$
or $\Harm{q,k}{\gat,\gan,\id}(\om)^{\bot_{\H{q,k}{}(\om)}}$.
In particular, it holds
\begin{align*}
\forall\;E&\in\H{q,k}{\gat}(\ed,\om)\cap R(\cd^{q+1,k}_{\gan})
&
\norm{E}_{\H{q,k}{}(\om)}
&\leq\widetilde{c}_{q,k}\norm{\ed E}_{\H{q+1,k}{}(\om)},\\
\forall\;E&\in\H{q,k}{\gan}(\cd,\om)\cap R(\ed^{q-1,k}_{\gat})
&
\norm{E}_{\H{q,k}{}(\om)}
&\leq\widetilde{c}_{q,k}\norm{\cd E}_{\H{q-1,k}{}(\om)}
\end{align*}
with
\begin{align*}
R(\cd^{q+1,k}_{\gan})
&=\H{q,k}{\gan,0}(\cd,\om)
\cap\Harm{q}{\gat,\gan,\id}(\om)^{\bot_{\L{q,2}{}(\om)}},\\
R(\ed^{q-1,k}_{\gat})
&=\H{q,k}{\gat,0}(\ed,\om)
\cap\Harm{q}{\gat,\gan,\id}(\om)^{\bot_{\L{q,2}{}(\om)}}.
\end{align*}
\end{theo}

\begin{proof}
To show the first estimate, we use a standard strategy and assume the contrary. Then there is a sequence 
$$(E_{n})\subset\H{q,k}{\gat}(\ed,\om)\cap\H{q,k}{\gan}(\cd,\om)\cap\Harm{q}{\gat,\gan,\id}(\om)^{\bot_{\L{q,2}{}(\om)}}$$
with $\norm{E_{n}}_{\H{q,k}{}(\om)}=1$ and 
$\norm{\ed E_{n}}_{\H{q+1,k}{}(\om)}+\norm{\cd E_{n}}_{\H{q-1,k}{}(\om)}\to0$.
Hence we may assume that $E_{n}$ converges weakly to some $E$ in
$\H{q,k}{}(\om)\cap\Harm{q}{\gat,\gan,\id}(\om)\cap\Harm{q}{\gat,\gan,\id}(\om)^{\bot_{\L{q,2}{}(\om)}}$.
Thus $E=0$. By Theorem \ref{theo:cptemb:derhamk} $(E_{n})$ converges strongly to $0$ in $\H{q,k}{}(\om)$,
in contradiction to $\norm{E_{n}}_{\H{q,k}{}(\om)}=1$.

The other two estimates follow with Theorem \ref{theo:clran:derhamk} by restriction.
\end{proof}

Note that by Theorem \ref{theo:highorderFAT1}
$$(\ed_{\gat}^{q,k})_{\bot}^{-1}:R(\ed_{\gat}^{q,k})\to D(\ed_{\gat}^{q,k}),\qquad
(\cd_{\gan}^{q,k})_{\bot}^{-1}:R(\cd_{\gan}^{q,k})\to D(\cd_{\gan}^{q,k})$$ 
are bounded.
The higher order versions of Theorem \ref{theo:regpotderhamzero}
and Theorem \ref{theo:highorderregdecoedpbcLip2} read as follows:

\begin{theo}[higher order bounded regular potentials and decompositions for the de Rham complex]
\label{highorderregpotedpbcLip}
Let $(\om,\gat)$ be a bounded strong Lipschitz pair and let $k\geq0$. 
Moreover, let $\PotQ_{\ed,\gat,1}^{q,k}$ be given from Lemma \ref{lem:highorderregdecoedpbcLip}. Then:
\begin{itemize}
\item[\bf(i)]
For all $q\in\{1,\dots,\dimom\}$
there exists a bounded linear regular potential operator
$$\PotP_{\ed,\gat}^{q,k}:=\PotQ_{\ed,\gat,1}^{q-1,k}(\ed_{\gat}^{q-1,k})_{\bot}^{-1}:
\H{q,k}{\gat,0}(\ed,\om)
\cap\Harm{q}{\gat,\gan,\eps}(\om)^{\bot_{\L{q,2}{\eps}(\om)}}
\To\H{q-1,k+1}{\gat}(\om),$$
such that 
$\ed\PotP_{\ed,\gat}^{q,k}=\id|_{\H{q,k}{\gat,0}(\ed,\om)\cap\Harm{q}{\gat,\gan,\eps}(\om)^{\bot_{\L{q,2}{\eps}(\om)}}}$.
In particular, the bounded regular representations
\begin{align*}
R(\ed_{\gat}^{q-1,k})
&=\H{q,k}{\gat,0}(\ed,\om)
\cap\Harm{q}{\gat,\gan,\eps}(\om)^{\bot_{\L{q,2}{\eps}(\om)}}\\
&=\H{q,k}{\gat}(\om)
\cap\ed\H{q-1}{\gat}(\ed,\om)
=\ed\H{q-1,k}{\gat}(\ed,\om)
=\ed\H{q-1,k+1}{\gat}(\om)
\end{align*}
hold and the potentials can be chosen such that they depend continuously on the data.
\item[\bf(ii)]
The bounded regular decompositions
\begin{align*}
\H{q,k}{\gat}(\ed,\om)
&=\H{q,k+1}{\gat}(\om)
+\H{q,k}{\gat,0}(\ed,\om)
=\H{q,k+1}{\gat}(\om)
+\ed\H{q-1,k+1}{\gat}(\om)\\
&=R(\widetilde\PotQ_{\ed,\gat,1}^{q,k})
\dotplus\H{q,k}{\gat,0}(\ed,\om)
=R(\widetilde\PotQ_{\ed,\gat,1}^{q,k})
\dotplus R(\widetilde\PotN_{\ed,\gat}^{q,k})
\end{align*}
hold with bounded linear regular decomposition operators
\begin{align*}
\widetilde\PotQ_{\ed,\gat,1}^{q,k}:=\PotP_{\ed,\gat}^{q+1,k}\ed^{q,k}_{\gat}:\H{q,k}{\gat}(\ed,\om)\to\H{q,k+1}{\gat}(\om),\qquad
\widetilde\PotN_{\ed,\gat}^{q,k}:\H{q,k}{\gat}(\ed,\om)\to\H{q,k}{\gat,0}(\ed,\om)
\end{align*}
satisfying $\widetilde\PotQ_{\ed,\gat,1}^{q,k}+\widetilde\PotN_{\ed,\gat}^{q,k}=\id_{\H{q,k}{\gat}(\ed,\om)}$.
Moreover, $\ed\widetilde\PotQ_{\ed,\gat,1}^{q,k}=\ed\PotQ_{\ed,\gat,1}^{q,k}=\ed^{q,k}_{\gat}$
and thus $\H{q,k}{\gat,0}(\ed,\om)$ is invariant under 
$\PotQ_{\ed,\gat,1}^{q,k}$ and $\widetilde\PotQ_{\ed,\gat,1}^{q,k}$.
It holds $R(\widetilde\PotQ_{\ed,\gat,1}^{q,k})=R(\PotP_{\ed,\gat}^{q+1,k})$ and 
$\widetilde\PotQ_{\ed,\gat,1}^{q,k}
=\PotP_{\ed,\gat}^{q+1,k}\ed^{q,k}_{\gat}
=\PotQ_{\ed,\gat,1}^{q,k}(\ed_{\gat}^{q,k})_{\bot}^{-1}\ed^{q,k}_{\gat}$.
Hence 
$\widetilde\PotQ_{\ed,\gat,1}^{q,k}|_{D((\ed_{\gat}^{q,k})_{\bot})}
=\PotQ_{\ed,\gat,1}^{q,k}|_{D((\ed_{\gat}^{q,k})_{\bot})}$
and thus $\widetilde\PotQ_{\ed,\gat,1}^{q,k}$
may differ from $\PotQ_{\ed,\gat,1}^{q,k}$ 
only on $\H{q,k}{\gat,0}(\ed,\om)$.
\item[\bf(ii')]
The bounded regular kernel decomposition 
$\H{q,k}{\gat,0}(\ed,\om)=\H{q,k+1}{\gat,0}(\ed,\om)+\ed\H{q-1,k+1}{\gat}(\om)$
holds.
\end{itemize}
\end{theo}

\begin{proof}
Lemma \ref{lem:highorderregdecoedpbcLip}
yields the bounded regular decomposition
\begin{align*}
D(\ed_{\gat}^{q,k})
=\H{q,k}{\gat}(\ed,\om)
=\H{q,k+1}{\gat}(\om)
+\ed\H{q-1,k+1}{\gat}(\om)
=\H{+}{1}
+\ed_{\gat}^{q-1,k}\H{+}{0}
\end{align*}
with $\H{+}{1}:=\H{q,k+1}{\gat}(\om)$ and $\H{+}{0}:=\H{q-1,k+1}{\gat}(\om)$
and $\H{}{1}:=\H{q,k}{\gat}(\om)$ and $\H{}{0}:=\H{q-1,k}{\gat}(\om)$.
Rellich's selection theorem shows that the assumptions 
of Lemma \ref{lem:cptembmaintheo} (i) and Theorem \ref{theo:cptembmaintheo1} as satisfied.
Note that it holds $D(\ed_{\gat}^{0,k})=\H{0,k+1}{\gat}(\om)$ and $D(\cd_{\gan}^{\dimom,k})=\H{\dimom,k+1}{\gan}(\om)$.
Theorem \ref{theo:cptembmaintheo1} (iii)-(iv') 
and Theorem \ref{theo:clran:derhamk} show the assertions (i) and (ii).
(ii') follows directly by (ii).
\end{proof}

Hodge $\star$-duality yields the corresponding results for the co-derivative as well, 
cf.~Theorem \ref{highorderregpotcdpbcLip}.

\begin{rem}
\label{highorderregpotedpbcLiprem}
Let us recall the bounded regular decompositions
from Theorem \ref{highorderregpotedpbcLip} (ii), e.g.,
\begin{align*}
\H{q,k}{\gat}(\ed,\om)
&=R(\widetilde\PotQ_{\ed,\gat,1}^{q,k})
\dotplus R(\widetilde\PotN_{\ed,\gat}^{q,k}).
\end{align*}
By Remark \ref{rem:regpotdeco2} we emphasise:
\begin{itemize}
\item[\bf(i)]
$\widetilde\PotQ_{\ed,\gat,1}^{q,k}$ 
and $\widetilde\PotN_{\ed,\gat}^{q,k}=1-\widetilde\PotQ_{\ed,\gat,1}^{q,k}$
are projections with 
$\widetilde\PotQ_{\ed,\gat,1}^{q,k}\widetilde\PotN_{\ed,\gat}^{q,k}
=\widetilde\PotN_{\ed,\gat}^{q,k}\widetilde\PotQ_{\ed,\gat,1}^{q,k}=0$.
\item[\bf(ii)]
For $I_{\pm}:=\widetilde\PotQ_{\ed,\gat,1}^{q,k}\pm\widetilde\PotN_{\ed,\gat}^{q,k}$
it holds $I_{+}=I_{-}^{2}=\id_{\H{q,k}{\gat}(\ed,\om)}$.
Therefore, $I_{+}$, $I_{-}^{2}$, as well as
$I_{-}=2\widetilde\PotQ_{\ed,\gat,1}^{q,k}-\id_{\H{q,k}{\gat}(\ed,\om)}$
are topological isomorphisms on $\H{q,k}{\gat}(\ed,\om)$.
\item[\bf(iii)]
There exists $c>0$ such that for all $E\in\H{q,k}{\gat}(\ed,\om)$
\begin{align*}
c\norm{\widetilde\PotQ_{\ed,\gat,1}^{q,k}E}_{\H{q,k+1}{}(\om)}
&\leq\norm{\ed E}_{\H{q+1,k}{}(\om)}
\leq\norm{E}_{\H{q,k}{}(\ed,\om)},\\
\norm{\widetilde\PotN_{\ed,\gat}^{q,k}E}_{\H{q,k}{}(\om)}
&\leq\norm{E}_{\H{q,k}{}(\om)}
+\norm{\widetilde\PotQ_{\ed,\gat,1}^{q,k}E}_{\H{q,k}{}(\om)}.
\end{align*}
\item[\bf(iii')]
For $E\in\H{q,k}{\gat,0}(\ed,\om)$ we have $\widetilde\PotQ_{\ed,\gat,1}^{q,k}E=0$
and $\widetilde\PotN_{\ed,\gat}^{q,k}E=E$, i.e.,
$\widetilde\PotQ_{\ed,\gat,1}^{q,k}|_{\H{q,k}{\gat,0}(\ed,\om)}=0$
and $\widetilde\PotN_{\ed,\gat}^{q,k}|_{\H{q,k}{\gat,0}(\ed,\om)}=\id_{\H{q,k}{\gat,0}(\ed,\om)}$.
In particular, $\widetilde\PotN_{\ed,\gat}^{q,k}$ is onto.
\end{itemize}
\end{rem}

Theorem \ref{highorderregpotedpbcLip} (ii') shows by induction
and by Hodge $\star$-duality:

\begin{cor}[higher order kernels for the de Rham complex]
\label{cor:highorderkernelderham}
Let $(\om,\gat)$ be a bounded strong Lipschitz pair and let $k,\ell\geq0$.
Then the bounded regular kernel decompositions
\begin{align*}
\H{q,k}{\gat,0}(\ed,\om)&=\H{q,\ell}{\gat,0}(\ed,\om)+\ed\H{q-1,k+1}{\gat}(\om),
&
\H{q,k}{\gan,0}(\cd,\om)&=\H{q,\ell}{\gan,0}(\cd,\om)+\cd\H{q+1,k+1}{\gan}(\om)
\intertext{hold. In particular, for $k=0$ and all $\ell\geq0$}
\H{q,0}{\gat,0}(\ed,\om)&=\H{q,\ell}{\gat,0}(\ed,\om)+\ed\H{q-1,1}{\gat}(\om),
&
\H{q,0}{\gan,0}(\cd,\om)&=\H{q,\ell}{\gan,0}(\cd,\om)+\cd\H{q+1,1}{\gan}(\om).
\end{align*}
\end{cor}

\subsection{Dirichlet/Neumann Forms}

By Lemma \ref{helmom} we recall the orthonormal Helmholtz decompositions
\begin{align}
\begin{aligned}
\label{helmcoho1}
\L{q,2}{\eps}(\om)
&=\ed\H{q-1,0}{\gat}(\ed,\om)
\oplus_{\L{q,2}{\eps}(\om)}
\eps^{-1}\H{q,0}{\gan,0}(\cd,\om)\\
&=\H{q,0}{\gat,0}(\ed,\om)
\oplus_{\L{q,2}{\eps}(\om)}
\eps^{-1}\cd\H{q+1,0}{\gan}(\cd,\om)\\
&=\ed\H{q-1,0}{\gat}(\ed,\om)
\oplus_{\L{q,2}{\eps}(\om)}
\Harm{q}{\gat,\gan,\eps}(\om)
\oplus_{\L{q,2}{\eps}(\om)}
\eps^{-1}\cd\H{q+1,0}{\gan}(\cd,\om),\\
\H{q,0}{\gat,0}(\ed,\om)
&=\ed\H{q-1,0}{\gat}(\ed,\om)
\oplus_{\L{q,2}{\eps}(\om)}
\Harm{q}{\gat,\gan,\eps}(\om),\\
\eps^{-1}\H{q,0}{\gan,0}(\cd,\om)
&=\Harm{q}{\gat,\gan,\eps}(\om)
\oplus_{\L{q,2}{\eps}(\om)}
\eps^{-1}\cd\H{q+1,0}{\gan}(\cd,\om).
\end{aligned}
\end{align}
Let us denote the $\L{q,2}{\eps}(\om)$-orthonormal projector onto 
$\eps^{-1}\H{q,0}{\gan,0}(\cd,\om)$ and $\H{q,0}{\gat,0}(\ed,\om)$ by
$$\pi_{\cd}:\L{q,2}{\eps}(\om)\to\eps^{-1}\H{q,0}{\gan,0}(\cd,\om),\qquad
\pi_{\ed}:\L{q,2}{\eps}(\om)\to\H{q,0}{\gat,0}(\ed,\om),$$
respectively. Then
$$\pi_{\cd}|_{\H{q,0}{\gat,0}(\ed,\om)}:\H{q,0}{\gat,0}(\ed,\om)\to\Harm{q}{\gat,\gan,\eps}(\om),\qquad
\pi_{\ed}|_{\eps^{-1}\H{q,0}{\gan,0}(\cd,\om)}:\eps^{-1}\H{q,0}{\gan,0}(\cd,\om)\to\Harm{q}{\gat,\gan,\eps}(\om)$$
are onto. Moreover, 
\begin{align*}
\pi_{\cd}|_{\ed\H{q-1,0}{\gat}(\ed,\om)}&=0,
&
\pi_{\ed}|_{\eps^{-1}\cd\H{q+1,0}{\gan}(\cd,\om)}&=0,\\
\pi_{\cd}|_{\Harm{q}{\gat,\gan,\eps}(\om)}&=\id_{\Harm{q}{\gat,\gan,\eps}(\om)},
&
\pi_{\ed}|_{\Harm{q}{\gat,\gan,\eps}(\om)}&=\id_{\Harm{q}{\gat,\gan,\eps}(\om)}.
\end{align*}
Therefore, by Corollary \ref{cor:highorderkernelderham} and for all $\ell\geq0$
\begin{align*}
\Harm{q}{\gat,\gan,\eps}(\om)
&=\pi_{\cd}\H{q,0}{\gat,0}(\ed,\om)
=\pi_{\cd}\H{q,\ell}{\gat,0}(\ed,\om),\\
\Harm{q}{\gat,\gan,\eps}(\om)
&=\pi_{\ed}\eps^{-1}\H{q,0}{\gan,0}(\cd,\om)
=\pi_{\ed}\eps^{-1}\H{q,\ell}{\gan,0}(\cd,\om).
\end{align*}
Hence with 
$$\H{q,\infty}{\gat,0}(\ed,\om):=\bigcap_{\ell\geq0}\H{q,\ell}{\gat,0}(\ed,\om),\qquad
\H{q,\infty}{\gan,0}(\cd,\om):=\bigcap_{\ell\geq0}\H{q,\ell}{\gan,0}(\cd,\om)$$
we get by the monotonicity of the Sobolev spaces the following result:

\begin{theo}[smooth pre-bases of Dirichlet/Neumann forms for the de Rham complex]
\label{theo:cohomologyinfty}
Let $(\om,\gat)$ be a bounded strong Lipschitz pair
and recall $d_{\om,\gat}^{q}$ from Remark \ref{rem:derhameps}. Then 
$$\pi_{\cd}\H{q,\infty}{\gat,0}(\ed,\om)
=\Harm{q}{\gat,\gan,\eps}(\om)
=\pi_{\ed}\eps^{-1}\H{q,\infty}{\gan,0}(\cd,\om).$$
Moreover, there exists a smooth $\ed$-\emph{pre-basis}
and a smooth $\cd$-\emph{pre-basis} of $\Harm{q}{\gat,\gan,\eps}(\om)$, i.e.,
there are linear independent smooth forms 
\begin{align*}
\B{q}{\ed,\gat}(\om)
&:=\{B_{\ed,\gat,\ell}^{q}\}_{\ell=1}^{d_{\om,\gat}^{q}}
\subset\H{q,\infty}{\gat,0}(\ed,\om),
&
\B{q}{\cd,\gan}(\om)
&:=\{B_{\cd,\gan,\ell}^{q}\}_{\ell=1}^{d_{\om,\gat}^{q}}
\subset\H{q,\infty}{\gan,0}(\cd,\om)
\end{align*}
such that $\pi_{\cd}\B{q}{\ed,\gat}(\om)$ and $\pi_{\ed}\eps^{-1}\B{q}{\cd,\gan}(\om)$
are both bases of $\Harm{q}{\gat,\gan,\eps}(\om)$. In particular,
$$\Lin\pi_{\cd}\B{q}{\ed,\gat}(\om)
=\Harm{q}{\gat,\gan,\eps}(\om)
=\Lin\pi_{\ed}\eps^{-1}\B{q}{\cd,\gan}(\om).$$
\end{theo}

Note that $(1-\pi_{\cd})$ and $(1-\pi_{\ed})$ are the $\L{q,2}{\eps}(\om)$-orthonormal projectors onto 
$\ed\H{q-1,0}{\gat}(\ed,\om)$ and $\eps^{-1}\cd\H{q+1,0}{\gan}(\cd,\om)$, respectively, i.e.,
\begin{align*}
(1-\pi_{\cd}):\L{q,2}{\eps}(\om)&\to\ed\H{q-1,0}{\gat}(\ed,\om),
&
(1-\pi_{\ed}):\L{q,2}{\eps}(\om)&\to\eps^{-1}\cd\H{q+1,0}{\gan}(\cd,\om).
\end{align*}
Then by \eqref{helmcoho1} and Corollary \ref{cor:highorderregdecoedpbcLip},
cf.~Theorem \ref{highorderregpotedpbcLip} (i), we have
\begin{align}
\begin{aligned}
\label{helmcoho2}
\H{q,0}{\gat,0}(\ed,\om)
&=\ed\H{q-1,0}{\gat}(\ed,\om)
\oplus_{\L{q,2}{\eps}(\om)}
\Harm{q}{\gat,\gan,\eps}(\om)\\
&=\ed\H{q-1,0}{\gat}(\ed,\om)
\oplus_{\L{q,2}{\eps}(\om)}
\Lin\pi_{\cd}\B{q}{\ed,\gat}(\om)\\
&=\ed\H{q-1,0}{\gat}(\ed,\om)
+(\pi_{\cd}-1)\Lin\B{q}{\ed,\gat}(\om)
+\Lin\B{q}{\ed,\gat}(\om)\\
&=\ed\H{q-1,0}{\gat}(\ed,\om)
+\Lin\B{q}{\ed,\gat}(\om),\\
\H{q,k}{\gat,0}(\ed,\om)
&=\ed\H{q-1,0}{\gat}(\ed,\om)\cap\H{q,k}{\gat,0}(\ed,\om)
+\Lin\B{q}{\ed,\gat}(\om),\\
&=\ed\H{q-1,k+1}{\gat}(\om)
+\Lin\B{q}{\ed,\gat}(\om).
\end{aligned}
\end{align}

\begin{theo}[higher order bounded regular direct decompositions for the de Rham complex]
\label{highorderregdecopbcLipinfty}
Let $(\om,\gat)$ be a bounded strong Lipschitz pair and let $k\geq0$. 
Then the bounded regular direct decompositions
\begin{align*}
\H{q,k}{\gat}(\ed,\om)
&=R(\widetilde\PotQ_{\ed,\gat,1}^{q,k})
\dotplus\H{q,k}{\gat,0}(\ed,\om),
&
\H{q,k}{\gat,0}(\ed,\om)
&=\ed\H{q-1,k+1}{\gat}(\om)
\dotplus\Lin\B{q}{\ed,\gat}(\om),\\
\H{q,k}{\gan}(\cd,\om)
&=R(\widetilde\PotQ_{\cd,\gan,1}^{q,k})
\dotplus\H{q,k}{\gan,0}(\cd,\om),
&
\H{q,k}{\gan,0}(\cd,\om)
&=\cd\H{q+1,k+1}{\gan}(\om)
\dotplus\Lin\B{q}{\cd,\gan}(\om)
\end{align*}
hold. Note that 
$R(\widetilde\PotQ_{\ed,\gat,1}^{q,k})\subset\H{q,k+1}{\gat}(\om)$
and $R(\widetilde\PotQ_{\cd,\gan,1}^{q,k})\subset\H{q,k+1}{\gan}(\om)$.
In particular, for $k=0$
\begin{align*}
\H{q,0}{\gat}(\ed,\om)
&=R(\widetilde\PotQ_{\ed,\gat,1}^{q,0})
\dotplus\H{q,0}{\gat,0}(\ed,\om),
&
\H{q,0}{\gat,0}(\ed,\om)
&=\ed\H{q-1,1}{\gat}(\om)
\dotplus\Lin\B{q}{\ed,\gat}(\om)\\
&&
&=\ed\H{q-1,1}{\gat}(\om)
\oplus_{\L{q,2}{\eps}(\om)}
\Harm{q}{\gat,\gan,\eps}(\om),\\
\H{q,0}{\gan}(\cd,\om)
&=R(\widetilde\PotQ_{\cd,\gan,1}^{q,0})
\dotplus\H{q,0}{\gan,0}(\cd,\om),
&
\eps^{-1}\H{q,0}{\gan,0}(\cd,\om)
&=\eps^{-1}\cd\H{q+1,1}{\gan}(\om)
\dotplus\eps^{-1}\Lin\B{q}{\cd,\gan}(\om)\\
&&
&=\eps^{-1}\cd\H{q+1,1}{\gan}(\om)
\oplus_{\L{q,2}{\eps}(\om)}
\Harm{q}{\gat,\gan,\eps}(\om)
\end{align*}
as well as
\begin{align*}
\L{q,2}{\eps}(\om)
&=\H{q,0}{\gat,0}(\ed,\om)
\oplus_{\L{q,2}{\eps}(\om)}
\eps^{-1}\cd\H{q+1,1}{\gan}(\om)\\
&=\ed\H{q-1,1}{\gat}(\om)
\oplus_{\L{q,2}{\eps}(\om)}
\eps^{-1}\H{q,0}{\gan,0}(\cd,\om).
\end{align*}
\end{theo}

\begin{proof}
Theorem \ref{highorderregpotedpbcLip} (ii) and \eqref{helmcoho2} show
$$\H{q,k}{\gat}(\ed,\om)
=R(\widetilde\PotQ_{\ed,\gat,1}^{q,k})
\dotplus\H{q,k}{\gat,0}(\ed,\om),\qquad
\H{q,k}{\gat,0}(\ed,\om)
=\ed\H{q-1,k+1}{\gat}(\om)
+\Lin\B{q}{\ed,\gat}(\om).$$
To prove the directness, let
$$\sum_{\ell=1}^{d_{\om,\gat}^{q}}\lambda_{\ell}B_{\ed,\gat,\ell}^{q}
\in\ed\H{q-1,k+1}{\gat}(\om)\cap\Lin\B{q}{\ed,\gat}(\om).$$
Then $0=\sum_{\ell}\lambda_{\ell}\pi_{\cd}B_{\ed,\gat,\ell}^{q}\in\Lin\pi_{\cd}\B{q}{\ed,\gat}(\om)$
and hence $\lambda_{\ell}=0$ for all $\ell$
as $\pi_{\cd}\B{q}{\ed,\gat}(\om)$ is a basis of 
$\Harm{q}{\gat,\gan,\eps}(\om)$
by Theorem \ref{theo:cohomologyinfty}.
Concerning the boundedness of the decompositions, let 
$$\H{q,k}{\gat,0}(\ed,\om)\ni E=\ed H+B,\qquad
H\in\H{q-1,k+1}{\gat}(\om),\quad
B\in\Lin\B{q}{\ed,\gat}(\om).$$
Then we have by Theorem \ref{highorderregpotedpbcLip} (i)
$\ed H\in R(\ed_{\gat}^{q-1,k})$ and $E_{\ed}:=\PotP_{\ed,\gat}^{q,k}\ed H\in\H{q-1,k+1}{\gat}(\om)$ solves
$\ed E_{\ed}=\ed H$ with $\norm{E_{\ed}}_{\H{q-1,k+1}{}(\om)}\leq c\norm{\ed H}_{\H{q,k}{}(\om)}$.
Therefore, 
$$\norm{E_{\ed}}_{\H{q-1,k+1}{}(\om)}
+\norm{B}_{\H{q,k}{}(\om)}
\leq c\big(\norm{\ed H}_{\H{q,k}{}(\om)}
+\norm{B}_{\H{q,k}{}(\om)}\big)
\leq c\big(\norm{E}_{\H{q,k}{}(\om)}
+\norm{B}_{\H{q,k}{}(\om)}\big).$$
Note that the mapping
$$I_{\mcH}:\Lin\B{q}{\ed,\gat}(\om)\to\Lin\pi_{\cd}\B{q}{\ed,\gat}(\om)=\Harm{q}{\gat,\gan,\eps}(\om);
B_{\ed,\gat,\ell}^{q}\mapsto\pi_{\cd}B_{\ed,\gat,\ell}^{q}$$
is a topological isomorphism 
(between finite dimensional spaces and with arbitrary norms). Thus 
$$\norm{B}_{\H{q,k}{}(\om)}
\leq c\norm{B}_{\L{q,2}{}(\om)}
\leq c\norm{\pi_{\cd}B}_{\L{q,2}{}(\om)}
=c\norm{\pi_{\cd}E}_{\L{q,2}{}(\om)}
\leq c\norm{E}_{\L{q,2}{}(\om)}
\leq c\norm{E}_{\H{q,k}{}(\om)}.$$
Finally, we see $E=\ed E_{\ed}+B\in\ed\H{q-1,k+1}{\gat}(\om)+\Lin\B{q}{\ed,\gat}(\om)$
and 
$$\norm{E_{\ed}}_{\H{q-1,k+1}{}(\om)}
+\norm{B}_{\H{q,k}{}(\om)}
\leq c\norm{E}_{\H{q,k}{}(\om)}.$$
Hodge $\star$-duality yields the other assertions.
\end{proof}

\begin{rem}[higher order bounded regular direct decompositions for the de Rham complex]
\label{rem:highorderregdecopbcLipinfty}
Note that by Theorem \ref{highorderregdecopbcLipinfty} we have, e.g.,
\begin{align*}
\H{q,k}{\gat}(\ed,\om)
&=R(\widetilde\PotQ_{\ed,\gat,1}^{q,k})
\dotplus\Lin\B{q}{\ed,\gat}(\om)
\dotplus\ed\H{q-1,k+1}{\gat}(\om)
=\H{q,k+1}{\gat}(\om)
+\ed\H{q-1,k+1}{\gat}(\om)
\end{align*}
with bounded linear regular direct decomposition operators
\begin{align*}
\widehat\PotQ_{\ed,\gat,1}^{q,k}:\H{q,k}{\gat}(\ed,\om)&\to R(\widetilde\PotQ_{\ed,\gat,1}^{q,k}),
&
R(\widetilde\PotQ_{\ed,\gat,1}^{q,k})&\subset\H{q,k+1}{\gat}(\om),\\
\widehat\PotQ_{\ed,\gat,\infty}^{q,k}:\H{q,k}{\gat}(\ed,\om)&\to\Lin\B{q}{\ed,\gat}(\om),
&
\B{q}{\ed,\gat}(\om)\subset\H{q,\infty}{\gat,0}(\ed,\om)&\subset\H{q,k+1}{\gat}(\om),\\
\widehat\PotQ_{\ed,\gat,0}^{q,k}:\H{q,k}{\gat}(\ed,\om)&\to
\H{q-1,k+1}{\gat}(\om)
\end{align*}
satisfying 
$\widehat\PotQ_{\ed,\gat,1}^{q,k}
+\widehat\PotQ_{\ed,\gat,\infty}^{q,k}
+\ed\widehat\PotQ_{\ed,\gat,0}^{q,k}
=\id_{\H{q,k}{\gat}(\ed,\om)}$.
A closer inspection of the latter proof 
allows for a more precise description of these bounded decomposition operators.

For this, let $E\in\H{q,k}{\gat}(\ed,\om)$. 
According to Theorem \ref{highorderregpotedpbcLip} and 
Remark \ref{highorderregpotedpbcLiprem} we decompose 
\begin{align*}
E&=E_{R}+E_{N}
\in R(\widetilde\PotQ_{\ed,\gat,1}^{q,k})\dotplus R(\widetilde\PotN_{\ed,\gat}^{q,k}),
&
R(\widetilde\PotN_{\ed,\gat}^{q,k})=\H{q,k}{\gat,0}(\ed,\om)=N(\ed_{\gat}^{q,k}),
\end{align*}
with $E_{R}=\widetilde\PotQ_{\ed,\gat,1}^{q,k}E$
and $E_{N}=\widetilde\PotN_{\ed,\gat}^{q,k}E$.
By Theorem \ref{highorderregdecopbcLipinfty} we further decompose
\begin{align*}
\H{q,k}{\gat,0}(\ed,\om)\ni E_{N}
=\ed E_{\ed}+B
\in\ed\H{q-1,k+1}{\gat}(\om)
\dotplus\Lin\B{q}{\ed,\gat}(\om).
\end{align*}
Then 
$\pi_{\cd}E_{N}
=\pi_{\cd}B\in\Harm{q}{\gat,\gan,\eps}(\om)$
and thus 
$B=I_{\mcH}^{-1}\pi_{\cd}B
=I_{\mcH}^{-1}\pi_{\cd}E_{N}\in\Lin\B{q}{\ed,\gat}(\om)$.
Therefore, 
$E_{\ed}=\PotP_{\ed,\gat}^{q,k}\ed E_{\ed}
=\PotP_{\ed,\gat}^{q,k}(E_{N}-B)
=\PotP_{\ed,\gat}^{q,k}(1-I_{\mcH}^{-1}\pi_{\cd})E_{N}$.
Finally, we see
\begin{align*}
\widehat\PotQ_{\ed,\gat,1}^{q,k}
&=\widetilde\PotQ_{\ed,\gat,1}^{q,k}
=\PotP_{\ed,\gat}^{q+1,k}\ed^{q,k}_{\gat}
=\PotQ_{\ed,\gat,1}^{q,k}(\ed_{\gat}^{q,k})_{\bot}^{-1}\ed^{q,k}_{\gat},\\
\widehat\PotQ_{\ed,\gat,\infty}^{q,k}
&=I_{\mcH}^{-1}\pi_{\cd}\widetilde\PotN_{\ed,\gat}^{q,k}
=I_{\mcH}^{-1}\pi_{\cd}(1-\widetilde\PotQ_{\ed,\gat,1}^{q,k}),\\
\widehat\PotQ_{\ed,\gat,0}^{q,k}
&=\PotP_{\ed,\gat}^{q,k}(1-I_{\mcH}^{-1}\pi_{\cd})\widetilde\PotN_{\ed,\gat}^{q,k}
=\PotP_{\ed,\gat}^{q,k}(1-I_{\mcH}^{-1}\pi_{\cd})(1-\widetilde\PotQ_{\ed,\gat,1}^{q,k}).
\end{align*}
\end{rem}

\begin{theo}[alternative Dirichlet/Neumann projections for the de Rham complex]
\label{Bcoho1}
Let $(\om,\gat)$ be a bounded strong Lipschitz pair. Then
\begin{align*}
\Harm{q}{\gat,\gan,\eps}(\om)\cap\B{q}{\ed,\gat}(\om)^{\bot_{\L{q,2}{\eps}(\om)}}
&=\{0\},
&
\eps^{-1}\H{q,0}{\gan,0}(\cd,\om)\cap\B{q}{\ed,\gat}(\om)^{\bot_{\L{q,2}{\eps}(\om)}}
&=\eps^{-1}\cd\H{q+1,0}{\gan}(\cd,\om),\\
\Harm{q}{\gat,\gan,\eps}(\om)\cap\B{q}{\cd,\gan}(\om)^{\bot_{\L{q,2}{}(\om)}}
&=\{0\},
&
\H{q,0}{\gat,0}(\ed,\om)\cap\B{q}{\cd,\gan}(\om)^{\bot_{\L{q,2}{}(\om)}}
&=\ed\H{q-1,0}{\gat}(\ed,\om).
\end{align*}
\end{theo}

\begin{proof}
For $H\in\Harm{q}{\gat,\gan,\eps}(\om)\cap\B{q}{\ed,\gat}(\om)^{\bot_{\L{q,2}{\eps}(\om)}}$
we have
\begin{align*}
0=\scp{H}{B_{\ed,\gat,\ell}^{q}}_{\L{q,2}{\eps}(\om)}
=\scp{\pi_{\cd}H}{B_{\ed,\gat,\ell}^{q}}_{\L{q,2}{\eps}(\om)}
=\scp{H}{\pi_{\cd}B_{\ed,\gat,\ell}^{q}}_{\L{q,2}{\eps}(\om)}
\end{align*}
and hence $H=0$ by Theorem \ref{theo:cohomologyinfty}.
Analogously, we see for 
$H\in\Harm{q}{\gat,\gan,\eps}(\om)\cap\B{q}{\cd,\gan}(\om)^{\bot_{\L{q,2}{}(\om)}}$
\begin{align*}
0=\scp{H}{B_{\cd,\gan,\ell}^{q}}_{\L{q,2}{}(\om)}
=\scp{\pi_{\ed}H}{\eps^{-1}B_{\cd,\gan,\ell}^{q}}_{\L{q,2}{\eps}(\om)}
=\scp{H}{\pi_{\ed}\eps^{-1}B_{\cd,\gan,\ell}^{q}}_{\L{q,2}{\eps}(\om)}
\end{align*}
and thus $H=0$. It holds
\begin{align}
\label{Borthoedcd}
\eps^{-1}\cd\H{q+1,0}{\gan}(\cd,\om)\bot_{\L{q,2}{\eps}(\om)}\B{q}{\ed,\gat}(\om),\qquad
\ed\H{q-1,0}{\gat}(\ed,\om)\bot_{\L{q,2}{}(\om)}\B{q}{\cd,\gan}(\om).
\end{align}
According to \eqref{helmcoho1} we can decompose
\begin{align*}
\eps^{-1}\H{q,0}{\gan,0}(\cd,\om)
&=\eps^{-1}\cd\H{q+1,0}{\gan}(\cd,\om)
\oplus_{\L{q,2}{\eps}(\om)}
\Harm{q}{\gat,\gan,\eps}(\om),\\
\H{q,0}{\gat,0}(\ed,\om)
&=\ed\H{q-1,0}{\gat}(\ed,\om)
\oplus_{\L{q,2}{\eps}(\om)}
\Harm{q}{\gat,\gan,\eps}(\om),
\end{align*}
which shows by \eqref{Borthoedcd} the other two assertions.
\end{proof}

\begin{cor}[alternative Dirichlet/Neumann projections for the de Rham complex]
\label{cor:Bcoho1}
Let $(\om,\gat)$ be a bounded strong Lipschitz pair and let $k\geq0$. Then
\begin{align*}
\eps^{-1}\H{q,k}{\gan,0}(\cd,\om)\cap\B{q}{\ed,\gat}(\om)^{\bot_{\L{q,2}{\eps}(\om)}}
&=\eps^{-1}\cd\H{q+1,k}{\gan}(\cd,\om)
=\eps^{-1}\cd\H{q+1,k+1}{\gan}(\om),\\
\H{q,k}{\gat,0}(\ed,\om)\cap\B{q}{\cd,\gan}(\om)^{\bot_{\L{q,2}{}(\om)}}
&=\ed\H{q-1,k}{\gat}(\ed,\om)
=\ed\H{q-1,k+1}{\gat}(\om).
\end{align*}
\end{cor}

\begin{proof}
We have by Theorem \ref{Bcoho1} and Theorem \ref{highorderregpotedpbcLip} (i)
\begin{align*}
\H{q,k}{\gat,0}(\ed,\om)\cap\B{q}{\cd,\gan}(\om)^{\bot_{\L{q,2}{}(\om)}}
&=\H{q,k}{\gat}(\om)\cap\H{q,0}{\gat,0}(\ed,\om)\cap\B{q}{\cd,\gan}(\om)^{\bot_{\L{q,2}{}(\om)}}\\
&=\H{q,k}{\gat}(\om)\cap\ed\H{q-1,0}{\gat}(\ed,\om)\\
&=\ed\H{q-1,k}{\gat}(\ed,\om)
=\ed\H{q-1,k+1}{\gat}(\om).
\end{align*}
Analogously,
\begin{align*}
\eps^{-1}\H{q,k}{\gan,0}(\cd,\om)\cap\B{q}{\ed,\gat}(\om)^{\bot_{\L{q,2}{\eps}(\om)}}
&=\eps^{-1}\H{q,k}{\gan}(\om)\cap\eps^{-1}\H{q,0}{\gan,0}(\cd,\om)\cap\B{q}{\ed,\gat}(\om)^{\bot_{\L{q,2}{\eps}(\om)}}\\
&=\eps^{-1}\H{q,k}{\gan}(\om)\cap\eps^{-1}\cd\H{q+1,0}{\gan}(\cd,\om)\\
&=\eps^{-1}\cd\H{q+1,k}{\gan}(\cd,\om)
=\eps^{-1}\cd\H{q+1,k+1}{\gan}(\om),
\end{align*}
completing the proof.
\end{proof}

Theorem \ref{highorderregdecopbcLipinfty} and 
$\star\Harm{q}{\gat,\gan,\id}(\om)=\Harm{\dimom-q}{\gan,\gat,\id}(\om)$
shows the following result:

\begin{theo}[cohomology groups of the de Rham complex]
\label{Bcoho2}
Let $(\om,\gat)$ be a bounded strong Lipschitz pair. 
Then ($\cong$ means isomorphic)
\begin{align*}
N(\ed_{\gat}^{q,k})/R(\ed_{\gat}^{q-1,k})
\cong\Lin\B{q}{\ed,\gat}(\om)
\cong\Harm{q}{\gat,\gan,\eps}(\om)
\cong\Lin\B{q}{\cd,\gan}(\om)
\cong N(\cd_{\gan}^{q,k})/R(\cd_{\gan}^{q+1,k}).
\end{align*}
In particular, the dimensions of the cohomology groups
(Dirichlet/Neumann forms) are independent of $k$ and $\eps$ and it holds
\begin{align*}
d_{\om,\gat}^{q}
=\dim\big(N(\ed_{\gat}^{q,k})/R(\ed_{\gat}^{q-1,k})\big)
=\dim\big(N(\cd_{\gan}^{q,k})/R(\cd_{\gan}^{q+1,k})\big).
\end{align*}
Moroever, $d_{\om,\gat}^{q}=d_{\om,\gan}^{\dimom-q}$.
\end{theo}

\begin{rem}
\label{rem:CM2}
For the case of either no or full boundary conditions, i.e.,
$\gat=\emptyset$ or $\gat=\ga$,
related results on regular potentials, regular decompositions,
as well as cohomology groups and their dimensions,
even for real Sobolev exponents $k\in\reals$, have been proved in \cite{CM2010a}
using integral equation representations and methods.
In particular, we refer to \cite[Theorem 1.1, Theorem 4.9]{CM2010a}.
\end{rem}

\section{Vector De Rham Complex}
\label{sec:vecderham}

We reformulate the results from Section \ref{sec:derham}
in the special case $\dimom=3$ and $q\in\{0,1,2,3\}$
using vector proxies.
Recall Section \ref{sec:defsobolevvecfields}
and let $\eps$ and $\mu$ be admissible weights.
To apply the FA-ToolBox from Section \ref{sec:FA}
for the vector de Rham complex, let $\grad$, $\rot$, and $\div$
be realised as densely defined (unbounded) linear operators
\begin{align*}
\mr{\grad}_{\gat}:D(\mr{\grad}_{\gat})\subset\L{2}{}(\om)&\to\L{2}{\eps}(\om);
&
u&\mapsto\grad u,\\
\mu^{-1}\mr{\rot}_{\gat}:D(\mu^{-1}\mr{\rot}_{\gat})\subset\L{2}{\eps}(\om)&\to\L{2}{\mu}(\om);
&
E&\mapsto\mu^{-1}\rot E,\\
\mr{\div}_{\gat}\mu:D(\mr{\div}_{\gat}\mu)\subset\L{2}{\mu}(\om)&\to\L{2}{}(\om);
&
H&\mapsto\div\mu H
\end{align*}
with domains of definition
$$D(\mr{\grad}_{\gat}):=\C{\infty}{\gat}(\om),\qquad
D(\mu^{-1}\mr{\rot}_{\gat}):=\C{\infty}{\gat}(\om),\qquad
D(\mr{\div}_{\gat}\mu):=\mu^{-1}\C{\infty}{\gat}(\om)$$
satisfying the complex properties
$$\mu^{-1}\mr{\rot}_{\gat}\mr{\grad}_{\gat}\subset0,\qquad
\mr{\div}_{\gat}\mu\,\mu^{-1}\mr{\rot}_{\gat}
=\mr{\div}_{\gat}\mr{\rot}_{\gat}\subset0.$$
Then the closures 
$$\grad_{\gat}:=\ol{\mr{\grad}_{\gat}},\qquad
\mu^{-1}\rot_{\gat}:=\ol{\mu^{-1}\mr{\rot}_{\gat}},\qquad
\div_{\gat}\mu:=\ol{\mr{\div}_{\gat}\mu}$$
and Hilbert space adjoints 
$$\grad_{\gat}^{*}=\mr{\grad}_{\gat}^{*},\qquad
(\mu^{-1}\rot_{\gat})^{*}=(\mu^{-1}\mr{\rot}_{\gat})^{*},\qquad
(\div_{\gat}\mu)^{*}=(\mr{\div}_{\gat}\mu)^{*}$$
are given by
\begin{align*}
\A_{0}:=\grad_{\gat}:D(\grad_{\gat})\subset\L{2}{}(\om)&\to\L{2}{\eps}(\om);
&
u&\mapsto\grad u,\\
\A_{1}:=\mu^{-1}\rot_{\gat}:D(\mu^{-1}\rot_{\gat})\subset\L{2}{\eps}(\om)&\to\L{2}{\mu}(\om);
&
E&\mapsto\mu^{-1}\rot E,\\
\A_{2}:=\div_{\gat}\mu:D(\div_{\gat}\mu)\subset\L{2}{\mu}(\om)&\to\L{2}{}(\om);
&
H&\mapsto\div\mu H,\\
\A_{0}^{*}=\grad_{\gat}^{*}=-\div_{\gan}\eps:D(\div_{\gan}\eps)\subset\L{2}{\eps}(\om)&\to\L{2}{}(\om);
&
E&\mapsto-\div\eps E,\\
\A_{1}^{*}=(\mu^{-1}\rot_{\gat})^{*}=\eps^{-1}\rot_{\gan}:D(\eps^{-1}\rot_{\gan})\subset\L{2}{\mu}(\om)&\to\L{2}{\eps}(\om);
&
H&\mapsto\eps^{-1}\rot H,\\
\A_{2}^{*}=(\div_{\gat}\mu)^{*}=-\grad_{\gan}:D(\grad_{\gan})\subset\L{2}{}(\om)&\to\L{2}{\mu}(\om);
&
u&\mapsto-\grad u
\end{align*}
with domains of definition
\begin{align*}
D(\A_{0})=D(\grad_{\gat})&=\H{1}{\gat}(\om),
&
D(\A_{0}^{*})=D(\div_{\gan}\eps)&=\eps^{-1}\H{}{\gan}(\div,\om),\\
D(\A_{1})=D(\mu^{-1}\rot_{\gat})&=\H{}{\gat}(\rot,\om),
&
D(\A_{1}^{*})=D(\eps^{-1}\rot_{\gan})&=\H{}{\gan}(\rot,\om),\\
D(\A_{2})=D(\div_{\gat}\mu)&=\mu^{-1}\H{}{\gat}(\div,\om),
&
D(\A_{2}^{*})=D(\grad_{\gan})&=\H{1}{\gan}(\om).
\end{align*}
As in Section \ref{sec:derham}, indeed 
the domains of definition of the adjoints are given as stated.

\begin{rem}
\label{rem:weakeqstrongderhamvec}
Note that by definition the adjoints are given by 
\begin{align*}
\grad_{\gat}^{*}=\mr{\grad}_{\gat}^{*}=-\bs{\div}_{\gan}\eps:
D(\bs{\div}_{\gan}\eps)\subset\L{2}{\eps}(\om)&\to\L{2}{}(\om),\\
(\mu^{-1}\rot_{\gat})^{*}=(\mu^{-1}\mr{\rot}_{\gat})^{*}=\eps^{-1}\bs{\rot}_{\gan}:
D(\eps^{-1}\bs{\rot}_{\gan})\subset\L{2}{\mu}(\om)&\to\L{2}{\eps}(\om),\\
(\div_{\gat}\mu)^{*}=(\mr{\div}_{\gat}\mu)^{*}=-\bs{\grad}_{\gan}:
D(\bs{\grad}_{\gan})\subset\L{2}{}(\om)&\to\L{2}{\mu}(\om)
\end{align*}
with domains of definition
\begin{align*}
D(\bs{\div}_{\gan}\eps)&=\eps^{-1}\bH{}{\gan}(\div,\om),\qquad
D(\eps^{-1}\bs{\rot}_{\gan})&=\bH{}{\gan}(\rot,\om),\qquad
D(\bs{\grad}_{\gan})&=\bH{1}{\gan}(\om).
\end{align*}
Lemma \ref{lem:wsbcderhamvec} (weak and strong boundary conditions coincide)
shows indeed that $\bs{\div}_{\gan}\eps=\div_{\gan}\eps$, 
$\eps^{-1}\bs{\rot}_{\gan}=\eps^{-1}\rot_{\gan}$, and
$\bs{\grad}_{\gan}=\grad_{\gan}$, in particular 
\begin{align*}
D(\bs{\div}_{\gan}\eps)
=\eps^{-1}\bH{}{\gan}(\div,\om)
&=\eps^{-1}\H{}{\gan}(\div,\om)
=D(\div_{\gan}\eps),\\
D(\eps^{-1}\bs{\rot}_{\gan})
=\bH{}{\gan}(\rot,\om)
&=\H{}{\gan}(\rot,\om)
=D(\eps^{-1}\rot_{\gan}),\\
D(\bs{\grad}_{\gan})
=\bH{1}{\gan}(\om)
&=\H{1}{\gan}(\om)
=D(\grad_{\gan}).
\end{align*}
\end{rem}

By definition we have densely defined and closed (unbounded) linear operators 
defining three dual pairs 
\begin{align*}
\big(\grad_{\gat},(\grad_{\gat})^{*}\big)
&=(\grad_{\gat},-\div_{\gan}\eps),\\
\big(\mu^{-1}\rot_{\gat},(\mu^{-1}\rot_{\gat})^{*}\big)
&=(\mu^{-1}\rot_{\gat},\eps^{-1}\rot_{\gan}),\\
\big(\div_{\gat}\mu,(\div_{\gat}\mu)^{*}\big)
&=(\div_{\gat}\mu,-\grad_{\gan}).
\end{align*}
Remark \ref{remhilcom} and Remark \ref{remhilcomclosure} show the complex properties 
\begin{align*}
\mu^{-1}\rot_{\gat}\grad_{\gat}&\subset0,
&
\div_{\gat}\mu\,\mu^{-1}\rot_{\gat}
=\div_{\gat}\rot_{\gat}&\subset0,\\
-\div_{\gan}\eps\,\eps^{-1}\rot_{\gan}
=-\div_{\gan}\rot_{\gan}&\subset0,
&
-\eps^{-1}\rot_{\gan}\grad_{\gan}&\subset0.
\end{align*}
The long primal and dual vector de Rham Hilbert complex \eqref{lhcomplex2},
cf.~\eqref{derhamcomplex1}, reads
\begin{equation}
\label{vecderhamcomplex1}
\footnotesize
\def\arrowlength{15ex}
\def\arrowdistance{.8}
\begin{tikzcd}[column sep=\arrowlength]
\reals_{\gat}
\arrow[r, rightarrow, shift left=\arrowdistance, "\iota_{\reals_{\gat}}"] 
\arrow[r, leftarrow, shift right=\arrowdistance, "\pi_{\reals_{\gat}}"']
& 
[-2em]
\L{2}{}(\om) 
\ar[r, rightarrow, shift left=\arrowdistance, "\grad_{\gat}"] 
\ar[r, leftarrow, shift right=\arrowdistance, "-\div_{\gan}\eps"']
&
[-1em]
\L{2}{\eps}(\om) 
\ar[r, rightarrow, shift left=\arrowdistance, "\mu^{-1}\rot_{\gat}"] 
\ar[r, leftarrow, shift right=\arrowdistance, "\eps^{-1}\rot_{\gan}"']
& 
[0em]
\L{2}{\mu}(\om) 
\arrow[r, rightarrow, shift left=\arrowdistance, "\div_{\gat}\mu"] 
\arrow[r, leftarrow, shift right=\arrowdistance, "-\grad_{\gan}"']
& 
[-1em]
\L{2}{}(\om) 
\arrow[r, rightarrow, shift left=\arrowdistance, "\pi_{\reals_{\gan}}"] 
\arrow[r, leftarrow, shift right=\arrowdistance, "\iota_{\reals_{\gan}}"']
&
[-2em]
\reals_{\gan}
\end{tikzcd}
\end{equation}
with the complex properties 
\begin{align*}
R(\iota_{\reals_{\gat}})&=N(\grad_{\gat})=\reals_{\gat},
&
\ol{R(\div_{\gan}\eps)}&=(\reals_{\gat})^{\bot_{\L{2}{}(\om)}},\\
R(\grad_{\gat})&\subset N(\mu^{-1}\rot_{\gat}),
&
R(\eps^{-1}\rot_{\gan})&\subset N(\div_{\gan}\eps),\\
R(\mu^{-1}\rot_{\gat})&\subset N(\div_{\gat}\mu),
&
R(\grad_{\gan})&\subset N(\eps^{-1}\rot_{\gan}),\\
\ol{R(\div_{\gat}\mu)}&=(\reals_{\gan})^{\bot_{\L{2}{}(\om)}},
&
R(\iota_{\reals_{\gan}})&=N(\grad_{\gan})=\reals_{\gan}.
\end{align*}
Recalling Remark \ref{firstadjoints}, we note that actually
$\iota_{\reals_{\gat}}\iota_{\reals_{\gat}}^{*}=\pi_{\reals_{\gat}}$ and 
$\iota_{\reals_{\gan}}\iota_{\reals_{\gan}}^{*}=\pi_{\reals_{\gan}}$
as self-adjoint projections on $\L{2}{}(\om)$.

Similar to \eqref{vecderhamcomplex1} (for simplicity let $\eps=\mu=1$)
we investigate the higher order de Rham complex
\begin{equation*}
\def\arrowlength{7ex}
\def\arrowdistance{0}
\begin{tikzcd}[column sep=\arrowlength]
\reals_{\gat}
\arrow[r, rightarrow, shift left=\arrowdistance, "\iota_{\reals_{\gat}}"] 
& 
\H{k}{\gat}(\om)
\arrow[r, rightarrow, shift left=\arrowdistance, "\grad^{k}_{\gat}"] 
& 
\H{k}{\gat}(\om)
\ar[r, rightarrow, shift left=\arrowdistance, "\rot^{k}_{\gat}"] 
& 
\H{k}{\gat}(\om)
\arrow[r, rightarrow, shift left=\arrowdistance, "\div^{k}_{\gat}"] 
& 
\H{k}{\gat}(\om)
\arrow[r, rightarrow, shift left=\arrowdistance, "\pi_{\reals_{\gan}}"] 
&
\reals_{\gan}
\end{tikzcd}
\end{equation*}
as well. More precisely, we consider the densely defined and closed linear operators
\begin{align*}
\grad_{\gat}^{k}:D(\grad_{\gat}^{k})\subset\H{k}{\gat}(\om)&\to\H{k}{\gat}(\om);
\,u\mapsto\grad u,
&
D(\grad_{\gat}^{k})&:=\H{k}{\gat}(\grad,\om)=\H{k+1}{\gat}(\om),\\
\rot_{\gat}^{k}:D(\rot_{\gat}^{k})\subset\H{k}{\gat}(\om)&\to\H{k}{\gat}(\om);
\,E\mapsto\rot E,
&
D(\rot_{\gat}^{k})&:=\H{k}{\gat}(\rot,\om),\\
\div_{\gat}^{k}:D(\div_{\gat}^{k})\subset\H{k}{\gat}(\om)&\to\H{k}{\gat}(\om);
\,H\mapsto\div H,
&
D(\div_{\gat}^{k})&:=\H{k}{\gat}(\div,\om).
\end{align*}
Note that the complex properties 
$R(\grad_{\gat}^{k})\subset N(\rot_{\gat}^{k})$ and
$R(\rot_{\gat}^{k})\subset N(\div_{\gat}^{k})$ 
hold.

\subsection{Regular Potentials and Decompositions}

For $\ed\in\{\grad,\rot,\div\}$
Lemma \ref{lem:highorderregdecoedpbcLip}, Corollary \ref{cor:highorderregdecoedpbcLip},
Theorem \ref{highorderregpotedpbcLip}, and Remark \ref{highorderregpotedpbcLiprem}
read as follows.

\begin{theo}[higher order bounded regular potentials and decompositions 
for the vector de Rham complex with partial boundary condition]
\label{highorderregpotpbcLipderhamvec}
Let $(\om,\gat)$ be a bounded strong Lipschitz pair and let $k\geq0$. Then:
\begin{itemize}
\item[\bf(i)]
The bounded regular decompositions
\begin{align*}
\bH{k}{\gat}(\rot,\om)
=\H{k}{\gat}(\rot,\om)
&=\H{k+1}{\gat}(\om)
+\grad\H{k+1}{\gat}(\om),\\
\bH{k}{\gat}(\div,\om)
=\H{k}{\gat}(\div,\om)
&=\H{k+1}{\gat}(\om)
+\rot\H{k+1}{\gat}(\om)
\end{align*}
hold with bounded linear regular decomposition operators
\begin{align*}
\PotQ_{\rot,\gat,1}^{k}:\H{k}{\gat}(\rot,\om)&\to\H{k+1}{\gat}(\om),
&
\PotQ_{\rot,\gat,0}^{k}:\H{k}{\gat}(\rot,\om)&\to\H{k+1}{\gat}(\om),\\
\PotQ_{\div,\gat,1}^{k}:\H{k}{\gat}(\div,\om)&\to\H{k+1}{\gat}(\om),
&
\PotQ_{\div,\gat,0}^{k}:\H{k}{\gat}(\div,\om)&\to\H{k+1}{\gat}(\om)
\end{align*}
satisfying $\PotQ_{\rot,\gat,1}^{k}+\grad\PotQ_{\rot,\gat,0}^{k}=\id_{\H{k}{\gat}(\rot,\om)}$
and $\PotQ_{\div,\gat,1}^{k}+\rot\PotQ_{\div,\gat,0}^{k}=\id_{\H{k}{\gat}(\div,\om)}$.
In particular, weak and strong boundary conditions coincide.
It holds $\rot\PotQ_{\rot,\gat,1}^{k}=\rot^{k}_{\gat}$
and thus $\H{k}{\gat,0}(\rot,\om)$ is invariant under $\PotQ_{\rot,\gat,1}^{k}$.
Analogously, $\div\PotQ_{\div,\gat,1}^{k}=\div^{k}_{\gat}$
and thus $\H{k}{\gat,0}(\div,\om)$ is invariant under $\PotQ_{\div,\gat,1}^{k}$.
\item[\bf(ii)]
The regular potential representations
{\small
\begin{align*}
R(\grad^{k}_{\gat})
=\grad\H{k+1}{\gat}(\om)
&=\H{k}{\gat,0}(\rot,\om)
\cap\Harm{}{\gat,\gan,\eps}(\om)^{\bot_{\L{2}{\eps}(\om)}}
=\H{k}{\gat}(\om)\cap R(\grad_{\gat}),\\
R(\rot^{k}_{\gat})
=\rot\H{k}{\gat}(\rot,\om)
=\rot\H{k+1}{\gat}(\om)
&=\H{k}{\gat,0}(\div,\om)
\cap\Harm{}{\gan,\gat,\eps}(\om)^{\bot_{\L{2}{}(\om)}}
=\H{k}{\gat}(\om)\cap R(\rot_{\gat}),\\
R(\div^{k}_{\gat})
=\div\H{k}{\gat}(\div,\om)
=\div\H{k+1}{\gat}(\om)
&=\H{k}{\gat}(\om)
\cap(\reals_{\gan})^{\bot_{\L{2}{}(\om)}}
=\H{k}{\gat}(\om)\cap R(\div_{\gat})
\end{align*}
}
hold. In particular, these spaces are closed subspaces of 
$\H{k}{\emptyset}(\om)=\H{k}{}(\om)$.
\item[\bf(iii)]
There exist bounded linear regular potential operators
\begin{align*}
\PotP_{\grad,\gat}^{k}:=(\grad^{k}_{\gat})_{\bot}^{-1}:
\H{k}{\gat,0}(\rot,\om)
\cap\Harm{}{\gat,\gan,\eps}(\om)^{\bot_{\L{2}{\eps}(\om)}}
&\To\H{k+1}{\gat}(\om),\\
\PotP_{\rot,\gat}^{k}:=\PotQ_{\rot,\gat,1}^{k}(\rot^{k}_{\gat})_{\bot}^{-1}:
\H{k}{\gat,0}(\div,\om)
\cap\Harm{}{\gan,\gat,\eps}(\om)^{\bot_{\L{2}{}(\om)}}
&\To\H{k+1}{\gat}(\om),\\
\PotP_{\div,\gat}^{k}:=\PotQ_{\div,\gat,1}^{k}(\div^{k}_{\gat})_{\bot}^{-1}:
\H{k}{\gat}(\om)
\cap(\reals_{\gan})^{\bot_{\L{2}{}(\om)}}
&\To\H{k+1}{\gat}(\om),
\end{align*}
such that 
\begin{align*}
\grad\PotP_{\grad,\gat}^{k}&=\id|_{\H{k}{\gat,0}(\rot,\om)\cap\Harm{}{\gat,\gan,\eps}(\om)^{\bot_{\L{2}{\eps}(\om)}}},\\
\rot\PotP_{\rot,\gat}^{k}&=\id|_{\H{k}{\gat,0}(\div,\om)\cap\Harm{}{\gan,\gat,\eps}(\om)^{\bot_{\L{2}{}(\om)}}},\\
\div\PotP_{\div,\gat}^{k}&=\id|_{\H{k}{\gat}(\om)\cap(\reals_{\gan})^{\bot_{\L{2}{}(\om)}}}.
\end{align*}
In particular, all potentials in (ii)
can be chosen such that they depend continuously on the data.
$\PotP_{\grad,\gat}^{k}$, $\PotP_{\rot,\gat}^{k}$, and $\PotP_{\div,\gat}^{k}$
are right inverses of $\grad$, $\rot$, and $\div$, respectively.
\item[\bf(iv)]
The bounded regular decompositions
\begin{align*}
\H{k}{\gat}(\rot,\om)
&=\H{k+1}{\gat}(\om)
+\H{k}{\gat,0}(\rot,\om)
=\H{k+1}{\gat}(\om)
+\grad\H{k+1}{\gat}(\om)\\
&=R(\widetilde\PotQ_{\rot,\gat,1}^{k})
\dotplus\H{k}{\gat,0}(\rot,\om)
=R(\widetilde\PotQ_{\rot,\gat,1}^{k})
\dotplus R(\widetilde\PotN_{\rot,\gat}^{k}),\\
\H{k}{\gat}(\div,\om)
&=\H{k+1}{\gat}(\om)
+\H{k}{\gat,0}(\div,\om)
=\H{k+1}{\gat}(\om)
+\rot\H{k+1}{\gat}(\om)\\
&=R(\widetilde\PotQ_{\div,\gat,1}^{k})
\dotplus\H{k}{\gat,0}(\div,\om)
=R(\widetilde\PotQ_{\div,\gat,1}^{k})
\dotplus R(\widetilde\PotN_{\div,\gat}^{k})
\end{align*}
hold with bounded linear regular decomposition operators
\begin{align*}
\widetilde\PotQ_{\rot,\gat,1}^{k}:=\PotP_{\rot,\gat}^{k}\rot^{k}_{\gat}:
\H{k}{\gat}(\rot,\om)&\to\H{k+1}{\gat}(\om),
&
\widetilde\PotN_{\rot,\gat}^{k}:\H{k}{\gat}(\rot,\om)&\to\H{k}{\gat,0}(\rot,\om),\\
\widetilde\PotQ_{\div,\gat,1}^{k}:=\PotP_{\div,\gat}^{k}\div^{k}_{\gat}:
\H{k}{\gat}(\div,\om)&\to\H{k+1}{\gat}(\om),
&
\widetilde\PotN_{\div,\gat}^{k}:\H{k}{\gat}(\div,\om)&\to\H{k}{\gat,0}(\div,\om)
\end{align*}
satisfying $\widetilde\PotQ_{\rot,\gat,1}^{k}+\widetilde\PotN_{\rot,\gat}^{k}=\id_{\H{k}{\gat}(\rot,\om)}$
and $\widetilde\PotQ_{\div,\gat,1}^{k}+\widetilde\PotN_{\div,\gat}^{k}=\id_{\H{k}{\gat}(\div,\om)}$.
It holds $\rot\widetilde\PotQ_{\rot,\gat,1}^{k}=\rot\PotQ_{\rot,\gat,1}^{k}=\rot^{k}_{\gat}$
and thus $\H{k}{\gat,0}(\rot,\om)$ is invariant under $\PotQ_{\rot,\gat,1}^{k}$ 
and $\widetilde\PotQ_{\rot,\gat,1}^{k}$.
Analogously, $\div\widetilde\PotQ_{\div,\gat,1}^{k}=\div\PotQ_{\div,\gat,1}^{k}=\div^{k}_{\gat}$
and thus $\H{k}{\gat,0}(\div,\om)$ is invariant under $\PotQ_{\div,\gat,1}^{k}$
and $\widetilde\PotQ_{\div,\gat,1}^{k}$. 
Moreover, we have $R(\widetilde\PotQ_{\rot,\gat,1}^{k})=R(\PotP_{\rot,\gat}^{k})$ and 
$\widetilde\PotQ_{\rot,\gat,1}^{k}
=\PotQ_{\rot,\gat,1}^{k}(\rot^{k}_{\gat})_{\bot}^{-1}\rot^{k}_{\gat}$.
Hence $\widetilde\PotQ_{\rot,\gat,1}^{k}|_{D((\rot^{k}_{\gat})_{\bot})}
=\PotQ_{\rot,\gat,1}^{k}|_{D((\rot^{k}_{\gat})_{\bot})}$
and thus $\widetilde\PotQ_{\rot,\gat,1}^{k}$ may differ from $\PotQ_{\rot,\gat,1}^{k}$
only on $\H{k}{\gat,0}(\rot,\om)$.
Analogously, it holds $R(\widetilde\PotQ_{\div,\gat,1}^{k})=R(\PotP_{\div,\gat}^{k})$ and 
$\widetilde\PotQ_{\div,\gat,1}^{k}
=\PotQ_{\div,\gat,1}^{k}(\div^{k}_{\gat})_{\bot}^{-1}\div^{k}_{\gat}$.
Hence we have that $\widetilde\PotQ_{\div,\gat,1}^{k}|_{D((\div^{k}_{\gat})_{\bot})}
=\PotQ_{\div,\gat,1}^{k}|_{D((\div^{k}_{\gat})_{\bot})}$
and thus $\widetilde\PotQ_{\div,\gat,1}^{k}$ may differ from $\PotQ_{\div,\gat,1}^{k}$
only on $\H{k}{\gat,0}(\div,\om)$.
\item[\bf(iv')]
The bounded regular kernel decompositions
$\H{k}{\gat,0}(\rot,\om)
=\H{k+1}{\gat,0}(\rot,\om)
+\grad\H{k+1}{\gat}(\om)$
and
$\H{k}{\gat,0}(\div,\om)
=\H{k+1}{\gat,0}(\div,\om)
+\rot\H{k+1}{\gat}(\om)$
hold.
\end{itemize}
\end{theo}

\begin{rem}
\label{highorderregpotpbcLipderhamvecrem}
Let us recall the bounded regular decompositions
from Theorem \ref{highorderregpotpbcLipderhamvec} (iv), e.g.,
\begin{align*}
\H{k}{\gat}(\rot,\om)
&=R(\widetilde\PotQ_{\rot,\gat,1}^{k})
\dotplus R(\widetilde\PotN_{\rot,\gat}^{k}).
\end{align*}
\begin{itemize}
\item[\bf(i)]
$\widetilde\PotQ_{\rot,\gat,1}^{k}$,
$\widetilde\PotN_{\rot,\gat}^{k}=1-\widetilde\PotQ_{\rot,\gat,1}^{k}$
are projections with 
$\widetilde\PotQ_{\rot,\gat,1}^{k}\widetilde\PotN_{\rot,\gat}^{k}
=\widetilde\PotN_{\rot,\gat}^{k}\widetilde\PotQ_{\rot,\gat,1}^{k}=0$.
\item[\bf(ii)]
For $I_{\pm}:=\widetilde\PotQ_{\rot,\gat,1}^{k}\pm\widetilde\PotN_{\rot,\gat}^{k}$
it holds $I_{+}=I_{-}^{2}=\id_{\H{k}{\gat}(\rot,\om)}$.
Therefore, $I_{+}$, $I_{-}^{2}$, as well as
$I_{-}=2\widetilde\PotQ_{\rot,\gat,1}^{k}-\id_{\H{k}{\gat}(\rot,\om)}$
are topological isomorphisms on $\H{k}{\gat}(\rot,\om)$.
\item[\bf(iii)]
There exists $c>0$ such that for all $E\in\H{k}{\gat}(\rot,\om)$
\begin{align*}
c\norm{\widetilde\PotQ_{\rot,\gat,1}^{k}E}_{\H{k+1}{}(\om)}
&\leq\norm{\rot E}_{\H{k}{}(\om)}
\leq\norm{E}_{\H{k}{}(\rot,\om)},\\
\norm{\widetilde\PotN_{\rot,\gat}^{k}E}_{\H{k}{}(\om)}
&\leq\norm{E}_{\H{k}{}(\om)}
+\norm{\widetilde\PotQ_{\rot,\gat,1}^{k}E}_{\H{k}{}(\om)}.
\end{align*}
\item[\bf(iii')]
For $E\in\H{k}{\gat,0}(\rot,\om)$ we have $\widetilde\PotQ_{\rot,\gat,1}^{k}E=0$
and $\widetilde\PotN_{\rot,\gat}^{k}E=E$.
In particular, $\widetilde\PotN_{\rot,\gat}^{k}$ is onto.
\item[\bf(iv)]
Literally, (i)-(iii') hold for $\div$ as well.
\end{itemize}
\end{rem}

\subsection{Zero Order Mini FA-ToolBox}

Theorem \ref{theo:cptemb:derham}, Theorem \ref{theo:minifatb:derham},
and Remark \ref{rem:derhameps} translate to the following results, 
cf.~\eqref{lhcomplex2} and Definition \ref{defihilcom2}
as well as \cite[Lemma 5.1, Lemma 5.2]{PW2020a}.

\begin{theo}[compact embedding for the vector de Rham complex]
\label{theo:cptemb:derhamvec}
Let $(\om,\gat)$ be a bounded strong Lipschitz pair. 
Then the embeddings 
\begin{align*}
D(\A_{0})
=\H{1}{\gat}(\om)
&\incl\L{2}{}(\om),\\
D(\A_{1})\cap D(\A_{0}^{*})
=\H{}{\gat}(\rot,\om)\cap\eps^{-1}\H{}{\gan}(\div,\om)
&\incl\L{2}{\eps}(\om),\\
D(\A_{2})\cap D(\A_{1}^{*})
=\mu^{-1}\H{}{\gat}(\div,\om)\cap\H{}{\gan}(\rot,\om)
&\incl\L{2}{\mu}(\om),\\
D(\A_{2}^{*})
=\H{1}{\gan}(\om)
&\incl\L{2}{}(\om)
\end{align*}
are compact, i.e., the long primal and dual vector de Rham Hilbert complex is compact. 
In particular, the complex is closed.
Moreover, the compactness of the embeddings 
is independent of $\eps$ and $\mu$.
\end{theo}

\begin{theo}[mini FA-ToolBox for the vector de Rham complex]
\label{theo:minifatb:derhamvec}
Let $(\om,\gat)$ be a bounded strong Lipschitz pair. Then
\begin{itemize}
\item[\bf(i)]
the ranges $R(\grad_{\gat})$, $R(\rot_{\gat})$, 
and $R(\div_{\gat})=(\reals_{\gan})^{\bot_{\L{2}{}(\om)}}$ are closed,
\item[\bf(ii)]
the inverse operators $(\grad_{\gat})_{\bot}^{-1}$, $(\mu^{-1}\rot_{\gat})_{\bot}^{-1}$
and $(\div_{\gat}\mu)_{\bot}^{-1}$ are compact,
\item[\bf(iii)]
the cohomology group $\Harm{}{\gat,\gan,\eps}(\om)
=\H{}{\gat,0}(\rot,\om)\cap\eps^{-1}\H{}{\gan,0}(\div,\om)$ has finite dimension,
which is independent of $\eps$,
\item[\bf(iv)]
the orthogonal Helmholtz-type decomposition
$$\L{2}{\eps}(\om)
=\grad\H{1}{\gat}(\om)
\oplus_{\L{2}{\eps}(\om)}
\Harm{}{\gat,\gan,\eps}(\om)
\oplus_{\L{2}{\eps}(\om)}
\eps^{-1}\rot\H{}{\gan}(\rot,\om)$$
holds,
\item[\bf(v)]
there exist $c_{\grad,\gat},c_{\rot,\gat},c_{\div,\gat}>0$ such that
\begin{align*}
\forall\,u&\in D\big((\grad_{\gat})_{\bot}\big)
&
\norm{u}_{\L{2}{}(\om)}&\leq c_{\grad,\gat}\norm{\grad u}_{\L{2}{\eps}(\om)},\\
\forall\,E&\in D\big((\div_{\gan}\eps)_{\bot}\big)
&
\norm{E}_{\L{2}{\eps}(\om)}&\leq c_{\grad,\gat}\norm{\div\eps E}_{\L{2}{}(\om)},\\
\forall\,E&\in D\big((\mu^{-1}\rot_{\gat})_{\bot}\big)
&
\norm{E}_{\L{2}{\eps}(\om)}&\leq c_{\rot,\gat}\norm{\mu^{-1}\rot E}_{\L{2}{\mu}(\om)},\\
\forall\,H&\in D\big((\eps^{-1}\rot_{\gan})_{\bot}\big)
&
\norm{H}_{\L{2}{\mu}(\om)}&\leq c_{\rot,\gat}\norm{\eps^{-1}\rot E}_{\L{2}{\eps}(\om)},\\
\forall\,H&\in D\big((\div_{\gat}\mu)_{\bot}\big)
&
\norm{H}_{\L{2}{\mu}(\om)}&\leq c_{\div,\gat}\norm{\div\mu H}_{\L{2}{}(\om)},\\
\forall\,u&\in D\big((\grad_{\gan})_{\bot}\big)
&
\norm{u}_{\L{2}{}(\om)}&\leq c_{\div,\gat}\norm{\grad u}_{\L{2}{\mu}(\om)},\\
\end{align*}
where 
\begin{align*}
D\big((\grad_{\gat})_{\bot}\big)
&=D(\grad_{\gat})\cap N(\grad_{\gat})^{\bot_{\L{2}{}(\om)}}
=D(\grad_{\gat})\cap R(\div_{\gan}\eps),\\
D\big((\div_{\gan}\eps)_{\bot}\big)
&=D(\div_{\gan}\eps)\cap N(\div_{\gan}\eps)^{\bot_{\L{2}{\eps}(\om)}}
=D(\div_{\gan}\eps)\cap R(\grad_{\gat}),\\
D\big((\mu^{-1}\rot_{\gat})_{\bot}\big)
&=D(\mu^{-1}\rot_{\gat})\cap N(\mu^{-1}\rot_{\gat})^{\bot_{\L{2}{\eps}(\om)}}
=D(\mu^{-1}\rot_{\gat})\cap R(\eps^{-1}\rot_{\gan}),
\end{align*}
which also gives $D\big((\eps^{-1}\rot_{\gan})_{\bot}\big)$,
$D\big((\div_{\gat}\mu)_{\bot}\big)$, and
$D\big((\grad_{\gan})_{\bot}\big)$
by interchanging $\eps$, $\mu$ and $\gat$, $\gan$,
\item[\bf(v')]
it holds for all
$E\in D(\mu^{-1}\rot_{\gat})\cap D(\div_{\gan}\eps)\cap\Harm{}{\gat,\gan,\eps}(\om)^{\bot_{\L{2}{\eps}(\om)}}$
$$\norm{E}_{\L{2}{\eps}(\om)}^2
\leq c_{\rot,\gat}^2\norm{\mu^{-1}\rot E}_{\L{2}{\mu}(\om)}^2
+c_{\grad,\gat}^2\norm{\div\eps E}_{\L{2}{}(\om)}^2,$$
\item[\bf(vi)]
$\Harm{}{\gat,\gan,\eps}(\om)=\{0\}$, if $\om$ is additionally extendable.
\end{itemize}
\end{theo}

\begin{rem}
Theorem \ref{theo:cptemb:derhamvec} and Theorem \ref{theo:minifatb:derhamvec}
hold more generally for bounded weak Lipschitz pairs $(\om,\gat)$, 
see \cite{BPS2016a,BPS2018a,BPS2019a}. 
\end{rem}

\subsection{Higher Order Mini FA-ToolBox and Dirichlet/Neumann Fields}

Theorem \ref{theo:cptemb:derhamvec} holds even for higher Sobolev orders,
cf.~Theorem \ref{theo:cptemb:derhamk}.

\begin{theo}[higher order compact embedding for the vector de Rham complex]
\label{theo:cptemb:derhamveck}
Let $(\om,\gat)$ be a bounded strong Lipschitz pair. 
Then for all $k\in\nat_{0}$ the embeddings
\begin{align*}
\H{k+1}{\gat}(\om)\cap\H{k}{\gan}(\om)
&\incl\H{k}{\ga}(\om),\\
\H{k}{\gat}(\rot,\om)\cap\H{k}{\gan}(\div,\om)
&\incl\H{k}{\ga}(\om),\\
\H{k}{\gat}(\div,\om)\cap\H{k}{\gan}(\rot,\om)
&\incl\H{k}{\ga}(\om),\\
\H{k}{\gat}(\om)\cap\H{k+1}{\gan}(\om)
&\incl\H{k}{\ga}(\om)
\end{align*}
are compact.
\end{theo}

\begin{rem}[higher order Friedrichs/Poincar\'e type estimates for the vector de Rham complex]
\label{rem:cptemb:derhamveck}
Analogues of Theorem \ref{theo:highorderFAT1} and 
Theorem \ref{theo:cptemb:derhamk2} hold.
In particular, for all $k\geq0$ there exists $\widetilde{c}_{k}>0$ such that for all
$E\in\H{k}{\gat}(\rot,\om)\cap\H{k}{\gan}(\div,\om)\cap\Harm{}{\gat,\gan,\id}(\om)^{\bot_{\L{2}{}(\om)}}$
$$\norm{E}_{\H{k}{}(\om)}^2
\leq\widetilde{c}_{k}^{2}\big(\norm{\rot E}_{\H{k}{}(\om)}^2
+\norm{\div E}_{\H{k}{}(\om)}^2\big).$$
\end{rem}

Theorem \ref{highorderregpotpbcLipderhamvec} (iv'), cf.~Corollary \ref{cor:highorderkernelderham},
shows by induction for all $k,\ell\geq0$
\begin{align}
\label{highorderkernelderhamvecell}
\H{k}{\gat,0}(\rot,\om)&=\H{\ell}{\gat,0}(\rot,\om)+\grad\H{k+1}{\gat}(\om),
&
\H{k}{\gat,0}(\div,\om)&=\H{\ell}{\gat,0}(\div,\om)+\rot\H{k+1}{\gat}(\om).
\end{align}
By Theorem \ref{theo:minifatb:derhamvec} (iv) we have the orthonormal Helmholtz decompositions
\begin{align}
\begin{aligned}
\label{helmcoho1vec}
\L{2}{\eps}(\om)
&=\grad\H{1}{\gat}(\om)
\oplus_{\L{2}{\eps}(\om)}
\eps^{-1}\H{}{\gan,0}(\div,\om)\\
&=\H{}{\gat,0}(\rot,\om)
\oplus_{\L{2}{\eps}(\om)}
\eps^{-1}\rot\H{}{\gan}(\rot,\om)\\
&=\grad\H{1}{\gat}(\om)
\oplus_{\L{2}{\eps}(\om)}
\Harm{}{\gat,\gan,\eps}(\om)
\oplus_{\L{2}{\eps}(\om)}
\eps^{-1}\rot\H{}{\gan}(\rot,\om),\\
\H{}{\gat,0}(\rot,\om)
&=\grad\H{1}{\gat}(\om)
\oplus_{\L{2}{\eps}(\om)}
\Harm{}{\gat,\gan,\eps}(\om),\\
\eps^{-1}\H{}{\gan,0}(\div,\om)
&=\Harm{}{\gat,\gan,\eps}(\om)
\oplus_{\L{2}{\eps}(\om)}
\eps^{-1}\rot\H{}{\gan}(\rot,\om).
\end{aligned}
\end{align}
Let us denote the $\L{2}{\eps}(\om)$-orthonormal projector onto 
$\eps^{-1}\H{}{\gan,0}(\div,\om)$ and $\H{}{\gat,0}(\rot,\om)$ by
$$\pi_{\div}:\L{2}{\eps}(\om)\to\eps^{-1}\H{}{\gan,0}(\div,\om),\qquad
\pi_{\rot}:\L{2}{\eps}(\om)\to\H{}{\gat,0}(\rot,\om)$$
respectively. Then
\begin{align*}
\pi_{\div}|_{\H{}{\gat,0}(\rot,\om)}:\H{}{\gat,0}(\rot,\om)&\to\Harm{}{\gat,\gan,\eps}(\om),\\
\pi_{\rot}|_{\eps^{-1}\H{}{\gan,0}(\div,\om)}:\eps^{-1}\H{}{\gan,0}(\div,\om)&\to\Harm{}{\gat,\gan,\eps}(\om)
\end{align*}
are onto. Moreover, 
\begin{align*}
\pi_{\div}|_{\grad\H{1}{\gat}(\om)}&=0,
&
\pi_{\rot}|_{\eps^{-1}\rot\H{}{\gan}(\rot,\om)}&=0,\\
\pi_{\div}|_{\Harm{}{\gat,\gan,\eps}(\om)}&=\id_{\Harm{}{\gat,\gan,\eps}(\om)},
&
\pi_{\rot}|_{\Harm{}{\gat,\gan,\eps}(\om)}&=\id_{\Harm{}{\gat,\gan,\eps}(\om)}.
\end{align*}
Therefore, by \eqref{highorderkernelderhamvecell} and for all $\ell\geq0$
\begin{align*}
\Harm{}{\gat,\gan,\eps}(\om)
&=\pi_{\div}\H{}{\gat,0}(\rot,\om)
=\pi_{\div}\H{\ell}{\gat,0}(\rot,\om),\\
\Harm{}{\gat,\gan,\eps}(\om)
&=\pi_{\rot}\eps^{-1}\H{}{\gan,0}(\div,\om)
=\pi_{\rot}\eps^{-1}\H{\ell}{\gan,0}(\div,\om).
\end{align*}
Hence with
$$\H{\infty}{\gat,0}(\rot,\om):=\bigcap_{k\geq0}\H{k}{\gat,0}(\rot,\om),\qquad
\H{\infty}{\gat,0}(\div,\om):=\bigcap_{k\geq0}\H{k}{\gat,0}(\div,\om)$$
we have the following result:

\begin{theo}[smooth pre-bases of Dirichlet/Neumann fields for the vector de Rham complex]
\label{theo:cohomologyinftyvec}
Let $(\om,\gat)$ be a bounded strong Lipschitz pair
and let $d_{\om,\gat}:=\dim\Harm{}{\gat,\gan,\eps}(\om)$. Then 
$$\pi_{\div}\H{\infty}{\gat,0}(\rot,\om)
=\Harm{}{\gat,\gan,\eps}(\om)
=\pi_{\rot}\eps^{-1}\H{\infty}{\gan,0}(\div,\om).$$
Moreover, there exists a smooth $\rot$-\emph{pre-basis}
and a smooth $\div$-\emph{pre-basis} of $\Harm{}{\gat,\gan,\eps}(\om)$, i.e.,
there are linear independent smooth fields 
\begin{align*}
\B{}{\rot,\gat}(\om)
&:=\{B_{\rot,\gat,\ell}\}_{\ell=1}^{d_{\om,\gat}}
\subset\H{\infty}{\gat,0}(\rot,\om),
&
\B{}{\div,\gan}(\om)
&:=\{B_{\div,\gan,\ell}\}_{\ell=1}^{d_{\om,\gat}}
\subset\H{\infty}{\gan,0}(\div,\om)
\end{align*}
such that $\pi_{\div}\B{}{\rot,\gat}(\om)$ and $\pi_{\rot}\eps^{-1}\B{}{\div,\gan}(\om)$
are both bases of $\Harm{}{\gat,\gan,\eps}(\om)$. In particular,
$$\Lin\pi_{\div}\B{}{\rot,\gat}(\om)
=\Harm{}{\gat,\gan,\eps}(\om)
=\Lin\pi_{\rot}\eps^{-1}\B{}{\div,\gan}(\om).$$
\end{theo}

Note that $(1-\pi_{\div})$ and $(1-\pi_{\rot})$ are the $\L{2}{\eps}(\om)$-orthonormal projectors onto 
$\grad\H{1}{\gat}(\om)$ and $\eps^{-1}\rot\H{}{\gan}(\rot,\om)$, respectively, i.e.,
\begin{align*}
(1-\pi_{\div}):\L{2}{\eps}(\om)&\to\grad\H{1}{\gat}(\om),
&
(1-\pi_{\rot}):\L{2}{\eps}(\om)&\to\eps^{-1}\rot\H{}{\gan}(\rot,\om).
\end{align*}
Then by \eqref{helmcoho1vec} and Theorem \ref{highorderregpotpbcLipderhamvec} (ii) we have, e.g.,
\begin{align}
\begin{aligned}
\label{helmcoho2vec}
\H{}{\gat,0}(\rot,\om)
&=\grad\H{1}{\gat}(\om)
\oplus_{\L{2}{\eps}(\om)}
\Harm{}{\gat,\gan,\eps}(\om)\\
&=\grad\H{1}{\gat}(\om)
\oplus_{\L{2}{\eps}(\om)}
\Lin\pi_{\div}\B{}{\rot,\gat}(\om)\\
&=\grad\H{1}{\gat}(\om)
+(\pi_{\div}-1)\Lin\B{}{\rot,\gat}(\om)
+\Lin\B{}{\rot,\gat}(\om)\\
&=\grad\H{1}{\gat}(\om)
+\Lin\B{}{\rot,\gat}(\om),\\
\H{k}{\gat,0}(\rot,\om)
&=\grad\H{1}{\gat}(\om)\cap\H{k}{\gat,0}(\rot,\om)
+\Lin\B{}{\rot,\gat}(\om),\\
&=\grad\H{k+1}{\gat}(\om)
+\Lin\B{}{\rot,\gat}(\om).
\end{aligned}
\end{align}

Similar to Theorem \ref{highorderregdecopbcLipinfty} we get:

\begin{theo}[higher order bounded regular direct decompositions for the vector de Rham complex]
\label{highorderregdecopbcLipinftyvec}
Let $(\om,\gat)$ be a bounded strong Lipschitz pair and let $k\geq0$. 
Then the bounded regular direct decompositions
\begin{align*}
\H{k}{\gat}(\rot,\om)
&=R(\widetilde\PotQ_{\rot,\gat,1}^{k})
\dotplus\H{k}{\gat,0}(\rot,\om),
&
\H{k}{\gat,0}(\rot,\om)
&=\grad\H{k+1}{\gat}(\om)
\dotplus\Lin\B{}{\rot,\gat}(\om),\\
\H{k}{\gan}(\div,\om)
&=R(\widetilde\PotQ_{\div,\gan,1}^{k})
\dotplus\H{k}{\gan,0}(\div,\om),
&
\H{k}{\gan,0}(\div,\om)
&=\rot\H{k+1}{\gan}(\om)
\dotplus\Lin\B{}{\div,\gan}(\om)
\end{align*}
hold. Note that 
$R(\widetilde\PotQ_{\rot,\gat,1}^{k})\subset\H{k+1}{\gat}(\om)$
and $R(\widetilde\PotQ_{\div,\gan,1}^{k})\subset\H{k+1}{\gan}(\om)$.
In particular, for $k=0$
{\small
\begin{align*}
\H{}{\gat}(\rot,\om)
&=R(\widetilde\PotQ_{\rot,\gat,1}^{0})
\dotplus\H{}{\gat,0}(\rot,\om),
&
\H{}{\gat,0}(\rot,\om)
&=\grad\H{1}{\gat}(\om)
\dotplus\Lin\B{}{\rot,\gat}(\om),\\
&&
&=\grad\H{1}{\gat}(\om)
\oplus_{\L{2}{\eps}(\om)}
\Harm{}{\gat,\gan,\eps}(\om),\\
\H{}{\gan}(\div,\om)
&=R(\widetilde\PotQ_{\div,\gan,1}^{0})
\dotplus\H{}{\gan,0}(\div,\om),
&
\eps^{-1}\H{}{\gan,0}(\div,\om)
&=\eps^{-1}\rot\H{1}{\gan}(\om)
\dotplus\eps^{-1}\Lin\B{}{\div,\gan}(\om)\\
&&
&=\eps^{-1}\rot\H{1}{\gan}(\om)
\oplus_{\L{2}{\eps}(\om)}
\Harm{}{\gat,\gan,\eps}(\om)
\end{align*}
}
as well as
\begin{align*}
\L{2}{\eps}(\om)
&=\H{}{\gat,0}(\rot,\om)
\oplus_{\L{2}{\eps}(\om)}
\eps^{-1}\rot\H{1}{\gan}(\om)
=\grad\H{1}{\gat}(\om)
\oplus_{\L{2}{\eps}(\om)}
\eps^{-1}\H{}{\gan,0}(\div,\om).
\end{align*}
\end{theo}

Remark \ref{rem:highorderregdecopbcLipinfty} holds here as well.
Noting
\begin{align}
\label{Borthorotdiv}
\eps^{-1}\rot\H{}{\gan}(\rot,\om)\bot_{\L{2}{\eps}(\om)}\B{}{\rot,\gat}(\om),\qquad
\grad\H{1}{\gat}(\om)\bot_{\L{2}{}(\om)}\B{}{\div,\gan}(\om)
\end{align}
we see:

\begin{theo}[alternative Dirichlet/Neumann projections for the vector de Rham complex]
\label{Bcoho1vec}
Let $(\om,\gat)$ be a bounded strong Lipschitz pair. Then
\begin{align*}
\Harm{}{\gat,\gan,\eps}(\om)\cap\B{}{\rot,\gat}(\om)^{\bot_{\L{2}{\eps}(\om)}}
&=\{0\},
&
\eps^{-1}\H{}{\gan,0}(\div,\om)\cap\B{}{\rot,\gat}(\om)^{\bot_{\L{2}{\eps}(\om)}}
&=\eps^{-1}\rot\H{}{\gan}(\rot,\om),\\
\Harm{}{\gat,\gan,\eps}(\om)\cap\B{}{\div,\gan}(\om)^{\bot_{\L{2}{}(\om)}}
&=\{0\},
&
\H{}{\gat,0}(\rot,\om)\cap\B{}{\div,\gan}(\om)^{\bot_{\L{2}{}(\om)}}
&=\grad\H{1}{\gat}(\om).
\end{align*}
Moreover, for all $k\geq0$
\begin{align*}
\eps^{-1}\H{k}{\gan,0}(\div,\om)\cap\B{}{\rot,\gat}(\om)^{\bot_{\L{2}{\eps}(\om)}}
&=\eps^{-1}\rot\H{k}{\gan}(\rot,\om)
=\eps^{-1}\rot\H{k+1}{\gan}(\om),\\
\H{k}{\gat,0}(\rot,\om)\cap\B{}{\div,\gan}(\om)^{\bot_{\L{2}{}(\om)}}
&=\grad\H{k+1}{\gat}(\om).
\end{align*}
\end{theo}

\begin{theo}[cohomology groups of the vector de Rham complex]
\label{Bcoho2vec}
Let $(\om,\gat)$ be a bounded strong Lipschitz pair. Then
\begin{align*}
N(\rot_{\gat}^{k})/R(\grad_{\gat}^{k})
\cong\Lin\B{}{\rot,\gat}(\om)
\cong\Harm{}{\gat,\gan,\eps}(\om)
\cong\Lin\B{}{\div,\gan}(\om)
\cong N(\div_{\gan}^{k})/R(\rot_{\gan}^{k}).
\end{align*}
In particular, the dimensions of the cohomology groups
(Dirichlet/Neumann fields) are independent of $k$ and $\eps$ and it holds
\begin{align*}
d_{\om,\gat}^{}
=\dim\big(N(\rot_{\gat}^{k})/R(\grad_{\gat}^{k})\big)
=\dim\big(N(\div_{\gan}^{k})/R(\rot_{\gan}^{k})\big).
\end{align*}
\end{theo}


\bibliographystyle{plain} 
\bibliography{/Users/paule/GoogleDriveData/Tex/input/bibTex/psz}


\appendix

\section{Results for the Co-Derivative}

By Hodge $\star$-duality we get the corresponding dual results 
from Section \ref{sec:derham} for the $\cd$-operator.

\begin{lem}[regular potential for $\cd$ without boundary condition]
\label{highorderregpotcd}
Let $\om\subset\rdimom$ be a bounded strong Lipschitz domain and let $k\geq0$
and $q\in\{0,\dots,\dimom-1\}$.
Then there exists a bounded linear regular potential operator
$$\PotP_{\cd,\emptyset}^{q,k}:
\H{q,k}{\emptyset,0}(\cd,\om)\cap\Harm{q}{\ga,\emptyset,\id}(\om)^{\bot_{\L{q,2}{}(\om)}}
\To\H{q+1,k+1}{0}(\ed,\rdimom),$$ 
such that $\cd\PotP_{\cd,\emptyset}^{q,k}=\id|_{\H{q,k}{\emptyset,0}(\cd,\om)\cap\Harm{q}{\ga,\emptyset,\id}(\om)^{\bot_{\L{q,2}{}(\om)}}}$, i.e.,
for all $E\in\H{q,k}{\emptyset,0}(\cd,\om)\cap\Harm{q}{\ga,\emptyset,\id}(\om)^{\bot_{\L{q,2}{}(\om)}}$ 
$$\cd\PotP_{\cd,\emptyset}^{q,k}E=E\quad\text{in }\om.$$
In particular, the bounded regular potential representations 
$$R(\cd_{\emptyset}^{q+1,k})
=\H{q,k}{\emptyset,0}(\cd,\om)\cap\Harm{q}{\ga,\emptyset,\id}(\om)^{\bot_{\L{q,2}{}(\om)}}
=\cd\H{q+1,k}{\emptyset}(\cd,\om)
=\cd\H{q+1,k+1}{\emptyset}(\om)
=\cd\H{q+1,k+1}{\emptyset,0}(\ed,\om)$$ 
hold and the potentials can be chosen such that they depend continuously on the data.
Especially, these are closed subspaces 
of $\H{q,k}{\emptyset}(\om)=\H{q,k}{}(\om)$ 
and $\PotP_{\cd,\emptyset}^{q,k}$ is a right inverse to $\cd$.
By a simple cut-off technique $\PotP_{\cd,\emptyset}^{q,k}$ may be modified to 
$$\PotP_{\cd,\emptyset}^{q,k}:
\H{q,k}{\emptyset,0}(\cd,\om)\cap\Harm{q}{\ga,\emptyset,\id}(\om)^{\bot_{\L{q,2}{}(\om)}}
\To\H{q+1,k+1}{}(\ed,\rdimom)$$ 
such that $\PotP_{\cd,\emptyset}^{q,k}E$ has a fixed compact support in $\rdimom$ 
for all $E\in\H{q,k}{\emptyset,0}(\cd,\om)\cap\Harm{q}{\ga,\emptyset,\id}(\om)^{\bot_{\L{q,2}{}(\om)}}$.
\end{lem}

\begin{lem}[regular potentials and decompostions for $\cd$ with partial boundary condition for extendable domains]
\label{highorderregpotdecocdpbc}
Let $(\om,\gan)$ be an extendable bounded strong Lipschitz pair and let $k\geq0$.
\begin{itemize}
\item[\bf(i)]
For $1\leq q\leq\dimom-1$ there exists a bounded linear regular potential operator
$$\PotP_{\cd,\gan}^{q,k}:\bH{q,k}{\gan,0}(\cd,\om)
\To\H{q+1,k+1}{}(\rdimom)\cap\H{q+1,k+1}{\gan}(\om),$$
such that $\cd\PotP_{\ed,\gan}^{q,k}=\id|_{\bH{q,k}{\gan,0}(\cd,\om)}$, i.e.,
for all $E\in\bH{q,k}{\gan,0}(\cd,\om)$
$$\cd\PotP_{\cd,\gan}^{q,k}E=E\quad\text{in }\om.$$
In particular, the bounded regular potential representations 
$$\bH{q,k}{\gan,0}(\cd,\om)
=\H{q,k}{\gan,0}(\cd,\om)
=\cd\H{q+1,k+1}{\gan}(\om)
=\cd\H{q+1,k}{\gan}(\cd,\om)$$
hold and the potentials can be chosen such that they depend continuously on the data.
Especially, these are closed subspaces 
of $\H{q,k}{\emptyset}(\om)=\H{q,k}{}(\om)$ 
and $\PotP_{\cd,\gan}^{q,k}$ is a right inverse to $\cd$.
The results extend literally to the case $q=0$ if $\gan\neq\ga$
and the case $q=\dimom$ is trivial since $\bH{\dimom,k}{\gan,0}(\cd,\om)=\reals_{\gan}$. 
For $q=0$ and $\gan=\ga$ the results still remain valid if
$\bH{0,k}{\ga,0}(\cd,\om)=\bH{0,k}{\ga}(\om)$
and $\H{0,k}{\ga,0}(\cd,\om)=\H{0,k}{\ga}(\om)$
are replaced by the slightly smaller spaces
$\bH{0,k}{\ga}(\om)\cap\reals^{\bot_{\L{0,2}{}(\om)}}$ 
and $\H{0,k}{\ga}(\om)\cap\reals^{\bot_{\L{0,2}{}(\om)}}$, respectively.
\item[\bf(ii)]
For all $0\leq q\leq\dimom$ the regular decompositions
\begin{align*}
\bH{q,k}{\gan}(\cd,\om)
=\H{q,k}{\gan}(\cd,\om)
&=\H{q,k+1}{\gan}(\om)+\cd\H{q+1,k+1}{\gan}(\om)\\
&=\PotQ_{\cd,\gan,1}^{q,k}\H{q,k}{\gan}(\cd,\om)
\dotplus\cd\PotQ_{\cd,\gan,0}^{q,k}\H{q,k}{\gan}(\cd,\om)\\
&=\PotQ_{\cd,\gan,1}^{q,k}\H{q,k}{\gan}(\cd,\om)
\dotplus\cd\H{q+1,k+1}{\gan}(\om)\\
&=\PotQ_{\cd,\gan,1}^{q,k}\H{q,k}{\gan}(\cd,\om)
\dotplus\H{q,k}{\gan,0}(\cd,\om)
\end{align*}
hold with bounded linear regular decomposition operators
\begin{align*}
\PotQ_{\cd,\gan,1}^{q,k}:=\PotP_{\cd,\gan}^{q-1,k}\cd:\H{q,k}{\gan}(\cd,\om)
&\to\H{q,k+1}{\gan}(\om),\\
\PotQ_{\cd,\gan,0}^{q,k}:=\PotP_{\cd,\gan}^{q,k}(1-\PotP_{\cd,\gan}^{q-1,k}\cd):\H{q,k}{\gan}(\cd,\om)
&\to\H{q+1,k+1}{\gan}(\om)
\end{align*}
satisfying
$\PotQ_{\cd,\gan,1}^{q,k}+\cd\PotQ_{\cd,\gan,0}^{q,k}
=\id|_{\H{q,k}{\gan}(\cd,\om)}$.
Moreover, it holds $\cd\PotQ_{\cd,\gan,1}^{q,k}=\cd^{q,k}_{\gan}$
and thus $\H{q,k}{\gan,0}(\cd,\om)$ is invariant under $\PotQ_{\cd,\gan,1}^{q,k}$.
$\PotQ_{\cd,\gan,1}^{q,k}\H{q,k}{\gan}(\cd,\om)
=R(\PotQ_{\cd,\gan,1}^{q,k})
=R(\PotP_{\cd,\gan}^{q-1,k})$
as well as
$\PotQ_{\cd,\gan,0}^{q,k}\H{q,k}{\gan}(\cd,\om)
=R(\PotQ_{\cd,\gan,0}^{q,k})
=R(\PotP_{\cd,\gan}^{q,k})$ hold.
\end{itemize}
\end{lem}

\begin{lem}[regular decompositions for $\cd$ with partial boundary condition]
\label{app:lem:highorderregdecoedpbcLip}
Let $(\om,\gan)$ be a bounded strong Lipschitz pair and let $k\geq0$. 
Then the bounded regular decompositions
$$\bH{q,k}{\gan}(\cd,\om)
=\H{q,k}{\gan}(\cd,\om)
=\H{q,k+1}{\gan}(\om)
+\cd\H{q+1,k+1}{\gan}(\om)$$
hold with bounded linear regular decomposition operators
\begin{align*}
\PotQ_{\cd,\gan,1}^{q,k}:\H{q,k}{\gan}(\cd,\om)\to\H{q,k+1}{\gan}(\om),\qquad
\PotQ_{\cd,\gan,0}^{q,k}:\H{q,k}{\gan}(\cd,\om)\to\H{q+1,k+1}{\gan}(\om)
\end{align*}
satisfying $\PotQ_{\cd,\gan,1}^{q,k}+\cd\PotQ_{\cd,\gan,0}^{q,k}=\id_{\H{q,k}{\gan}(\cd,\om)}$.
In particular, weak and strong boundary conditions coincide.
Moreover, it holds $\cd\PotQ_{\cd,\gan,1}^{q,k}=\cd^{q,k}_{\gan}$
and thus $\H{q,k}{\gan,0}(\cd,\om)$ is invariant under $\PotQ_{\cd,\gan,1}^{q,k}$.
\end{lem}

\begin{theo}[higher order bounded regular potentials and decompositions for $\cd$ with partial boundary condition]
\label{highorderregpotcdpbcLip}
Let $(\om,\gan)$ be a bounded strong Lipschitz pair and let $k\geq0$.
Moreover, let $\PotQ_{\cd,\gan,1}^{q,k}$ be given from Lemma \ref{app:lem:highorderregdecoedpbcLip}. Then:
\begin{itemize}
\item[\bf(i)]
For all $q\in\{0,\dots,\dimom-1\}$ there exists a bounded linear regular potential operator
$$\PotP_{\cd,\gan}^{q,k}:=\PotQ_{\cd,\gan,1}^{q+1,k}(\cd_{\gan}^{q+1,k})_{\bot}^{-1}:
\H{q,k}{\gan,0}(\cd,\om)
\cap\Harm{q}{\gat,\gan,\eps}(\om)^{\bot_{\L{q,2}{}(\om)}}
\To\H{q+1,k+1}{\gan}(\om),$$
such that 
$\cd\PotP_{\cd,\gan}^{q,k}=\id|_{\H{q,k}{\gan,0}(\cd,\om)
\cap\Harm{q}{\gat,\gan,\eps}(\om)^{\bot_{\L{q,2}{}(\om)}}}$.
In particular, the bounded regular representations 
\begin{align*}
R(\cd_{\gan}^{q+1,k})
&=\H{q,k}{\gan,0}(\cd,\om)
\cap\Harm{q}{\gat,\gan,\eps}(\om)^{\bot_{\L{q,2}{}(\om)}}\\
&=\H{q,k}{\gan}(\om)
\cap\cd\H{q+1}{\gan}(\cd,\om)
=\cd\H{q+1,k}{\gan}(\cd,\om)
=\cd\H{q+1,k+1}{\gan}(\om)
\end{align*}
hold and the potentials can be chosen such that they depend continuously on the data.
\item[\bf(ii)]
The bounded regular decompositions
\begin{align*}
\H{q,k}{\gan}(\cd,\om)
&=\H{q,k+1}{\gan}(\om)
+\H{q,k}{\gan,0}(\cd,\om)
=\H{q,k+1}{\gan}(\om)
+\cd\H{q+1,k+1}{\gan}(\om)\\
&=R(\widetilde\PotQ_{\cd,\gan,1}^{q,k})
\dotplus\H{q,k}{\gan,0}(\cd,\om)
=R(\widetilde\PotQ_{\cd,\gan,1}^{q,k})
\dotplus R(\widetilde\PotN_{\cd,\gan}^{q,k})
\end{align*}
hold with bounded linear regular decomposition operators
\begin{align*}
\widetilde\PotQ_{\cd,\gan,1}^{q,k}:=\PotP_{\cd,\gan}^{q-1,k}\cd^{q,k}_{\gan}:\H{q,k}{\gan}(\cd,\om)\to\H{q,k+1}{\gan}(\om),\qquad
\widetilde\PotN_{\cd,\gan}^{q,k}:\H{q,k}{\gan}(\cd,\om)\to\H{q,k}{\gan,0}(\cd,\om)
\end{align*}
satisfying $\widetilde\PotQ_{\cd,\gan,1}^{q,k}+\widetilde\PotN_{\cd,\gan}^{q,k}=\id_{\H{q,k}{\gan}(\cd,\om)}$.
Moreover, $\cd\widetilde\PotQ_{\cd,\gan,1}^{q,k}=\cd\PotQ_{\cd,\gan,1}^{q,k}=\cd^{q,k}_{\gan}$
and thus $\H{q,k}{\gan,0}(\cd,\om)$ is invariant under 
$\PotQ_{\cd,\gan,1}^{q,k}$ and $\widetilde\PotQ_{\cd,\gan,1}^{q,k}$.
It holds $R(\widetilde\PotQ_{\cd,\gan,1}^{q,k})=R(\PotP_{\cd,\gan}^{q-1,k})$ and 
$\widetilde\PotQ_{\cd,\gan,1}^{q,k}
=\PotP_{\cd,\gan}^{q-1,k}\cd^{q,k}_{\gan}
=\PotQ_{\cd,\gan,1}^{q,k}(\cd_{\gan}^{q,k})_{\bot}^{-1}\cd^{q,k}_{\gan}$.
Hence 
$\widetilde\PotQ_{\cd,\gan,1}^{q,k}|_{(\cd_{\gan}^{q,k})_{\bot}}
=\PotQ_{\cd,\gan,1}^{q,k}|_{(\cd_{\gan}^{q,k})_{\bot}}$
and thus $\widetilde\PotQ_{\cd,\gan,1}^{q,k}$
may differ from $\PotQ_{\cd,\gan,1}^{q,k}$ 
only on $\H{q,k}{\gan,0}(\cd,\om)$.
\item[\bf(ii')]
The bounded regular kernel decomposition 
$\H{q,k}{\gan,0}(\cd,\om)=\H{q,k+1}{\gan,0}(\cd,\om)+\cd\H{q+1,k+1}{\gan}(\om)$
holds.
\end{itemize}
\end{theo}

Note that
Remark \ref{highorderregpotedpbcLipremdirneu} and
Remark \ref{highorderregpotedpbcLiprem} hold with obvious modifications.


\vspace*{5mm}
\hrule
\vspace*{3mm}


\end{document}